\documentclass[12pt,oneside,a4]{amsart}
\usepackage{typearea}
\usepackage{amsfonts}
\usepackage{amsmath}
\usepackage{amssymb}
\usepackage{amsthm}
\usepackage{hyperref}
\usepackage{mathrsfs}
\usepackage{comment}
\usepackage{color}
\usepackage{xypic}
\usepackage{catchfile}
\usepackage{xstring}
\usepackage{gitinfo2}
\usepackage{lmodern}

\newcommand{\IC}{\mathbb{C}}

\newcommand{\IN}{\mathbb{N}}

\newcommand{\IQ}{\mathbb{Q}}
\newcommand{\IQbar}{\overline{\mathbb{Q}}}
\newcommand{\IR}{\mathbb{R}}

\newcommand{\IZ}{\mathbb{Z}}
\newcommand{\IF}{\mathbb{F}}

\newcommand{\mat}[1]{\mathrm{Mat}_{#1}}

\newcommand{\trans}[1]{{#1}^{t}}

\renewcommand{\cD}{\mathcal{D}}
\newcommand{\cV}{\mathcal{V}}
\newcommand{\cT}{\mathcal{T}}

\newcommand{\cO}{\mathcal{O}}

\newcounter{maincounter}

\numberwithin{maincounter}{section}
\numberwithin{equation}{section}

\newtheorem{thm}[maincounter]{Theorem}
\newtheorem{theorem}[maincounter]{Theorem}
\newtheorem{lemma}[maincounter]{Lemma}
\newtheorem{remark}[maincounter]{Remark}
\newtheorem{corollary}[maincounter]{Corollary}

\newtheorem{defi}[maincounter]{Definition}
\newtheorem{propo}[maincounter]{Proposition}
\newtheorem{proposition}[maincounter]{Proposition}
\newtheorem{example}[maincounter]{Example}

\renewcommand{\subset}{\subseteq}
\renewcommand{\supset}{\supseteq}

\newcommand{\ssm}{\smallsetminus}

\renewcommand{\mat}{\mathrm{Mat}}

\makeatletter
\def\imod#1{\allowbreak\mkern10mu({\operator@font mod}\,\,#1)}
\makeatother

\newcommand{\bfz}{\boldsymbol{z}}

\newcommand{\crit}[1]{\mathrm{Crit}({#1})}
\newcommand{\pco}{\mathrm{PCO}}
\newcommand{\ep}{\mathrm{ep}}

\newcommand{\branch}{}

\IfFileExists{.git/HEAD}{%
  
  \StrGobbleRight{\input{.git/HEAD}}{1}[\head]
  \StrBehind[2]{\head}{/}[\branch]
  \IfFileExists{.git/refs/heads/\branch}{%
    }{%
    }
    
}  

\begin{document}

\title{Dynamical Canonical Heights  and Finite Trees}

\author{Philipp Habegger and Harry Schmidt}
\date{\today}

\maketitle

\begin{abstract}
  We establish lower bounds for the canonical height of a wandering
  point that decays like the square of the field degree. Our methods and
  results apply to centered, postcritically finite, hyperbolic
  polynomials of prime power degree whose coefficients are algebraic
  integers. Our approach is ultimately inspired by an
  idea of Dimitrov that led to his proof of the Schinzel--Zassenhaus
  Conjecture.
  The height lower bound is derived from the lower bound of the local
  canonical height at an archimedean place. In previous work, the
  authors used a result of Dubinin on the transfinite diameter of a
  star-shaped tree.
  In this paper,  we develop tools to bound the
  transfinite diameter of more general finite trees in the complex
  plane.
  We construct these trees using the Hubbard tree of a postcritically finite
  polynomial. 
  Thurston's notion of core entropy helps us analyze
  combinatorial properties of the Hubbard tree.
\end{abstract}

\tableofcontents




%
\section{Introduction}

The Dynamical Lehmer Conjecture predicts the asymptotically
optimal lower bound
for the canonical height of a wandering algebraic point in terms of
its degree, see for example
Conjecture~15.2~\cite{BIJMST:currenttrends}. It is linked
to a classical and open problem of Lehmer~\cite{Lehmer:33} regarding the Mahler
measure of integral polynomials.

We provide new evidence towards the Dynamical Lehmer Conjecture for
univariate polynomials. This paper and the authors's earlier
work~\cite{hs:2021lower} were inspired by Dimitrov's
resolution of the Schinzel--Zassenhaus Conjecture~\cite{dimitrov:SZ}.

Before we come to our results we introduce some basic terminology. 
Let $F$ be a field and let $\overline F$ be an algebraic closure of
$F$. Let $f\in F[T]$ be of degree at least $2$. For $f\in F[T]$ and
$n\in\IN=\{1,2,\ldots\}$
we let $f^{(n)} = f\circ\cdots\circ f \in F[T]$ denote the
$n$-th iterate of $f$. It is convenient to set $f^{(0)}=T$. The
\textit{$f$-forward orbit} of $\alpha\in\overline F$ is
$\{f^{(n)}(\alpha) : n\in \IN_0\}$ with
$\IN_0= \{0\}\cup\IN$. We say  $\alpha$ is
 \emph{$f$-wandering} if its $f$-forward orbit is infinite
and \emph{$f$-preperiodic} otherwise. Finally, $\alpha$ is
\emph{$f$-periodic} if $f^{(n)}(\alpha)=\alpha$ for some $n\ge 1$. 
For convenience we sometimes
drop the prefix $f$- in $f$-wandering, $f$-preperiodic, etc.

If $F=\IQ$,
or more generally, if $F$ is a number field,
then the \emph{canonical} or \emph{Call--Silverman height} $\widehat h_f\colon
\overline F\rightarrow [0,\infty)$ is a key tool in arithmetic
dynamics. We recall its definition below. Here we just state the
dynamical version of Kronecker's Theorem. The canonical height
$\widehat h_f$ vanishes precisely on the $f$-preperiodic points in
$\overline F$.

Our first result will be a consequence of a
more general result stated in Theorem~\ref{thm:mainintro} below.

\begin{theorem}
  \label{thm:unicriticalwandering}
  Let $p$ be a prime number and $e\in\IN$.  
  We suppose that $b$ is an algebraic number such that  
  $0$ is periodic under $f = T^{p^e} + b$. Then there
  exists $c=c(p,e,b)>0$ such that all  
  $f$-wandering algebraic numbers $\alpha$ satisfy
  \begin{equation*}
    \hat h_f(\alpha) \ge \frac{c}{[\IQ(\alpha):\IQ]^2}. 
  \end{equation*}
\end{theorem}

This theorem applies to $f=T^2-1$ and other
polynomials of the shape $T^p+b$, that were 
also covered in the authors' earlier work~\cite{hs:2021lower}. In this
earlier paper, our polynomials were required to satisfy
the additional \textit{Quill Hypothesis}. This technical hypothesis
can be dropped thanks to the methods developed here.
Below we state a theorem that holds for $f$ of a more general shape.

From now on, $F$ will be a number field. Let $\cO_F$ denote the ring
of algebraic integers in $F$ and $M_F$ its set of places, see
Section~\ref{sec:notation}.

For a real number $t$ we write $\log^+ t =
\log\max\{1,t\}$. Let $L/F$ be a finite field extension and let
$v\in M_L$. The \textit{local canonical height} $\lambda_{f,v} \colon L_v
\rightarrow[0,\infty)$ with respect to $f$ and $v$ is defined as 
\begin{equation}
  \label{def:lambdafv}
  \lambda_{f,v}(\alpha) = \lim_{n\rightarrow\infty} \frac{\log^+ |f^{(n)}(\alpha)|_v}{d^n}
\end{equation}
for $\alpha \in L$; see Remark 3.30~\cite{Silverman:gtm241}. 
The \textit{canonical height} with respect to $f$
 is the function $\hat h_f \colon L\rightarrow[0,\infty)$
given by
\begin{equation}
  \label{def:canonicalhgt}
  \hat h_{f}(\alpha) = \sum_{v\in M_L} \frac{d_v}{[L:\IQ]}
  \lambda_{f,v}(\alpha)
\end{equation}
for all $\alpha\in L$; the local degrees $d_v$ are discussed in
Section~\ref{sec:notation}.
The value $\hat h_{f}(\alpha)$ is independent of the choice of number
field
containing $\alpha$. 
So we consider $\hat h_f$ as a function $\overline
F\rightarrow\IR$. 

In Theorem~\ref{thm:unicriticalwandering} we ask that $0$ is an
$f$-periodic point.
For $f=T^{p^e}+b$ the element $0$ is a
 \textit{critical point}, \textit{i.e.}, a root of the derivative of
 $f$.
We  call a general  $f$ \textit{postcritically finite} if
all critical points in  $\overline F$ are preperiodic.
Thus the polynomials considered in our first theorem are postcritically finite. 

Our more general results allow for a wider class of polynomials. 
The \textit{barycenter} of a non-constant and
monic polynomial $T^d+f_1T^{d-1}+\cdots+f_0$
with coefficients in a field of characteristic $0$ is
\begin{equation}
  \label{def:barycenter}
  -\frac{f_1}{d}.
\end{equation}
A polynomial is called \textit{centered} if its barycenter vanishes.
We recall the notion of a {hyperbolic} complex polynomial in
Definition~\ref{def:hyperbolic} of the Appendix. 
The following theorem is a special case of Corollary~\ref{cor:canhgtlb}.
\begin{theorem}
  \label{thm:mainintro}
  Let $F$ be a number field, let $p$ be a prime that is unramified in
  $F$, and let
  $f\in \cO_F[T]$ be monic of degree
  $p^e$ with $e\in\IN$. Suppose $f$ is postcritically finite and
  that $\sigma(f)$ is hyperbolic for some $\sigma\colon
  F\rightarrow\IC$. Suppose furthermore that the barycenter of $f$
  is an algebraic integer. 
  Then there exists $c(F,f)>0$  with the following property. 
  If $\alpha\in \overline F$  is $f$-wandering, then
  \begin{equation*}
    \hat h_f(\alpha) \ge \frac{c(F,f)}{[F(\alpha):F]^{2}}.  
  \end{equation*}
\end{theorem}

Baker~\cite{Baker:06} proved a height lower bound for elements of a
number field that is inverse linear in the degree after omitting
finitely many exceptions. His bound holds for all {rational} maps of
degree at least $2$.

By an argument of DeMarco, being postcritically finite and hyperbolic
is a Galois-invariant property. With her kind permission, we produce
this fact in  Lemma~\ref{lem:pcfhyperbolicgalois}. 

Our height lower bound is ultimately a consequence of a lower
bound for the  local canonical height at the archimedean places.
The following result
is a special case of Theorem~\ref{thm:main} and establishes a
Schinzel--Zassenhaus type lower bound in the dynamical setting.
We let $M_L^\infty$
denote the set of archimedean places of a number field $L$. 

\begin{theorem}
  \label{thm:main2intro}
  Let $F,p,$ and $f$ be as in Theorem~\ref{thm:mainintro}.
  In particular, $p$ is unramified in $F$.
  Then there exists $c(F,f)>0$  with the following property.
  If
  $\alpha\in \overline F$ is an $f$-wandering algebraic integer
  contained in a finite extension $L$ of $F$, then
  \begin{equation}
    \label{eq:thmmain2introbound}
    \max_{v\in M_L^\infty}
    \lambda_{f,v}(\alpha) \ge \frac{c(F,f)}{[L:F]}. 
  \end{equation}
\end{theorem}

In some cases, this theorem generalizes Theorem~1.1~\cite{hs:2021lower} of the
authors.

If $\alpha\in \overline F$  is $f$-periodic, then the minimal 
 $n\in\IN$ with $f^{(n)}(\alpha)=\alpha$ is called the
\textit{$f$-period length} of $\alpha$. We denote it by  $\mathrm{per}(\alpha)$.
If $\alpha$ is $f$-preperiodic, there exists
$m\in\IN_0$ such that $f^{(m)}(\alpha)$ is $f$-periodic. The minimal
such $m$ is the \textit{$f$-preperiod length} of $\alpha$, in symbols:
$\mathrm{preper}(\alpha)$. 

While the canonical height of preperiodic point is $0$, our methods
can also recover information from this case.   We refer to Theorem~\ref{thm:mainpreper} and
Section~\ref{sec:proofs} for a stronger variant as well as the concept
of linearly conjugate.
We do not need to assume that $f$ is hyperbolic in the theorem below.
Let $\lceil x\rceil$ denote the smallest integer greater or equal to
$x\in\IR$. Properties of Chebyshev polynomials are recalled around
(\ref{eq:chebynorm}).

\begin{theorem}
  \label{thm:intropreper}
  Let $F$ be a number field
  and let $f\in \cO_F[T]$ be monic of degree $d\ge 2$ which is a
  power of $p$.
  Suppose $f$ is postcritically finite. Suppose furthermore that the
  barycenter of $f$ is an algebraic integer and that $f$ is not
  $\overline F$-linearly conjugate to a Chebyshev polynomial or a
  negative Chebyshev polynomial.
  We set
  \begin{equation*}
    \mathfrak{e} = \max\{ e_v : v\in M_F, v\mid p\}\quad\text{and}\quad
    \mathfrak{f} =\mathrm{lcm}\{ f_v : v\in M_F, v\mid p\}
  \end{equation*}
   where $e_v$ denotes
  the ramification index and 
  $f_v$ denotes the residue degree of the place $v$, respectively.
  Let $\alpha\in \overline F$ be $f$-preperiodic, then 
  \begin{equation*}
    \mathrm{preper}(\alpha)\le
    \left(1+\left\lceil \frac{\log\mathfrak{e}}{\log d}\right\rceil\right)\left(1+\frac{\log[F(\alpha):F]}{\log
      d}\right)
    \quad\text{and}\quad
    \mathrm{per}(\alpha)\le p\mathfrak f[F(\alpha):F]. 
  \end{equation*}
\end{theorem}

The basic strategy employed in this paper, and
also~\cite{hs:2021lower}, is based on Dimitrov's proof of
the Schinzel--Zassenhaus Conjecture~\cite{dimitrov:SZ}. Dimitrov's
Theorem states that if a non-zero algebraic integer $\alpha$ is not a
root of unity, then $\log |\sigma(\alpha)|\ge \frac{\log 2}{4[\IQ(\alpha):\IQ]}$
for some embedding $\sigma\colon \IQ(\alpha)\rightarrow\IC$. 
A key step in Dimitrov's work  involves a finite topological tree in
the complex plane, we refer to Section~\ref{sec:trees} for our
terminology. 
This star-shaped tree, called a \textit{hedgehog}, is the
 union of line segments beginning at
$0$.
The hedgehog is required to contain the zeros and
poles of a certain rational function $\frac AB$, with $A,B$ rational
polynomials of equal degree that depend on $\alpha$.
The complement of the tree inside the
Riemann sphere is a simply connected domain. The algebraic function
$(\frac AB)^{1/2}$ is well-defined in a neighborhood of the point at
infinity. By the Monodromy Theorem  from complex analysis, this algebraic
function extends to a holomorphic map on the complement of the
hedgehog. Using congruence relations modulo $4$, attributed to
Smyth~\cite{smyth:coloringproof} and
Arnold~\cite{arnold:eulerfermat}, Dimitrov  shows that 
the power series representation of $(\frac AB)^{1/2}$ has integral
coefficients. The P\'olya--Bertrandias
Theorem  implies that the algebraic function is
a rational function, if the \textit{transfinite diameter} of the hedgehog is
strictly less than $1$.
The latter condition holds if the conjugates of $\alpha$ are
sufficiently close to the unit circle. Moreover, rationality of
$(\frac AB)^{1/2}$ imposes a strong restriction on $A$ and $B$ and
thus on $\alpha$.  Dimitrov used a deep result of   
Dubinin~\cite{Dubinin} to bound the transfinite diameter
of the hedgehog from above.

In \cite{hs:2021lower} we follow this basic strategy in the
dynamically setting. In particular, we used Dubinin's Theorem and the
dynamics of $f$ to construct a suitable tree. The geometric features
of a hedgehog ties it to the complex unit circle, which is the
\textit{Julia set} of the polynomial $T^2$ but most
likely not the Julia set of $f$. So it may
come as a surprise that it has application to the dynamics of
polynomials encountered in \cite{hs:2021lower} with different Julia
sets. However, the results in \cite{hs:2021lower} are restricted to
polynomials that satisfy the numerical Quill Hypothesis. This
condition facilitates the use of Dubinin's Theorem and is conjecturally
not necessary.

The current paper introduces several innovations. One novelty is that
we replace hedgehogs by a construction that is informed from complex
dynamics. We also avoid using Dubinin's Theorem. Rather we construct a
suitable finite topological tree $\cT$ using tools in complex
dynamics. Our starting point is the \textit{Hubbard tree} $H_f$ of a
postcritically finite polynomial $f$, introduced by Douady and
Hubbard~\cite{DH:etudedynI}. This finite topological tree contains all
critical points of $f$. The Hubbard tree is mapped to itself under $f$
and is thus contained in the \textit{filled Julia} $K_f$ set of $f$.
We do not directly work with the Hubbard tree. Rather, our
construction involves its preimages
 $$H_f^{(s)}= (f^{(s)})^{-1}(H_f)$$ with
varying parameter $s\in\IN_0$. Each preimage is again a finite topological
tree. Furthermore, as $f(H_f)\subset H_f$, the sequence is nested
\begin{equation}
  \label{eq:Hfsequence}
H_f =   H_f^{(0)} \subset  H_f^{(1)} \subset H_f^{(2)}\subset
\cdots\subset  H_f^{(s)}\subset \cdots.
\end{equation}
Intuitively,  $H_f^{(s)}$  is a combinatorial avatar of $K_f$.
The final tree $\cT$ is a suitable subtree of $H_f^{(s)}$  extended
by new leaves given by the Galois conjugates of a point of small height.
The tree $\cT$ is our dynamical substitute of the hedgehog.

In order to bound the transfinite diameter of $\cT$ from above, we
expand an approach explored by the first-named
author~\cite{hab:dubinin}. This earlier paper contains an elementary
approach to Dubinin's Theorem based on the Vandermonde matrix. The
result is an upper bound for the transfinite diameter of a hedgehog,
albeit with a suboptimal numerical constant.

In the current paper, we develop the ideas in \cite{hab:dubinin} by
introducing some new tools. These include divided differences and
estimates that are common in transcendence theory.
The goal is to bound from above the transfinite diameter of $\cT$ using
Vandermonde determinants.

Our approach requires a good
combinatorial understanding of the Hubbard tree and its preimages
$H_f^{(s)}$.
Let us return to the sequence (\ref{eq:Hfsequence}) that is core to
our approach.
It is not difficult to see that the number of edges of $H_f^{(s)}$ can
be at least $d^s$ for all $s\ge 0$; here $d=\deg f$. For example, when $f=T^2-2$, the
Hubbard tree has vertices $-2,0,2$ and two edges $[-2,0],[0,2]$. The
preimages  $H_f^{(s)}$ consists of $2^{s+1}$ edges aligned on
$[-2,2]$. 

However, this growth rate is too fast for our application. And we will
need to work with a suitable subtree $\cT\subset H_f^{(s)}$ with fewer
edges asymptotically in $s$.

Let $v_s$ denote the number of vertices of $H_f^{(s)}$. Suppose we are
presented with a set of vertices $\cV$ of $H_f^{(s)}$ such that
$\mathsf{p}=\#\cV/v_s$. In this model, we think
of an arbitrary vertex lying in $\cV$ with  probability
$\mathsf{p}$.  Now consider the convex hull $\langle
\cV\rangle$ of $\cV$, the smallest subtree of $H_f^{(s)}$ containing
all vertices in $\cV$. Here is the basic question: is
$\mathsf{q}=\#\langle \cV\rangle / v_s$, the proportion of vertices in the
convex hull in relation to the total number of vertices, also small? In other
words, does the convex hull of a sparse set of vertices have
a sparse set of vertices?

For example, if $f=T^2-2$ as above, then one needs to traverse all
$2^{s+1}$ edges in $H_f^{(s)}$ to pass from $-2$ to $2$.
So if our given set of vertices is $\{-2,2\}$, then the convex hull is
the whole tree. 
In the notation of the previous paragraph
$\mathsf{p}=2/(2^{s+1}+1)$ and $\mathsf{q}=1$. Unfortunately, the answer to the question above is
no, at least in general.

However, the Chebyshev polynomial $T^2-2$ is an outlier. For many
postcritically finite polynomials, including $T^2-1$ we will be able
to do better. This requires a better understanding of the
combinatorial properties of the Hubbard tree $H_f$ and the action of
$f$ on $H_f$ induced by the inclusion $f(H_f) \subset H_f$.

We will rely on Thurston's exponential \textit{core entropy}
$\rho(f)\ge 1$ introduced in \cite{thurstonetal:degreed}.
It is the spectral radius associated to the action of $f$ on the edges
of $H_f$.  
A fundamental inequality states  that $\rho(f)\le \deg f$. The upper bound is attained for
example for $T^2-2$. On the other hand, we will see that
$\rho(T^2-1) = (\sqrt 5+1)/2<2$.
More generally,  $\rho(f)<\deg f$ if $f$ is postcritically finite and not
conjugate to a Chebyshev polynomial up-to sign.

If $f$ has non-maximal
core entropy, \textit{i.e.}, if $\rho(f)<\deg f$, then the answer to
the question above on the convex hull of a sparse set of vertices is yes. In other words, a sparse set of vertices
has a sparse convex hull in $H_f^{(s)}$.

The diameter of the sequence of trees $H_f^{(s)}$ grows asymptotically
at most like
$\rho(f)^s$, up-to a polynomial factor in $s$. However, the number of
edges may be $\ge d^s$. Roughly speaking, having non-maximal core
entropy implies that the family of trees $H_f^{(s)}$ exhibits good
expanding properties, an essential property in our construction of the
 tree $\cT$. 

As in Dimitrov's work and the author's earlier paper in the dynamical
setting,
we establish
a congruence relation modulo a prime $p$.
To this end, 
we systematically use Epstein's work~\cite{epstein:12}.
Epstein proved that postcritically finite $f$ as in
Theorem~\ref{thm:mainintro}
satisfy a strong
integrality condition 
 above the prime $p$. 
Indeed, it suffices to assume that $f$ is
postcritically bounded above $p$, which follows from postcritical
finiteness. Another consequence of Epstein's work allows us to connect
the iterates $f^{(k)}$ with a lift of the Frobenius
map, see Lemma~\ref{lem:fkcongpsqr}. This connection is  the basis of
the desired 
congruence relation modulo a power of $p$.

Let us now discuss some cases where the hypothesis of
Theorem~\ref{thm:mainintro} is met. Part (ii) of the lemma below uses
work of Epstein and Poonen~\cite{epstein:12}. Below, $\pco_{\ge 1}(f)$
denotes the postcritical orbit of $f$, see
Section~\ref{sec:notation}.

\begin{lemma}
  \label{lem:critperimplieshyp}
  Let $f$ be  a polynomial of degree $d\ge 2$
  with coefficients in a field of characteristic $0$.
  \begin{enumerate}
  \item [(i)]   
    Suppose $f\in \IC[T]$ and
    that the forward orbit of each point in $\pco_{\ge 1}(f)$
    contains a critical point. This hold, for example, if all critical
    points of $f$ are periodic. Then $f$ is postcritically finite and
    hyperbolic.
  \item[(ii)] Suppose $f=T^d+b$. If $f$ is postcritically finite, then
    $b$ is an algebraic integer. Moreover, if $0$ is $f$-periodic,
    then   $\IQ(b)/\IQ$ is unramified
    above all prime numbers dividing $d$.
  \end{enumerate}
\end{lemma}

Therefore, our Theorems~\ref{thm:mainintro}, \ref{thm:main2intro}, and
\ref{thm:intropreper} apply to $T^{d}+b$ when $d$ is a prime power and
when the critical point $0$ is periodic.
This is the cornerstone for the height lower bound in Theorem~\ref{thm:unicriticalwandering}. 
Moreover, we  have the next corollary regarding preperiodic points. 

\begin{corollary}
  \label{cor:unicritialpreperiodic}
  Let $p$ be a prime number and $d=p^e$ for some $e\in\IN$. Suppose
  $b$ is an algebraic number such that  
  the critical point $0$ of $f = T^d + b$ is periodic.
  We set
  \begin{equation*}
    \mathfrak{f} =\mathrm{lcm}\{ f_v : v\in M_{\IQ(b)}, v\mid p\}.
  \end{equation*}
  where
  $f_v$ denotes the residue degree of the place $v$.
  Then all algebraic and $f$-preperiodic $\alpha$ satisfy
  \begin{equation}
    \label{eq:unicritialpreperiodic}
    d^{\mathrm{preper}(\alpha)-1} \le [\IQ(b,\alpha):\IQ(b)]
    \quad\text{and}\quad \mathrm{per}(\alpha) \le
    p\mathfrak{f}[\IQ(b,\alpha):\IQ(b)].
  \end{equation}
\end{corollary}

The upper bound on the period length of $\alpha$ grows linearly, while
the upper bound for the preperiod length grows logarithmically.
Let us explain this striking difference and why the upper bound
for the period length is essentially best possible for $f=T^2$.

Say $b=0$. Then in Corollary~\ref{cor:unicritialpreperiodic} we may take $\mathfrak
f=1$ because the base field is $\IQ$. Here $0$ is periodic of period
length $1$. The preperiodic points are $0$ and all roots of unity. Let
$\zeta$ be a root of unity of order $N$. We decompose $N=2^\nu M$
where $\nu\ge 0$ and $M\in\IN$ is odd.

If $N$ is odd, \textit{i.e.}, $\nu=0$, then $\zeta$ is periodic under $T^2$
with period length 
$\mathrm{ord}_{(\IZ/N\IZ)^\times}(2)$, the multiplicative order of $2$ modulo
$N$. If $2$ generates $(\IZ/N\IZ)^\times$, then $\mathrm{per}(\zeta) =
\phi(N)$, where $\phi$ is the Euler totient function.
So the second inequality in (\ref{eq:unicritialpreperiodic})
implies $\phi(N) \le 2[\IQ(\zeta):\IQ]$. The left-hand side
equals $[\IQ(\zeta):\IQ]$, so the upper bound is optimal up-to the
factor $2$. 
Of course, in general $2$ does not generate $(\IZ/N\IZ)^\times$.
A deep
conjecture of Artin states that $2$ generates $(\IZ/N\IZ)^\times$ for
infinitely many primes $N$. On the other hand, it is known
unconditionally that $2$ is a generator if $N\ge 3$ is a power of $3$. 

For $\nu \ge 1$, the root of unity $\zeta$ is preperiodic with
$\mathrm{preper}(\zeta) = \nu$.
The first inequality in (\ref{eq:unicritialpreperiodic})
states $2^{\nu-1}\le [\IQ(\zeta):\IQ]$. In other words,
$2^{\nu-1}\le \phi(N)=\phi(2^\nu)\phi(M) = 2^{\nu-1}\phi(M)$. This
bound is sharp  if $M=1$. 

We use a result of Buff, Floyd, Koch, and Parry~\cite{BFKP:Gleason} to
obtain the following corollary.

\begin{corollary}
  \label{cor:unicritialdeg2preperiodic}
  Let $b$ be an algebraic number such that the critical point $0$
  of $f = T^2 + b$ is periodic of period length $n$. Let
  $\alpha$ be an
  $f$-preperiodic algebraic number, then
  \begin{equation}
    \label{eq:unicriticaldeg2preperiodic}
    2^{\mathrm{preper}(\alpha)-1} \le [\IQ(b,\alpha):\IQ(b)]
    \quad\text{and}\quad \mathrm{per}(\alpha) \le
    n[\IQ(b,\alpha):\IQ(b)].
  \end{equation}
\end{corollary}

Returning to the example with $f=T^2$, we have $n=1$ in the context of
Corollary~\ref{cor:unicritialdeg2preperiodic}. Let $\zeta$ be again a
root of unity of order $N=2^\nu M$ with odd $M$. If $N$ is odd, then
the second bound in (\ref{eq:unicriticaldeg2preperiodic}) reads
$\mathrm{ord}_{(\IZ/N\IZ)^\times}(2) \le [\IQ(\zeta):\IQ]$. This time
the upper bound is sharp if $2$ generates $(\IZ/N\IZ)^\times$
while the first bound in (\ref{eq:unicriticaldeg2preperiodic})
remains unchanged.

Another example that shows that this is sharp is $f = T^2 -1$ for
which $\mathrm{per}(0) = \mathrm{per}(1) = 2$ and $\mathrm{preper}(1)
= 1$.
Moreover, if $\alpha$ satisfies $f^{(s)}(\alpha)=1$, then
$\mathrm{preper}(\alpha) = s+1$ and $[\IQ(\alpha):\IQ]\le 2^{s}$. 
So equality holds on the left of (\ref{eq:unicriticaldeg2preperiodic})
for all $s\ge 0$. 

Our work strengthens the case for an analogy between elliptic curves
with complex multiplication (CM) and postcritically finite maps. In
both cases, preperiodic points enjoy Galois theoretic properties that
do not hold generically. More specifically, if an elliptic curve $E$
has CM, then a torsion point of order $ N$ satisfies \cite[Theorem
6.2]{BourdonClark:Torsion}
\begin{align}
  \label{torsionphi}
  [\mathbb{Q}(j(E), P):\mathbb{Q}(j(E))] \geq \phi(N)/6
\end{align}
where $j(E)$ is the $j$-invariant of $E$. Recall that $\phi(N)
\gg_\epsilon N^{1-\epsilon}$. Such uniformity cannot hold for all
elliptic curves. This can be seen by picking a non-torsion
multi-section of the Legendre family $E_\lambda: Y^2 =
X(X-1)(X-\lambda)$ and showing that it takes torsion values of
arbitrarily high order. This follows for example by using the openness of
the Betti map. So for large enough $N$,  there exists
$\lambda_0 \in \IC$, such that $s(\lambda_0) = (2,
\sqrt{2(2-\lambda_0)})$ has order $N$. We may even take $\lambda_0$ to
be algebraic. Since $j(E_\lambda) \in
\mathbb{Q}(\lambda)$ and $[\mathbb{Q}(s(\lambda_0), \lambda_0):
\mathbb{Q}(\lambda_0)] \leq 2$, we see that a bound such as
(\ref{torsionphi}) can not hold in general. A result of Andr\'e states
that there are at most finitely many $\lambda_0 \in
\IC\setminus\{0,1\}$ such that $E_{\lambda_0}$ has CM and
$s(\lambda_0)$ is torsion.

However, the analogy between CM elliptic curves and postcritically finite
polynomials has its limits.
On the one hand, an elliptic curve with CM defined over a number field
has potentially good reduction at all finite places.
On the other hand, a postcritically finite polynomial $f$ may fail to have
integral coefficients. For example, Ingram~\cite{Ingram:PCF} exhibited
$T^3-\frac 32 T^2$ which is postcritically finite and does not have
potentially good reduction above $2$.
See Section~\ref{sec:examples} for more details. 

It is natural to ask if there are monic polynomials with at least $2$
critical points to which the results in this paper apply. Indeed, in
Section \ref{sec:examples} we show that $(T^2-1)^2\in\IZ[T]$ is a
postcritically finite, hyperbolic quartic with barycenter $0$ and with
$3$ critical points.

Yet the authors do not know if there exists a monic, hyperbolic,
postcritically finite cubic with $2$ critical points, algebraic
coefficients, and integral barycenter.

In future work, we will consider more general $f$, including the
subhyperbolic case.

The paper is structured as follows. In Section~\ref{sec:transfinitei}
we recall basic facts on the transfinite diameter of a compact subset
of the complex plane. Furthermore, we provide an upper bound that
provides a dynamical variant of the first-named author's
bound~\cite{hab:dubinin}. Section~\ref{sec:trees} contains some
background material on the Hubbard tree of a postcritically finite
polynomial. We also cover properties of the core entropy of such
polynomials. In Section~\ref{sec:dynhedgehog} we combine the efforts
of the previous two sections and complete our construction of the
finite topological tree. In Section~\ref{sec:powerseries} we study
integrality conditions which are key to power series construction.
Finally, in Section~\ref{sec:proofs} we prove
Theorems~\ref{thm:mainintro}, \ref{thm:main2intro}, and
\ref{thm:intropreper}
as well as their
generalizations. We cover some additional examples in
Section~\ref{sec:examples}.
Appendix~\ref{app:hyperpolys} collects some classical
definitions and facts on the complex dynamics of polynomials.

\subsection{Notation}
\label{sec:notation}

We describe some further notation.

We let $B_r(z)\subset\IC$ denote the open disk of radius $r>0$
centered at $z\in\IC$.
If $f\in\IC[T]$ has degree at least $2$, then $J_f$ is the Julia set
of $f$ and $K_f$ is the filled Julia set of $f$. Some fundamental
properties are recalled in the Appendix.

Let $F$ be a field, we  fix an algebraic closure $\overline F$ of $F$.
Let $f\in F[T]$ be of degree at least $2$.  Then $\crit{f}$
denote the set of critical points of $f$, \textit{i.e.}, the roots of
the derivative $f'$ in $\overline F$. The \textit{postcritical
orbit} of $f$ is
\begin{equation}
  \label{def:pcoge1}
  \pco_{\ge 1}(f) = \{f^{(n)}(c) : n\in \IN\text{ and } c\in \crit{f}\}.
\end{equation}

When working with rational maps and their dynamics on the projective
line, it is natural to allow the point at infinity to be a critical
point. For example, the point at infinity is always critical for a
polynomial. In this paper, we work exclusively with polynomials. So we
allow ourselves to ignore the point at infinity when considering
critical points of $f$.

Let $F$ be a number field. A \textit{place} of a number field is a
non-trivial absolute value that, when restricted to $\IQ$, is either
the standard absolute value or the $p$-adic absolute value for some
prime number $p$. Let $L$ be a finite field extension of $F$. We write
$M_L$ for the set of places of $L$. We identify each $v\in M_L$ with
the corresponding absolute value $|\cdot|_v\colon L\rightarrow\IR$.
Let us describe places concretely. If $v\in M_L$ is
\textit{archimedean}, then there is a homomorphism $\sigma\colon
L\rightarrow\IC$, unique up-to complex conjugation, such that
$|\alpha|_v = |\sigma(\alpha)|$ for all $\alpha\in L$. We write
$M_L^\infty$ for the set of archimedean places of $L$ and
$v\mid\infty$ to signify that $v$ is archimedean. If $v$ is
\textit{nonarchimedean}, then our normalization yields $|p|_v = 1/p$
for a unique prime number $p$. In this case we write $v\mid p$. For
any place $v$ of $L$, let $L_v$ denote a completion of $L$ with
respect to $v$. We set $d_v=[L_v:\IQ_w]$ where $w$ is the restriction
of $v$ to $\IQ$.

All rings are understood to be commutative with a multiplicative unit.

\subsection{Acknowledgments}
The authors thank Laura DeMarco and Vesselin Dimitrov
for discussions, and the former in particular for the
proof of Lemma~\ref{lem:pcfhyperbolicgalois}. We thank Alessio Cangini
and Tom Tucker
for comments and corrections.
The first-named author acknowledges
support by  the Swiss National Science Foundation grant ``Rational
points, arithmetic dynamics, and heights'' (Nr. 200020\_219397). The second named author thanks the Engineering and Physical Sciences Research Council for support
under the grant ``Heights and equidistribution in dynamical systems and Diophantine geometry'' (UKRI3683).  


\section{Bounding the Transfinite Diameter}
\label{sec:transfinitei}

Let $K\subset \IC$ be a compact and non-empty set. For an integer
$N\ge 2$, the \textit{$N$-th logarithmic diameter} of $K$ is
\begin{equation*}
  \mathsf{d}_N(K) = \frac{2}{N(N-1)} \max_{z_1,\ldots,z_N\in
    K}\log\prod_{1\le i<j\le N} |z_j-z_i| \in\{-\infty\}\cup\IR.
\end{equation*}

It is well-known, see Theorem 5.5.2~\cite{Ransford}, that
$\mathsf{d}_{N+1}(K)\le \mathsf{d}_{N}(K)$ for all $N\ge 2$. Thus the limit
\begin{equation}
  \label{def:dinfty}
  \mathsf{d}_\infty(K) = \lim_{N\rightarrow\infty} \mathsf{d}_N(K) \in \{-\infty\}\cup\IR
\end{equation}
exists and is called the \textit{logarithmic transfinite diameter} of
$K$. It is convenient to set $\mathsf{d}_\infty(\emptyset)=-\infty$. The
\textit{transfinite diameter} is the exponential of the
logarithmic transfinite diameter.

The goal of this section is to prove the following upper bound for the
transfinite diameter of a compact set $K$. Here the $K$ in question is
closely related to the filled Julia set $K_f$ of a polynomial $f$. 
We will apply it later to bound from above the
transfinite diameter of certain trees.
The notion of a hyperbolic polynomial is recalled in
Definition~\ref{def:hyperbolic}.
For $K\subset\IC$ as above, we let
$\mathrm{dist}(z,K) = \min\{|w-z| : w\in K\}$.

\begin{proposition}
  \label{prop:transfiniteapplication}
  Let $f\in\IC[T]$ be monic of degree $d\ge 2$ and hyperbolic. Let
  $r>0$ be sufficiently small in terms of $f$. Let $\delta >0$. There
  exist $c_{1}(f,\delta)>0$ and $c_2(f,\delta)>0$ with the following property. Let
  $n,s\in\IN_0$ with $n\le c_2(f,\delta) d^s$. For each $i\in \{1,\ldots,n\}$ let
  $\mathcal{D}_i\subset\IC$ be an open disk of radius $r$ centered at a point of
  the Julia set
  $J_f$ and let $V_i$ be
  a connected component of $(f^{(s)})^{-1}(\mathcal{D}_i)$. Let
  $K\subset\IC$ be a compact set such that
  \begin{equation}
    \label{eq:propKcontained}
    K \subset V_1\cup\cdots\cup V_n \cup \left\{z\in K_f :
    \mathrm{dist}(f^{(s)}(z),J_f)\ge \delta\right\}.
  \end{equation}
  Then
  \begin{equation}
    \label{eq:corcDNub}
    \mathsf{d}_N(K) \le -\frac{c_1(f,\delta)}{d^s} +
    O_{f}\left(\frac{\log N}{N}\right)
  \end{equation}
  for all integers $N\ge 2$; the implicit constant depends only on $f$. In particular,
  $$\mathsf{d}_\infty(K)\le -\frac{c_1(f,\delta)}{d^s}.$$
\end{proposition}

\subsection{Partial Transfinite Diameters}
Let $f\in\IC[T]$ be monic of degree $d\ge 2$ throughout this subsection.

Let $\mathfrak i \in \{0,\ldots,d-1\}^{\IN_0}$ be a sequence. We let
$\mathfrak{i}_j$ denote the value of $\mathfrak{i}$ at $j\in\IN_0$. 
The set of possible  base-$d$ digit vectors of a non-negative integer is
\begin{equation}
  \label{def:frakI}
  \begin{aligned}
    \mathfrak{I}
    =\bigl\{  \mathfrak{i} \in\{0,\ldots,d-1\}^{\IN_0} : \mathfrak{i}_j =
    0\text{ for all but finitely many $j$} \bigr\}.
  \end{aligned}
\end{equation}
Let $\mathfrak{i}\in \mathfrak{I}$.
So there exists $m$ with  $\mathfrak{i}_j=0$ for all 
$j>m$. We define
\begin{equation}
  \label{eq:definefi}
  f^\mathfrak{i} = (f^{(0)})^{\mathfrak{i}_0}(f^{(1)})^{\mathfrak{i}_1}(f^{(2)})^{\mathfrak{i}_2}\cdots
  (f^{(m)})^{\mathfrak{i}_m}.  
\end{equation}
Then $f^{\mathfrak{i}}$ lies in $\IC[T]$ and has degree $\sum_{j\ge 0}
\mathfrak{i}_j d^j$. For $i\in\IN_0$  we let
$\mathfrak{i}(i) \in \mathfrak{I}$ denote the digit vector of $i$ in
base-$d$.
Therefore, $\deg f^{\mathfrak{i}(i)} = i$.

Let $z_0,\ldots,z_m\in \IC$ and let $N\in\IN$ satisfy $N\ge 2$. 
The matrix 
\begin{equation}
  \label{def:A}
  A = \left(
    \begin{array}{cccc}
      f^{\mathfrak{i}(0)}(z_0)
      & f^{\mathfrak{i}(0)}(z_1)
      &\cdots
      &f^{\mathfrak{i}(0)}(z_m)
      \\
      \vdots
      & \vdots
      &
      & \vdots
      \\
      f^{\mathfrak{i}(N-1)}(z_0)
      &f^{\mathfrak{i}(N-1)}(z_1)
      &\cdots
      &f^{\mathfrak{i}(N-1)}(z_m)
    \end{array}
  \right) \in \mat_{N,m+1}(\IC).
\end{equation}
Baker and Rumely~\cite{BR:06}, see also~\cite{Baker:06}, studied 
 the case of a square matrix, \textit{i.e.}, if $N=m+1$. Then the
 determinant is just a Vandermonde determinant.
For us, it is important
to also cover the case of non-square matrices as in \cite{hab:dubinin}. 
A substitute for the determinant is the square root of $\det
(\overline{A}^tA) \ge 0$ where $t$ indicates transposition and
$\overline{\cdot}$ indicates complex conjugation.
It is convenient to work logarithmically and normalize by
dividing by ${N\choose 2}$. Thus we define
\begin{equation}
  \label{def:dynvandermonde}
  \mathsf{d}_{f,N}(z_0,\ldots,z_m)=\frac{2}{N(N-1)} \log \sqrt{\det({\overline
    A}^t{A})} \in \{-\infty\}\cup\IR.
\end{equation}
This value is invariant under permuting the $z_j$.
It is convenient to extend our definition to the case $m=-1$ by
setting $\mathsf{d}_{f,N}() = 0$ for the empty tuple.

If $N<m+1$, then $\det(\trans{\overline A}A)=0$. So only the case $N\ge
m+1$ will be relevant for us.

\begin{example}
  Say $f = T^2$ and so $d=2$. Then $f^{(i)} = T^{2^i}$ and
  $f^{\mathfrak{i}(i)} = T^i$ for all integers $i\ge 0$.
  If $N=m+1$,
  then   $A$ as in (\ref{def:A})
  equals the Vandermonde matrix $(z_j^{i})_{0\le i,j\le N-1}$.
  Then $\sqrt{\det (\overline{A}^tA)} = |\det A| = \prod_{i<j}|z_j-z_i|$.
  In this case
  $\mathsf{d}_{f,N}(z_0,\ldots,z_{N-1})$ is the $N$-th logarithmic diameter
  of $\{z_0,\ldots,z_{N-1}\}$.
\end{example}

In the square matrix case $N=m+1$ we recover a Vandermonde determinant
for all monic $f$.
Indeed, a more general version of  following lemma is due to
Baker--Rumely, see Lemma 6.3~\cite{BR:06}.
\begin{lemma}
  \label{lem:br}
  Let  $z_0,\ldots,z_{N-1}\in\IC$ with $N\ge 2$. Then
  \begin{equation*}
    \mathsf{d}_{f,N}(z_0,\ldots,z_{N-1}) = \frac{2}{N(N-1)}\log \prod_{0 \le i<j\le
      N-1} |z_j-z_i|.
  \end{equation*}
\end{lemma}
\begin{proof}
  For each $i\in\IN_0$ and $j\in \{0,\ldots,N-1\}$,
  the expression $f^{\mathfrak{i}(i)}(z_j)$ is a monic polynomial in
  $z_j$ of degree $i$.
  Let $A$ be the square matrix given by (\ref{def:A}) in this setting.
  Then $\det(\overline{A}^t A)  = |\!\det A|^2$.
  So $\mathsf{d}_{f,N}(z_0,\ldots,z_{N-1}) = \frac{2}{N(N-1)} \log |\!\det A|$
  by definition. 
  Moreover,  $\det A$
  equals $\prod_{0\le i<j\le N-1} (z_j-z_i)$ by
  the generalized Vandermonde identity, see Proposition~1
  \cite{krattenthaler:adc}. 
\end{proof}

We come to an application of the Hadamard--Fischer Inequality,
compare also Lemma~2.6~\cite{hab:dubinin}. It allows us to decompose the
$N$-th logarithmic transfinite diameter  into parts. 

\begin{lemma}
  \label{lem:fischer}
   Let $N,m_1,\ldots,m_n\in\IZ$
  with $N\ge 2$ and $m_i\ge -1$ for all $i$.
  Let 
  $\bfz_i\in\IC^{m_i+1}$ for all
  $i\in \{1,\ldots,n\}$, then
  \begin{equation*}
    \mathsf{d}_{f,N} (\bfz_1,\ldots,\bfz_n) \le \mathsf{d}_{f,N}(\bfz_1)+\cdots + \mathsf{d}_{f,N}(\bfz_n). 
  \end{equation*}
\end{lemma}
\begin{proof}
  It suffices to prove the bound in the case $m_i\ge 0$ for all $i$.
  By induction on $n$ we may reduce to the case $n=2$.

    Say $A_1\in \mat_{N,m_1+1}(\IC)$ is as in (\ref{def:A}) with
  $z_0,\ldots,z_m$ replaced by the coordinates of $\bfz_1$.
  Similarly, let $A_2\in\mat_{N,m_2+1}(\IC)$ arise from the coordinates
  of $\bfz_2$.
  Finally, set $A = (A_1 | A_2) \in \mat_{N,(m_1+1)+(m_2+1)}(\IC)$.
  By our  definition
  (\ref{def:dynvandermonde}),
  it suffices to prove
  \begin{equation}
    \label{eq:fischertoshow}
    \det(\overline{A}^tA) \le \det(\overline{A_1}^tA_1)
    \det(\overline{A_2}^tA_2).
  \end{equation}

  The hermitian matrix
  \begin{equation*}
    \overline{A}^t A = 
    \left(
      \begin{array}{cc}
        \overline{A_1}^tA_1 & *\\
        * & \overline{A_2}^tA_2
      \end{array}
    \right)\in \mat_{(m_1+1)+(m_2+1)}(\IC)
  \end{equation*}
  is positive semi-definite.
  We apply Fischer's Inequality,
  Theorem 13.5.5~\cite{Mirsky:IntroLA}. If $\overline{A}^tA$ is
  positive definite, then the hypothesis of the reference is
  satisfied and (\ref{eq:fischertoshow}) follows.  In general,  
  $\overline{A}^t A+\epsilon E$ is hermitian and positive definite for all
  $\epsilon>0$ where $E$ is the appropriate unit matrix.
  So Fischer's Inequality implies
  $\det( \overline{A}^tA +\epsilon E)\le
  \det(\overline{A_1}^tA_1+\epsilon E')
  \det(\overline{A_2}^tA_2+\epsilon E'')$
  for all $\epsilon > 0$; here $E',E''$ are again unit matrices.
  We take $\epsilon\rightarrow 0$ from the right and conclude
  (\ref{eq:fischertoshow}). 
\end{proof}

Recall that $f\in\IC[T]$ is monic of degree $d\ge 2$. 
The \textit{local canonical height} of $\alpha\in\IC$ with respect to $f$ is
\begin{equation}
  \label{def:lambdafcomplex}
  \lambda_f(\alpha) = \lim_{n\rightarrow\infty} \frac{\log^+
  |f^{(n)}(\alpha)|}{d^n}\ge 0.
\end{equation}
Then $\lambda_f\colon\IC\rightarrow [0,\infty)$ is a continuous map,
by Theorem 3.27 and Remark 3.30~\cite{Silverman:gtm241}.

The next proposition is our main tool in the proof of
Proposition~\ref{prop:transfiniteapplication}. Its proof will be
delivered in Section~\ref{subsec:genvander}.

\begin{proposition}
  \label{prop:vandermondepreboundhyperbolic}
  Let $f\in\IC[T]$ be monic of degree $d\ge 2$ and hyperbolic.
  Let $\epsilon\in (0,1]$ be sufficiently small in terms of $f$ and
  satisfy the hypothesis of  Lemma~\ref{lem:fninvuniform_new1}. 
  Let $\widehat w\in J_{f}$ and let $w_0\in\IC$
  such that $\overline{B_{R}(w_0)}\subset B_\epsilon(\widehat w)$ with
  $R>0$.
  Let $s\in\IN_0$ and suppose $z_0,\ldots,z_m$ all lie in the same
  connected component of 
  $(f^{(s)})^{-1}(B_{R/256}(w_0))$.
  Then for all integers $N$ with $N\ge \max\{2,m+1\}$ we have
  \begin{equation}
    \label{eq:thmvandermondepreboundhyperbolicbd}
    \mathsf{d}_{f,N}(z_0,\ldots,z_m)\le
    -p^2+
    2p \frac{\max_{w\in \overline{B_R(w_0)}} \lambda_f(w)}{d^s}  + O_f\left(p\frac{\log N}{N}\right)   
  \end{equation}
  where $p=(m+1)/N$ and where the implicit constant depends only on $f$.
\end{proposition}

\subsection{Divided Differences}
\label{subsec:genvander}

Let $m\in\IN_0$. For an integer $k\ge 0$ let
\begin{equation*}
  h_k(X_0,\ldots,X_m) = \sum_{l_0 + \cdots + l_m = k} X_0^{l_0}\cdots X_m^{l_m}
  \in \IZ[X_0,\ldots,X_m]
\end{equation*}
denote the \textit{complete homogeneous symmetric polynomial} of
degree $k$ in the unknowns $X_0,\ldots,X_m$; here $l_0,\ldots,l_m$
run over non-negative integers. So $h_0=1, h_1 =X_0+\cdots + X_m$, etc.
For an integer $k<0$, we define $h_k = 0$.

We have the identity
\begin{equation}
  \label{eq:dd_fundamental}
  (X_0-X_m) h_k(X_0,\ldots,X_m) = h_{k+1}(X_0,X_1,\ldots,X_{m-1}) - h_{k+1}(X_1,\ldots,X_m).
\end{equation}

Let $R$ be an integral domain. Let $f\in R[T]$ with $f = \sum_{k\ge 0}
f_k T^k$. For a tuple $(z_0,\ldots,z_m)\in R^{m+1}$ with $m\ge 0$ we
define
\begin{equation}
  \label{def:fz_0zm}
  f[z_0,\ldots,z_m] = \sum_{k\ge 0} f_{k} h_{k-m}(z_0,\ldots,z_m) \in R.
\end{equation}

Observe that $f[z_0,\ldots,z_m]=0$ if $m>\deg f$. If $f\not=0$ and
$m=\deg f$, then $f[z_0,\ldots,z_m]$ is the leading term of $f$.
In the case $m=0$ we have $f[z_0]=f(z_0)$. 
Each $h_k$ is symmetric in its arguments, so $f[z_0,\ldots,z_m]$ is
symmetric in $z_0,\ldots,z_m$. We thus recover from
(\ref{eq:dd_fundamental})  the familiar divided
differences relations
\begin{equation}
  \label{eq:ddrelation}
  \begin{aligned}
    (z_0-z_m) f[z_0,\ldots,z_m] &=
    f[z_0,z_1,\ldots,z_{m-1}]-f[z_1,\ldots,z_m]\\
    & = f[z_1,\ldots,z_{m-1},z_0]- f[z_1,\ldots,z_{m-1},z_m].
  \end{aligned}
\end{equation}
If $z_0=\cdots=z_m$, then $h_{k}(z_0,\ldots,z_m) = {m+k\choose
  k}z_0^k$ and $f[z_0,\ldots,z_m] =
\frac{1}{m!}\frac{\mathrm{d}^mf}{\mathrm{d}T^m}(z_0)$; this last
expression  is
well-defined
 even if the characteristic of $R$ is positive. 

For all $z,z_0,\ldots,z_{m-1}\in R$ we claim that
\begin{equation}
  \label{eq:hermiteinterpol}
  f(z) = \sum_{i=0}^{m-1} f[z_0,\ldots,z_i] \prod_{j=0}^{i-1} (z-z_j)
  +f[z_0,\ldots,z_{m-1},z] \prod_{j=0}^{m-1} (z-z_j).
\end{equation}
Indeed, this is trivial for $m=0$. If
 the claim holds for $m\ge 0$, then 
 \begin{equation*}
   f(z) = 
  \sum_{i=0}^{m-1} f[z_0,\ldots,z_i] \prod_{j=0}^{i-1} (z-z_j)
  + (f[z_0,\ldots,z_m]  + (z-z_m)f[z_0,\ldots,z_m,z])
  \prod_{j=0}^{m-1} (z-z_j)
\end{equation*}
by (\ref{eq:ddrelation}) and as $f[\,\cdots]$ is symmetric. We
recover (\ref{eq:hermiteinterpol}) for $m+1$.
Thus the claim follows by induction.

For $m > \deg f$ we recover the
Hermite Interpolation Formula  for polynomials
\begin{equation*}
  f(z) =\sum_{i=0}^{m-1} f[z_0,\ldots,z_i] \prod_{j=1}^{i-1} (z-z_j).
\end{equation*}

The following lemma extends the generalized Vandermonde
identity, Proposition~1~\cite{krattenthaler:adc}. The
proof is also given by Krattenthaler in Lemma~22, \textit{loc.cit.},
following Lascoux who pointed out the connection to divided differences.
For sake of completeness we provide the argument in our notation. 

\begin{lemma}
  \label{lem:vandermonde}
  Let $m\in\IN_0,f_0,\ldots,f_m\in R[T]$ and $z_0,\ldots,z_m\in R$. We set  
    $A = (v_0,\ldots,v_m)\in \mathrm{Mat}_{m+1}(R)$
  where
  \begin{equation*}    
    v_j=\trans{(f_0(z_j),f_1(z_j),\ldots,f_m(z_j))}\text{ for }j\in \{0,\ldots,m\}.
  \end{equation*}  
  Then
  \begin{equation}
    \label{eq:genvandermonde}
    \det A = 
    \det
    \bigl(f_i[z_0],f_i[z_0,z_1],f_i[z_0,z_1,z_2],\ldots,f_i[z_0,\ldots,z_m]\bigr)_{0\le
      i\le m}\prod_{0\le i<j\le m} (z_j-z_i).
  \end{equation}
\end{lemma} 
\begin{proof}
  The determinant is multilinear and alternating. So its value does
  not change  when
  subtracting the first column of the argument from the second column, then the first
  column from the third column, and so on until we have reached the
  final column.   We find
  \begin{alignat*}1
    \det A &=
    \det(v_0,v_1-v_0,v_2-v_0,\ldots,v_m-v_0)
  \end{alignat*}
  because the first column remained unchanged. 

  For $j\in \{1,\ldots,m\}$ we use (\ref{eq:ddrelation}) for $m=1$  
  to  factor
  $v_j-v_{0} = (f_i(z_j)-f_i(z_0))^t_{i} = (f_i[z_j]-f_i[z_0])_i^t
  = (z_j-z_{0})
  \bigl(f_i[z_j,z_0]\bigr)_i^t$
  where here and below the index $i$
  runs through $\{0,\ldots,m\}$. 
  As the determinant  is linear in each  column we factor out
  $z_j-z_0$ and get
  \begin{equation*}
    \det A = 
    \det(A_2)\prod_{j=1}^m (z_j-z_0) \quad\text{with}\quad
    A_2=
    \bigl( f_i[z_0], f_i[z_0,z_1], \ldots ,f_i[z_0,z_m]\bigr)_{i} \in
    \mat_{m+1}(R).
  \end{equation*}
  
  We continue with column operations starting by  subtracting the second
  column of $A_2$ from the third column of $A_2$, then the second column
  of $A_2$ from the fourth column of $A_2$ etc. 
   We find
  \begin{alignat*}1
    \det A_2 = \det \bigl(f_i[z_0], f_i[z_0,z_1], f_i[z_0,z_2]-f_i[z_0,z_1],
    \ldots,
    f_i[z_0,z_m]-f_i[z_0,z_1] \bigr)_i
  \end{alignat*}
  as now the first two columns remained unchanged. 
  Let us apply (\ref{eq:ddrelation}) with $m=2$ to all but the first two
  columns. So $f_i[z_0,z_j]-f_i[z_0,z_1] = (z_j-z_1) f_i[z_0,z_1,z_j]$
  for all $j\in \{2,\ldots,m\}$. We now factor  $z_j-z_1$ out of the
  determinant 
  for all $j\ge 2$ and
  conclude
  \begin{equation*}
    \det A_2 = \det(A_3)     \prod_{j=2}^m (z_j-z_1)    
  \end{equation*}
  with
  \begin{equation*}  
    A_3 = \bigl(f_i[z_0], f_i[z_0,z_1], f_i[z_0,z_1,z_2],
    \ldots,
    f_i[z_0,z_1,z_m] \bigr)_i.
  \end{equation*}
  Together with  the equality we obtain
  $\det A = \det(A_3)\prod_{j=1}^m (z_j-z_0)
     \prod_{j=2}^m (z_j-z_1)$. 

  We apply the same procedure to the last $m-3$ columns of $A_3$,
  apply (\ref{eq:ddrelation}), obtain $A_4$ then repeat continue with
  the last $m-4$ columns of $A_4$ etc. After exhausting all columns we
  arrive at (\ref{eq:genvandermonde}).
\end{proof}

We now switch from working over a general base ring $R$ to working
over $\IC$. Our goal is to use the previous lemmas to bound from above the
partial transfinite diameter~(\ref{def:dynvandermonde}).

Recall that $\mathfrak{I}$ was defined in (\ref{def:frakI}); we also
use the notation (\ref{eq:definefi}) for $f^{\mathfrak{i}}$
and $f^{\mathfrak{i}(i)}$ defined just below for $i\ge 0$ an integer.

\begin{lemma}
  \label{lem:DfNddbound}
  Let $f\in \IC[T]$ be monic of degree $d\ge 2$ and let
  $z_0,\ldots,z_m\in\IC$. Let $N\ge \max\{2,m+1\}$ be an integer.
  There exist pairwise distinct integers $i_0,\ldots,i_m\in \{0,\ldots,N-1\}$ with
  \begin{equation}
    \label{eq:DfNz0zmbd}
    \mathsf{d}_{f,N}(z_0,\ldots,z_m) \le \frac{2}{N(N-1)}\left(\sum_{j=0}^m
    \log  \,\bigl|f^{\mathfrak{i}(i_j)}[z_0,\ldots,z_j]\bigr|
    + \sum_{0\le i<j\le m}\log|z_j-z_i|\right) + \frac{2(m+1)\log
  N}{N(N-1)}. 
  \end{equation}
\end{lemma}
\begin{proof}
  Let $A\in \mathrm{Mat}_{N,m+1}(\IC)$ be the matrix (\ref{def:A})
  used in the definition (\ref{def:dynvandermonde}).
  
  The Cauchy--Binet formula states $\det(\overline A^t A) = \sum_{A'}
  |\det A'|^2$  where $A'$ ranges over the $(m+1)\times(m+1)$ submatrices
  of $A$.
  Let $A'$ be such a submatrix with $|\det A'|$ maximal. Then
  $\det(\overline A^t A) \le {N\choose m+1} |\det A'|^2$. 

  We apply Lemma~\ref{lem:vandermonde} to $A'$.
  So $\det A'=\det(A'') \prod_{0\le i<j\le m} (z_j-z_i)$
  where the entries of $A''\in \mathrm{Mat}_{m+1}(\IC)$ 
  are of the form  $f^{\mathfrak{i}(i)}[z_0,\ldots,z_j]$; each
  $i \in\{0,\ldots,N-1\}$ corresponds to one of
  the $m+1$ rows selected from $A'$ and $j\in \{0,\ldots,m\}$.
  
  By the Leibniz Formula there are pairwise distinct
   $i_0,\ldots,i_m\in\{0,\ldots,N-1\}$ with
  $|\det A''|\le (m+1)! \prod_{j=0}^m
  |f^{\mathfrak{i}(i_j)}[z_0,\ldots,z_j]|$.  
  And so
  $$
  \sqrt{\det(\overline A^t A)} \le {N\choose m+1}^{1/2} (m+1)!
  \prod_{j=0}^m 
  |f^{\mathfrak{i}(i_j)}[z_0,\ldots,z_j]|\prod_{0\le i<j\le m}|z_j-z_i|.
  $$
  
  Certainly, ${N\choose m+1}^{1/2}(m+1)! \le {N\choose m+1}(m+1)! \le
  N^{m+1}$. The lemma follows from the definition
  (\ref{def:dynvandermonde}).  
\end{proof}

Our next task is to bound the expressions
$|f^{\mathfrak{i}(i_j)}[z_0,\ldots,z_j]|$ appearing in
(\ref{eq:DfNz0zmbd}) from above. Our tool of choice is  complex analysis.
More precisely, we need the following simple consequence of
Cook's~\cite{Cook:DD} Cauchy Integral Formula for
divided differences.

\begin{lemma}
  \label{lem:cifdd}
  Let $f\in\IC[T]$ and 
  let $\gamma \colon [0,1]\rightarrow \IC$ be a
  smooth Jordan curve such that the interior
  component of $\Gamma=\gamma([0,1])$ contains $z_0,\ldots,z_j \in
  \IC$ where $j\ge 0$.
  Let $L(\gamma)$ denote the length of $\gamma$, 
  then
  \begin{equation*}
    |f[z_0,\ldots,z_j]| \le \frac{L(\gamma)}{2\pi} \frac{\max_{z\in
        \Gamma} |f(z)|}{\prod_{k=0}^{j} \mathrm{dist}(z_k,\Gamma)}.
  \end{equation*}
\end{lemma}
\begin{proof}
  By the Cauchy Integral Formula for divided differences,
  see 2.0 in \cite{Cook:DD} with $\Omega=\IC$, we have
  \begin{equation}
    \label{eq:cifdd}
    f[z_0,\ldots,z_j] = \frac{\pm 1}{2\pi i} \int_\gamma
    \frac{f(z)}{\prod_{k=0}^{j} (z-z_k)} \mathrm{d}z. 
  \end{equation}
  The sign-ambiguity arises as we have not indicated an orientation.
  Cook's hypothesis amounts to $j\ge 2$. However, (\ref{eq:cifdd}) holds
  for $j=0$ and $j=1$ as well. 
  Indeed, if $j=0$, then $f[z_0] = f(z_0)$ and the formula reduces to
  the Cauchy Integral Formula.
  If $j=1$ and $z_0=z_1$, then $f[z_0,z_1] = f'(z_0)$. We again
  refer to the Cauchy Integral Formula. Finally, if $j=1$ and
  $z_0\not=z_1$, then $f[z_0,z_1] = (f(z_0)-f(z_1))/(z_0-z_1)$.
  Moreover,
  $\frac{1}{2\pi i}\int_\gamma \frac{f(z)}{(z-z_0)(z-z_1)} \mathrm{d}z =
  \pm f[z_0,z_1]$ by the Residue Theorem.
  
  The lemma follows from (\ref{eq:cifdd}) and an application of
  the triangle inequality.
\end{proof}

In the special case $z_0=\cdots=z_j$, Cook's result reduces to the
Cauchy Integral Formula since
$f[z_0,\ldots,z_0]= \frac{1}{j!} \frac{\mathrm{d}^j
  f}{\mathrm{d}T^j}(z_0)$.

\begin{lemma}
  \label{lem:fngrowthlambda}
  Let $f\in\IC[T]$ be monic of degree at least $2$. We have
  \begin{equation*}
    \log^+ |f^{\mathfrak{i}}(z)|\le \deg(f^{\mathfrak{i}})
\lambda_f(z) + O_f\left(\log(1+\deg f^{\mathfrak{i}})\right)
  \end{equation*}
  for all $\mathfrak{i}\in\mathfrak{I}$ and all $z\in\IC$. Here the
  implicit constant depends only on $f$. 
\end{lemma}
\begin{proof}
  If $\mathfrak{i}=0$, then the claim follows as $f^{\mathfrak{i}}=1$
  in this case. So we
  may assume $\mathfrak{i}\not=0$.
  
  The local canonical height $\lambda_f$ and $z\mapsto \log^+|z|$ differ by
  a bounded function on $\IC$, see  Theorem
  3.27~\cite{Silverman:gtm241}. In particular, 
  there exists $B=B(f)\ge 1$ such that $\log^+ |z| \le \log
  B+\lambda_f(z)$ for all $z\in\IC$. As $\lambda_f (f^{(j)}(z)) = d^j
  \lambda_f(z)$ with $d=\deg f$ we obtain
  \begin{equation*}
    \max\{1,|f^{(j)}(z)|\} \le B e^{ d^j \lambda_f(z)} \quad\text{for all
      $z\in\IC$ and all $j\in \IN_0$.}
  \end{equation*}
  
  We set
  $i = \sum_{j=0}^u \mathfrak{i}_j
  d^j$ where $\mathfrak{i}_j\in \{0,\ldots,d-1\}$
  and $\mathfrak{i}_j=0$ for all $j>u$. 
  Then $i = \deg f^{\mathfrak i}\ge 1$ and
  \begin{equation*}   
    \max\{1,|f^{\mathfrak{i}}(z)|\} = \prod_{j=0}^u
    |f^{(j)}(z)|^{\mathfrak{i}_j}
    \le \prod_{j=0}^u (B^{\mathfrak{i}_j} e^{\mathfrak{i}_j d^j\lambda_f(z)})
    \le B^{(d-1)(u+1)} e^{\sum_{j=0}^u\mathfrak{i}_j d^j \lambda_f(z)}
    =B^{(d-1)(u+1)} e^{i \lambda_f(z)}. 
  \end{equation*}
  If $\mathfrak{i}_j\not=0$, then $j \le (\log i)/\log d$.  So $u+1\le
  (\log i)/\log d+1$ and
  \begin{equation*}
    \max\{1,|f^{\mathfrak{i}}(z)|\}
    \le B^{d-1} i^{\frac{d-1}{\log d}\log B} e^{i \lambda_f(z)}.
  \end{equation*}
  The lemma follows on
  applying logarithm
  and since $i=\deg f^{\mathfrak i}$. 
\end{proof}

We suppose from now on that $f\in\IC[T]$ is monic of degree $d\ge 2$
and hyperbolic.

Before proceeding, let us set up some notation. Let
$\epsilon=\epsilon(f)>0$ be sufficiently small in terms of $f$, as in 
Lemma~\ref{lem:fninvuniform_new1}. We may assume $\epsilon\le 1$.
Let $\widehat w\in J_{f}$ and $s\in\IN_0$.
Suppose $V$ is a connected component of
$(f^{(s)})^{-1}(B_{\epsilon}(\widehat w))$. 
The map
$f^{(s)}|_V\colon V\rightarrow B_{\epsilon}(\widehat w)$ is
biholomorphic by Lemma~\ref{lem:fninvuniform_new1}.
We fix the unique $\widehat z\in V$ with
$f^{(s)}(\widehat z)=\widehat w$. 
The inverse of $f^{(s)}|_V$ is
$g\colon B_{\epsilon}(\widehat w)\rightarrow V$.
In particular,  $g(B_\epsilon(\widehat w)) = V$ and $\widehat z =
g(\widehat w)$. 

We fix real parameters $R$ and $r$  with
\begin{equation}
  \label{eq:rRhyp}
  0 < r \le \frac R2.
\end{equation}

\begin{lemma}
  \label{lem:ddbound}
  Suppose $w_0\in \IC$ with
  $\overline{B_{R}(w_0)} \subset B_{\epsilon}(\widehat w)$.
  Let $j\in\IN_0$ and $z_0,\ldots,z_j\in V$ such that
  $f^{(s)}(z_0),\ldots,f^{(s)}(z_j) \in B_r(w_0)$.
  Then for all $\mathfrak i \in \mathfrak{I}$ we have
  \begin{equation*}
    \log \bigl|f^{\mathfrak{i}}[z_0,\ldots,z_j]\bigr|\le
    \frac{\deg(f^{\mathfrak{i}})}{d^s}    \max_{w\in \overline{B_R(w_0)}} \lambda_f(w)
    + j \log\frac{16 |{f^{(s)}}'(\widehat z)|}{R} 
    +O_f(\log(1+\deg
    f^{\mathfrak{i}}))
  \end{equation*}
  where the implicit constant depends only on $f$. 
\end{lemma}
\begin{proof} 
  By hypothesis,  the boundary of $B_R(w_0)$ lies in the domain
  $B_\epsilon(\widehat w)$ of $g$.  We consider the  smooth Jordan curve
  \begin{equation}
    \label{eq:defjordancurve}
    [0,1]\ni t\mapsto \gamma(t) = g\left(w_0+R e^{2\pi it}\right)
  \end{equation}
  with image contained in $V$; recall that $g$ is biholomorphic.

  To bound the divided differences we
  will use  Lemma~\ref{lem:cifdd} applied to
  $\gamma$ and $f^{\mathfrak{i}}$.
  Indeed, 
  \begin{equation}
    \label{eq:fiddbound1_0}
    \log |f^{\mathfrak{i}}[z_0,\ldots,z_j]| \le
    \log \frac{L(\gamma)}{2\pi} + \max_{z\in
      \gamma([0,1])}|f^{\mathfrak{i}}(z)|
    -\sum_{k=0}^j \log
    \mathrm{dist}(z_k,\gamma([0,1])).
  \end{equation}

  We have
  \begin{equation*}
    |\gamma'(t)| = 2\pi R \left|g'( w_0+ R e^{2\pi it})\right|.
  \end{equation*}
  Therefore, $|\gamma'(t)|\le {4\pi} R|g'(\widehat w)|$ by
  Lemma~\ref{lem:fninvuniform_new1}(i).  
  Recall that $g$ is the inverse of $f^{(s)}|_V$. The chain
  rule implies $g'(f^{(s)}(\widehat z)) {f^{(s)}}'(\widehat z)=1$. So
  $g'(\widehat w) = 1/{f^{(s)}}'(\widehat z)$. Thus
  \begin{equation}
    \label{eq:Lgammabound}
    L(\gamma)\le     \frac{4\pi R}{|{f^{(s)}}'(\widehat z)|}.  
  \end{equation}

  Since
  $\lambda_f(f^{(s)}(z)) = d^s\lambda_f(z)$ for all $z\in\IC$ we find
  \begin{equation}
    \label{eq:smalllambdafzj}
    \lambda_f(z)
    \le \frac{\lambda_{\mathrm{max}}}{d^s}\quad\text{for all}\quad
    z\in g(\overline{B_{R}(w_0)}),
  \end{equation}
  where
  \begin{equation*}
    \lambda_{\mathrm{max}} = \max_{w\in \overline{B_R(w_0)}} \lambda_f(w).
  \end{equation*}
  Lemma~\ref{lem:fngrowthlambda}
  yields
  $\log^+ |f^{\mathfrak i}(z)| \le \deg(f^{\mathfrak i}) \lambda_f(z)+  O_f(\log(1+  \deg f^{\mathfrak i}))$
     for all $z\in\IC$.
  So
  \begin{equation}
    \label{eq:absffrakibd}
    \log^+ |f^{\mathfrak i}(z)| \le
    \frac{\deg(f^{\mathfrak i})}{d^s}\lambda_{\mathrm{max}}+
    O_f\left(\log(1+  \deg f^{\mathfrak i})\right)
    \quad\text{for all}\quad z\in \gamma([0,1])
  \end{equation}
  by (\ref{eq:smalllambdafzj}).

  By hypothesis, 
  $f^{(s)}(z_0),\ldots,f^{(s)}(z_j)$ lie in $B_{r}(w_0)\subset
  B_R(w_0)$.
  So each image $g(f^{(s)}(z_k))=z_k$ is in the 
  interior 
  component of $\gamma([0,1])$.  
  We apply (\ref{eq:Lgammabound}) and (\ref{eq:absffrakibd}) to
  (\ref{eq:fiddbound1_0}) to find
  \begin{equation}
    \label{eq:fiddbound1_1}
    \log |f^{\mathfrak{i}}[z_0,\ldots,z_j]| \le
    \log \frac{2R}{|{f^{(s)}}'(z)|} + \frac{\deg
      (f^{\mathfrak{i}})}{d^s}\lambda_{\mathrm{max}} -\sum_{k=0}^j \log
    \mathrm{dist}(z_k,\gamma([0,1])) + O_f(\log(1+  \deg f^{\mathfrak
      i})).  
  \end{equation}

  We proceed by bounding each distance in the sum from below. 
  Say $z\in \gamma([0,1])$ and $k\in \{0,\ldots,j\}$.
  Our immediate task is to bound $|z_k-z|$ from below.  
  Note that $z$ lies on the curve $\gamma([0,1])\subset
  g(\overline{B_{R}(w_0)})\subset V$ and that
  $z_k \in g(B_r(w_0))\subset V$.
  We apply
  Lemma~\ref{lem:fninvuniform_new1}(ii) to $z_k$ and $z$ to find
  \begin{equation*}
    |z_k-z|\ge \frac 18
    \frac{|f^{(s)}(z_k)-f^{(s)}(z)|}{|{f^{(s)}}'(\widehat z)|}.
  \end{equation*}
  Furthermore,
  (\ref{eq:defjordancurve}) implies
  $|f^{(s)}(z)-w_0|=R$. The hypothesis 
  implies $|f^{(s)}(z_k)-w_0|<r$.
  So
  $|f^{(s)}(z_k)-f^{(s)}(z)|>R-r\ge R/2$ by (\ref{eq:rRhyp}). Therefore,
  $|z_k-z| \ge \frac{R}{16} |{f^{(s)}}'(\widehat z)|^{-1}$. As $z$
  was an arbitrary point in $\gamma([0,1])$ we conclude
  \begin{equation*}
    \mathrm{dist}(z_k,\gamma([0,1])) \ge
    \frac{R}{16{|{f^{(s)}}'(\widehat z)|}}\quad\text{for all}\quad
    k\in \{0,\ldots,j\}. 
  \end{equation*}

  To conclude the proof we insert this last bound into
  (\ref{eq:fiddbound1_1}) to
  obtain
  \begin{alignat*}1
    \log \bigl|f^{\mathfrak{i}}[z_0,\ldots,z_j]\bigr| &\le
    \frac{\deg (f^{\mathfrak{i}})}{d^s}    \lambda_{\mathrm{max}} 
    + \log\frac{2R}{|{f^{(s)}}'(\widehat z)|} -
    (j+1)\log\frac{R}{16|{f^{(s)}}'(\widehat z)|}+O_f(\log(1+\deg
    f^{\mathfrak{i}}))\\
    &\le  \frac{\deg (f^{\mathfrak{i}})}{d^s}\lambda_{\mathrm{max}}
    + j \log\frac{16 |{f^{(s)}}'(\widehat z)|}{R} + \log 32
    +O_f(\log(1+\deg
    f^{\mathfrak{i}})).
  \end{alignat*}
  If $\deg f^{\mathfrak{i}}\ge 1$, then we may safely include $\log
  32$ in the error term. If the degree vanishes, \textit{i.e.}, if
  $\mathfrak i=0$, and if $j\ge 1$, then
  $f^{\mathfrak{i}}[z_0,\ldots,z_j]$ vanishes and the lemma follows
  trivially.
  Finally, if $\mathfrak i=0$ and  $j=0$, then $f^{\mathfrak
    i}[z_0]=1$ and the lemma also follows.
\end{proof}

\begin{proof}[Proof of Proposition~\ref{prop:vandermondepreboundhyperbolic}]
  Let $\epsilon\in (0,1]$ be as in the hypothesis. So
  Lemma~\ref{lem:fninvuniform_new1} applies.
  We choose $r = R/256$, so (\ref{eq:rRhyp}) is satisfied. We abbreviate
  $\lambda_{\mathrm{max}} = \max_{w\in \overline{B_R(w_0)}}
  \lambda_f(w)$.
  Note that   $(f^{(s)})^{-1}(B_{R/256}(w_0))$ is contained in 
  $(f^{(s)})^{-1}(B_{\epsilon}(\widehat w))$. It follows from the
  hypothesis, that
  some  connected component, say $V$, of
  $(f^{(s)})^{-1}(B_{\epsilon}(\widehat w))$ contains
  $z_0,\ldots,z_m$. 
  We set $\widehat z$ to be the unique point in $V$ with
  $f^{(s)}(\widehat z)= \widehat
  w$, just as before Lemma~\ref{lem:ddbound}.
  Note that $f^{(s)}(z_j) \in B_{R/256}(w_0)=B_{r}(w_0)$ for all $j$
  by hypothesis.
  
  We shall apply Lemma~\ref{lem:DfNddbound} combined with
  Lemma~\ref{lem:ddbound}. Say $i_j$ come from the former lemma, then
  $\deg f^{\mathfrak{i}(i_j)} =
  i_j\le N-1$. For brevity, we write $X = \frac{N(N-1)}{2} \mathsf{d}_{f,N}(z_0,\ldots,z_m)$.  
  Then we find
  \begin{equation}
    \label{eq:boundX1}
    \begin{aligned}
      X \le  
      \frac{(N-1)(m+1)}{d^s} \lambda_{\mathrm{max}} +
      \sum_{j=0}^m j\log\frac{16|{f^{(s)}}'(\widehat z)|}{R} +\sum_{0\le
        i<j\le m} \log|z_j-z_i|
      +  O_{f}\left((m+1)\log N\right).
    \end{aligned}
  \end{equation}
  
  For all $i,j$, Lemma~\ref{lem:fninvuniform_new1}(ii) yields
  \begin{equation*}
    |z_i-z_j| 
    \le 2|{f^{(s)}}'(\widehat z)|^{-1}|f^{(s)}(z_i)-f^{(s)}(z_j)| 
  \end{equation*}
  as $z_i$ and $z_j$ both lie in $V$.
  By hypothesis $|f^{(s)}(z_i)-f^{(s)}(z_j)|\le R/128=2r$.
  So $|z_i-z_j|\le 4r |{f^{(s)}}'(\widehat z)|^{-1}$.
  We insert these bounds into (\ref{eq:boundX1}) and get
  \begin{alignat*}1
    X&\le  \frac{(N-1)(m+1)}{d^s}\lambda_{\mathrm{max}} + \sum_{j=0}^m
    j\log\frac{16|{f^{(s)}}'(\widehat z)|}{R} +\sum_{0\le i<j\le m}\log
    \frac{4r}{|{f^{(s)}}'(\widehat z)|} + O_f((m+1)\log N)
    \\
    &=  \frac{(N-1)(m+1)}{d^s} \lambda_{\mathrm{max}}+
    \frac{m(m+1)}{2} \left(\log\frac{16|{f^{(s)}}'(\widehat z)|}{R}
      +\log\frac{4r}{|{f^{(s)}}'(\widehat z)|}\right)
    + O_f((m+1)\log N)
    \\
    &=
      \frac{(N-1)(m+1)}{d^s}    \lambda_{\mathrm{max}}  +
    \frac{m(m+1)}{2}\log\frac{64r}{R} +O_f((m+1)\log N).
  \end{alignat*}

  Recall that $R=256r$ and so $\log(64r/R)=-\log 4\le -1$.
  Therefore,
  \begin{equation*}
    \mathsf{d}_{f,N}(z_0,\ldots,z_m)=\frac{2}{N(N-1)}X
    \le 2p\frac{\lambda_{\mathrm{max}}}{d^s} - \frac{m}{N-1}p +O_f\left(p\frac{\log N}{N-1}\right)
  \end{equation*}
  with $p=(m+1)/N$.
  Observe that 
  $-\frac{m}{N-1}p =-p^2 + \frac{p(1-p)}{N-1} \le
  -p^2 + \frac{2p}{N}$.
  The proposition follows.
\end{proof}

\subsection{Proof of Proposition~\ref{prop:transfiniteapplication}}

Before we come to the proof of
Proposition~\ref{prop:transfiniteapplication}, we show
an elementary application of the  Cauchy--Schwarz inequality.

\begin{lemma}
  \label{lem:qflowerbound}
  Let $n\in\IN_0$ and $p_0,\ldots,p_n\ge 0$ with $p_0+\cdots+p_n=1$ and
  $\alpha,\beta\in (0,\infty)$. Suppose $X>0$ satisfies
  \begin{equation*}    
    X \ge \left(\alpha+2\beta+\frac{\beta^2}{2\alpha}\right)n. 
  \end{equation*}
  Then
  \begin{equation*}
    \alpha\frac{p_0^2}{X} +\sum_{i=1}^n\left(p_i^2 - \beta\frac{p_i}{X}\right)
    \ge
    \frac{\alpha}{2X}.
  \end{equation*} 
\end{lemma}
\begin{proof}
  The claim is clear for $n=0$ as then $p_0=1$. So let us assume $n\ge 1$.
  The Cauchy--Schwarz Inequality implies
  $\sum_{i=1}^n p_i \le n^{1/2} (\sum_{i=1}^n p_i^2)^{1/2}$. So
  $\sum_{i=1}^n p_i^2 \ge (1-p_0)^2/n$ because $p_1+\cdots+p_n=1-p_0$.
  Therefore,
  \begin{alignat*}1
    \alpha\frac{p_0^2}{X}  +\sum_{i=1}^n\left(p_i^2 -\beta\frac{p_i}{X}\right) &\ge 
    \alpha\frac{p_0^2}{X}  +\frac{(1-p_0)^2}{n} - \beta\frac{1-p_0}{X}\\
    &= \left(\frac{\alpha}{X} +\frac 1n\right)p_0^2 +
    \left(\frac{\beta}{X}-\frac{2}{n}\right)p_0 +\frac{1}{n} - \frac{\beta}{X}.
  \end{alignat*}

  Let us factor out $X^{-1}$, so
  \begin{equation}
    \label{eq:factoroutXinv}
    \alpha\frac{p_0^2}{X}  +\sum_{i=1}^n\left(p_i^2 -\beta\frac{p_i}{X}\right)
    \ge
    \left(\left(\alpha+\frac{X}{n}\right)p_0^2 +
      \left(\beta-2\frac{X}{n}\right)p_0 + \frac{X}{n} - \beta\right)\frac{1}{X} .
  \end{equation}

  We consider the factor in front of $\frac 1X$
  on the right-hand side
  as a quadratic polynomial in $p_0$.
  Its leading coefficient is positive and so its minimal value equals
  \begin{equation}
    \label{eq:minvalueXn}
    -\frac{(\beta/2-X/n)^2}{\alpha+X/n} + \frac Xn - \beta
    = -\frac{(\beta^2/4+\alpha\beta)\frac{n}{X}-\alpha}{\alpha\frac{n}{X}+1}.
  \end{equation}

  On its domain, the
  derivative of $z\mapsto (az+b)/(cz+d)$ has the same sign as $ad-bc$.
  The derivative of $z\mapsto ((\beta^2/4+\alpha\beta)z-\alpha)/(\alpha z+1)$
  has the same sign as $\beta^2/4+\alpha\beta+\alpha^2>0$.  It
  follows that  (\ref{eq:minvalueXn}) is decreasing in
  $z=n/X$. We have
  $-1/\alpha < n/X \le 2\alpha/(2\alpha^2 + 4\alpha\beta + \beta^2)$ by  
  hypothesis.  Therefore,
  \begin{equation*}    
    -\frac{(\beta/2-X/n)^2}{\alpha+X/n} + \frac Xn - \beta
    \ge
       -\frac{(\beta^2/4+\alpha\beta)\frac{2\alpha}{2\alpha^2+4\alpha\beta+\beta^2}-\alpha}{\alpha\frac{2\alpha}{2\alpha^2+4\alpha\beta+\beta^2}+1}
       =\frac{\alpha}{2}. 
  \end{equation*}  
  The lemma now follows from (\ref{eq:factoroutXinv}). 
\end{proof}

We now prove Proposition~\ref{prop:transfiniteapplication}. We let $f$
be as in the hypothesis, $d=\deg f\ge 2,$ and $\delta >0$. 

The final claim of the proposition follows from taking the limit
$N\rightarrow \infty$ in (\ref{eq:corcDNub}). So
it suffices to prove (\ref{eq:corcDNub}). 

We fix $\epsilon \in (0,1]$  sufficiently small such that
Proposition~\ref{prop:vandermondepreboundhyperbolic} applies.
Thus $\epsilon$ depends only on $f$.

We may assume that $r$ 
from Proposition~\ref{prop:transfiniteapplication}
is small in terms of $f$. So we assume
\begin{equation*}
  r\le \frac{\epsilon}{512}.
\end{equation*}

Let $n,s\in\IN_0$ with $n\le c_2(f,\delta)d^s$; where $c_2(f,\delta)$
and also $c_1(f,\delta)$ will be fixed during the argument.

Let us suppose that the proposition is known if $\delta\le \epsilon/2$. Let
$\delta > \epsilon /2$. Then  $K$ is a subset of 
\begin{equation*}
  \bigcup_{i=1}^n {V_i} \cup \{z\in K_f :
  \mathrm{dist}(f^{(s)}(z),J_f)\ge\epsilon /2\}.
\end{equation*}
As the proposition holds for $\delta=\epsilon/2$, we find 
$\mathsf{d}_N(K)\le -c_1(f,\epsilon/2)/d^s + O_{f}((\log
N)/N)$ if $n\le c_2(f,\epsilon/2)d^s$ holds.
We have reduced the proof of the current proposition to the case
\begin{equation}
  \label{eq:deltaepshalf}
  \delta \le \frac{\epsilon}{2}.
\end{equation}

Let $N\in\IZ$ with $N\ge 2$. The $N$-th logarithmic diameter is
monotone, \textit{i.e.}, $\mathsf{d}_N(K)\le \mathsf{d}_N(K')$ when
$K\subset K'\subset \IC$. So we may assume that
\begin{equation}
  \label{eq:Kisdeltafilled}
  K \supset \{z\in K_f :    \mathrm{dist}(f^{(s)}(z),J_f)\ge \delta\}.
\end{equation}

We fix a tuple $(z_1,\ldots,z_N)$ of points in $K$ for which the
maximum is attained in the $N$-th logarithmic diameter, \text{i.e.}, 
\begin{equation}
  \label{eq:feketechoice}
  \mathsf{d}_N(K) = \frac{2}{N(N-1)} \log\prod_{1\le i<j\le N} |z_j-z_i|.
\end{equation}
Such a tuple exists by compactness and is often called a
\textit{Fekete $N$-tuple} in the literature. We may assume without
loss of generality that $z_1,\ldots,z_N$ are pairwise distinct.

Using the structure imposed by  (\ref{eq:propKcontained}) we will
permute and partition $(z_1,\ldots,z_N)$
into $n+1$ tuples as follows.

First we collect all points $z_{0,0},\ldots,z_{0,m_0}$ that lie in
$K_f$
and satisfy 
$\mathrm{dist}(f^{(s)}(z_{0,j}),J_f)\ge \delta$.
The number of such points is $m_0+1\ge 0$.
Then each remaining point lies in ${V_i}$ for some
$i\in\{1,\ldots,n\}$  (which need not be unique).
For each remaining point we fix such an $i$.
Thus for all $i\in \{1,\ldots,n\}$ we have an $(m_i+1)$-tuple
$(z_{i,0},\ldots,z_{i,m_i})$ of points in ${V_i}$ with $m_i\ge -1$.

Thus we rearrange our Fekete $N$-tuple $(z_1,\ldots,z_N)$ as
\begin{equation*}
  \left(z_{0,0},\ldots,z_{0,m_0},z_{1,0},\ldots,z_{1,m_1},\ldots,
    z_{n,0},\ldots,z_{n,m_n}\right)
\end{equation*}
where $m_i \ge -1$ and
\begin{equation*}
  N = (m_0+1) + (m_1+1)+\cdots +(m_n+1). 
\end{equation*}
For brevity, we set $\bfz_i = (z_{i,0},\ldots,z_{i,m_i})$ and 
$p_i = (m_i+1)/N$ for all $i$. We have
$p_0+p_1+\cdots+p_n=1$. The case $n=0$ is allowed, but then $p_0=1$.

Let us consider the zeroth tuple more closely and make a simple but
important observation.  We claim that
\begin{equation}
  \label{eq:crucialcommentmaxmod}
  \mathrm{dist}(f^{(s)}(z_{0,j}),J_f)=\delta\quad\text{for all}\quad j\in
  \{0,\ldots,m_0\}. 
\end{equation}
Indeed, by construction the distance is certainly at least $\delta$. 
Suppose  $\mathrm{dist}(f^{(s)}(z_{0,j}),J_f)>\delta$ for some $j$.
Then this inequality holds on an 
open neighborhood of $z_{0,j}$.
Recall that $K_f\ssm J_f$ is 
open in $\IC$ and contains $z_{0,j}$.
Thus $K$ contains an open neighborhood
of $z_{0,j}$ by (\ref{eq:Kisdeltafilled}). We consider
$\prod_{1\le u<v\le N}(z_u-z_v)$ as a non-constant polynomial in
$z_{0,j}$, holding the other $z_i$ fixed.
This contradicts the maximum principle as $z_{0,j}$ is
member of a Fekete $N$-tuple for $K$.

We proceed by bounding $\mathsf{d}_N(K)$ from above. We have
\begin{equation}
  \label{eq:DNKsumbound}
  \mathsf{d}_N(K)=\frac{2}{N(N-1)} \log\prod_{1\le i<j\le N} |z_j-z_i|
  = \mathsf{d}_{f,N}(z_1,\ldots,z_N)=\mathsf{d}_{f,N}(\bfz_0,\bfz_1,\ldots,\bfz_n)
\end{equation}
where the middle equality used Lemma~\ref{lem:br}.
Lemma~\ref{lem:fischer} implies
\begin{equation}
  \label{eq:bakerchangeofcoordinates}
  \mathsf{d}_N(K) \le \sum_{i=0}^n \mathsf{d}_{f,N}(\bfz_i). 
\end{equation}
We continue by bounding each term $\mathsf{d}_{f,N}(\bfz_i)$  from above. The two cases $i=0$ and
$i\in \{1,\ldots,n\}$ are treated separately. 
Below $c_3(f,\delta)$ and $c_4(f)$ are positive and depend only on
$f,\delta$ and $f$, respectively. Moreover, the constants implicit in
$O_f(\cdot)$ depend only on $f$.

\begin{lemma}
  \label{lem:DfNieq0}
  We have
  \begin{equation*}
    \mathsf{d}_{f,N}(\bfz_0)
    \le -c_3(f,\delta)\frac{p_0^2}{d^s}+ O_{f}\left(p_0\frac{\log N}{N}\right).
  \end{equation*}
\end{lemma}
\begin{proof}
  We may assume $p_0>0$.
  By compactness there exist $\widetilde w_1,\ldots,\widetilde w_t\in
  K_f \cap J_{f,=\delta}$, where  $J_{f,=\delta} = \{z\in\IC :\mathrm{dist}(z,J_f)=\delta\}$, such that
  \begin{equation*}
    K_f \cap J_{f,=\delta} \subset \bigcup_{i=1}^t B_{\delta/512}(\widetilde
    w_i). 
  \end{equation*}
  Observe that $t$ and the $\widetilde w_i$ depend on $f$ and
  $\delta$.

  Let $i\in \{1,\ldots,t\}$. We claim that ${B_{\delta}(\widetilde w_i)}\subset  K_f$. 
  As the distance of $\widetilde w_i$ to $J_f$ equals $\delta$, we find
  ${B_{\delta}(\widetilde w_i)} \subset \IC\ssm J_f$.
  So the  disk ${B_{\delta}(\widetilde w_i)}$  lies in a connected
  component of $\IC\ssm J_f$.
  Since $\widetilde w_i\in K_f$ we find that ${B_{\delta}(\widetilde
    w_i)}$ does not meet  $\IC\ssm K_f$
  by Lemma~\ref{lem:fatoucomponents}(v). 
  So
  ${B_{\delta}(\widetilde
    w_i)}\subset K_f$, as desired.

  Let us  fix $\widehat w_i \in J_f$ with
  $|\widehat w_i-\widetilde w_i|=\delta$
  for all $i\in \{1,\ldots,t\}$.

  Each coordinate $z_{0,j}$ of $\bfz_0$ satisfies $f^{(s)}(z_{0,j}) \in
  K_f\cap J_{f,=\delta}$ by (\ref{eq:crucialcommentmaxmod}).
  Thus each  coordinate of $\bfz_0$ lies
  in a connected component
  of  $(f^{(s)})^{-1}(B_{\delta/512}(\widetilde w_i))$ for some
  $i\in\{1,\ldots,t\}$.
  Let $\tau$ be the number of such connected components as $i$ varies.  
  We collect 
  coordinates according to  the  connected components containing them.
  After permuting coordinates,
  $\bfz_0=(\bfz_{0,1},\ldots,\bfz_{0,\tau})$
  where each $\bfz_{0,j}$ consists of coordinates in a fixed connected
  component of some $(f^{(s)})^{-1}(B_{\delta/512}(\widetilde
  w_i))$.

  For given $i$, the preimage $(f^{(s)})^{-1}(B_{\delta/512}(\widetilde w_i))$
  has at most $\deg f^{(s)}=d^s$ connected components. This holds as
  $f^{(s)}$ induces a proper and open map $\IC\rightarrow\IC$. 
  There are $t$ possibilities for
  $i$. So $\tau \le d^s t$. Recall that $t$ depends only on $f$ and
  $\delta$.
  
  We apply Proposition~\ref{prop:vandermondepreboundhyperbolic} to
  $\widehat w = \widehat w_i, w_0 = \widetilde w_i, R=\delta/2,$ and
  $\epsilon= \epsilon$.
  By (\ref{eq:deltaepshalf}) we find
  $\overline{B_{\delta/2}(\widetilde w_i)} \subset
  B_{\epsilon}(\widehat w_i)$.  
  Observe that  $\lambda_f$ vanishes
  on ${B_{\delta}(\widetilde w_i)}\subset K_f$ and on
  $\overline{B_{\delta/2}(\widetilde w_i)}$.
  So
  \begin{equation*}
    \mathsf{d}_{f,N}(\bfz_{0,j}) \le -p_{0,j}^2 + O_f\left(p_{0,j}\frac{\log N}{N}\right)
  \end{equation*}
  for all $j\in \{1,\ldots ,\tau\}$; here $p_{0,j}\in [0,1]$ equals the number
  of coordinates of $\bfz_{0,j}$ divided by $N$.  
  Observe that $p_0 = \sum_{j=1}^\tau p_{0,j}$. By Lemma~\ref{lem:fischer} we have
  \begin{equation*}
    \mathsf{d}_{f,N}(\bfz_0) \le\sum_{j=1}^\tau \mathsf{d}_{f,N}(\bfz_{0,j})
    \le -\sum_{j=1}^\tau p_{0,j}^2 + O_f\left(p_0\frac{\log N}{N}\right).
  \end{equation*}
  The Cauchy--Schwarz inequality implies
  $p_0 = \sum_{j=1}^\tau p_{0,j}\le \tau^{1/2}\left(\sum_{j=1}^\tau
    p_{0,j}^2\right)^{1/2}$. So
  \begin{equation*}
    \mathsf{d}_{f,N}(\bfz_0) \le -\frac{p_0^2}{\tau} + O_f\left(p_0\frac{\log N}{N}\right).
  \end{equation*}
  The lemma follows from $\tau \le d^s t$ with $c_3(f,\delta)=t$. 
\end{proof}

Next we bound $\mathsf{d}_{f,N}(\bfz_i)$ in
(\ref{eq:bakerchangeofcoordinates})
from above for $i\ge 1$.
\begin{lemma}
  \label{lem:DfNige1}
  We have
  \begin{equation}
    \label{eq:DfNige1}
    \mathsf{d}_{f,N}(\bfz_i) \le -p_i^2 + c_4(f)\frac{ p_i}{d^s}  +
    O_{f}\left(p_i\frac{\log N}{N}\right)
    \quad\text{for all}\quad i\in \{1,\ldots,n\}.    
  \end{equation}
\end{lemma}
\begin{proof}
  Let $i\in \{1,\ldots,n\}$. We may assume $p_i>0$. By the hypothesis 
  of Proposition~\ref{prop:transfiniteapplication}, the
  coordinates of $\bfz_i$ all lie in the
  same connected component $V_i$ of
  $(f^{(s)})^{-1}(B_{r}(\widehat w))$ for some $\widehat w\in J_f$
  (which is allowed to depend on $i$).
  Recall that $r \le \epsilon / 512$. 
  This time we will apply Proposition~\ref{prop:vandermondepreboundhyperbolic} to
  $\widehat w=\widehat w,w_0=\widehat w, R=\epsilon/2,$ and $\epsilon=\epsilon$.

  As $\epsilon \le 1$ and $\widehat w \in J_f$ we find that
  $\overline{B_{\epsilon/2}(\widehat w)}$ is contained in
  the compact set $\{z\in \IC : \mathrm{dist}(z,J_f)\le 1/2\}$.
  In particular, $\lambda_f$ restricted to
  $\overline{B_{\epsilon/2}(w_0)}$ is bounded from above by
  $c_4(f)/2>0$ (which depends only on $f$).
  Proposition~\ref{prop:vandermondepreboundhyperbolic} implies
  (\ref{eq:DfNige1}).
\end{proof}
We bound each term on the right of (\ref{eq:bakerchangeofcoordinates})
from above separately. For $i=0$ we use Lemma \ref{lem:DfNieq0}. For
$i>0$ we refer to Lemma \ref{lem:DfNige1}. As $p_0+p_1+\cdots+p_n=1$
we can simplify the sum over the error terms and obtain
\begin{equation}
  \label{eq:cDfNup}
    \mathsf{d}_N(K) \le - \left(c_3(f,\delta)\frac{p_0^2}{d^s}
  +\sum_{i=1}^n \left(p_i^2 -c_4(f)\frac{p_i}{d^s}\right)\right) +O_f\left(\frac{\log
      N}{N}\right). 
\end{equation}

We apply the inequality in Lemma~\ref{lem:qflowerbound} to $\alpha =
c_3(f,\delta),\beta = c_4(f),$ and $X = d^s$. The condition
 $d^s\ge (\alpha+2\beta+\frac{\beta^2}{2\alpha})n$
is met, if $c_2(f,\delta) \le 1$ from
the hypothesis is small enough in terms of $f$ and $\delta$. Thus by
(\ref{eq:cDfNup}) we find
\begin{equation*}
  \mathsf{d}_\infty(K) \le -\frac{c_1(f,\delta)}{d^s}  + O_f\left(\frac{\log N}{N}\right)
\end{equation*}
with $c_1(f,\delta)=c_3(f,\delta)/2$. This
completes the proof of (\ref{eq:corcDNub}) and thus of Proposition
\ref{prop:transfiniteapplication}.


\section{The Hubbard Tree and its Core Entropy}
\label{sec:trees}

The Hubbard tree encapsulates combinatorial data attached to a
postcritically finite polynomial.
We will review properties of  the Hubbard tree and its core entropy.
Our main reference for the Hubbard tree is the work of Douady and
Hubbard~\cite{DH:etudedynI}, see Chapter IV.3~\textit{loc.cit.} and also 
Poirier's work~\cite{poirier:hubbardtrees}.
The Hubbard tree plays an important role in our construction of a
finite topological tree mentioned in the introduction.

The main results of this section
are Corollary~\ref{cor:convexhull} and Proposition \ref{prop:rhoflessthand}. In
the latter we give a self-contained characterization of postcritically
finite polynomials with maximal core entropy
based on an idea of Alsed\`a--Fagella~\cite{AlsedaFagella}.

\subsection{Graphs and Trees}

For us a \textit{graph} is a pair
$G=(V,E,\ep)$ where $V$ is a set, $E$ is a set, and $\ep$ is an injective map
from $E$ to the set of subsets of $V$ consisting of precisely two
elements. An  element of $V$ is called a \textit{vertex} of $G$, an
element of $E$ is call an \textit{edge} of $G$, and $\ep$ maps each edge
to its two endpoints. We call $G$ finite if $V$ and $E$ are finite.
Our graphs are undirected and simple. 

A \textit{path} in $G$ is a sequence of vertices $(v_0,\ldots,v_n)$,
with $n\in\IN_0$, such that for each $i\in \{1,\ldots,n\}$ there is an
edge $e_i$ with $\ep(e_i) = \{v_{i-1},v_i\}$. We call $n$ the
\textit{length} of the path. A path is called \textit{simple} if no
vertex appears twice. A non-empty graph is \textit{connected} if
any two vertices arise as vertices of a path. The \textit{diameter}
$\mathrm{diam}(G)$ of a finite connected graph $G$ is the maximal
length of a simple path in $G$. A \textit{cycle} is a path of length
$n\ge 1$ with pairwise distinct edges and $v_0=v_n$. A \textit{tree}
is a connected graph without cycles. By \textit{finite topological tree}
we mean a finite tree embedded in $\IC$. 

\subsection{The Hubbard Tree}
\label{sec:hubbard}

Let $f\in \IC[T]$ have degree $d \ge 2$. We assume that $f$ is
postcritically finite, \textit{i.e.}, each critical point of $f$ is
preperiodic. Here we survey the construction of the Hubbard tree of
$f$, its basic properties, and provide some examples.

By Lemmas~\ref{lem:fatoucomponents} and \ref{lem:postcriticalprops} we
have the following properties. The Julia set $J_f$ and the filled
Julia set $K_f$ are both compact, locally connected,
arcwise connected. Moreover, $\IC\ssm K_f$ is connected. The set
$K_f\ssm J_f$ is open in $\IC$ as $J_f$ is the boundary of $K_f$. Let
$\pi_0(K_f\ssm J_f)$ denote the (possibly empty) set of connected
components of $K_f\ssm J_f$. These connected components are also
called \emph{bounded Fatou components} of $f$. For all bounded Fatou
components $U$ of $f$, the image $f(U)$ is again a bounded Fatou
component of $f$. Finally, all bounded Fatou components of $f$ are
simply connected.

We summarize some consequences of Proposition
IV.2.2~\cite{DH:etudedynI}. For each $U\in \pi_0(K_f\ssm J_f)$ we can
find a biholomorphic map $\varphi_U \colon U\rightarrow \Delta$; here
$\Delta = \{z\in\IC : |z|<1\}$. The collection
of $\varphi_U$ can be chosen in a compatible way. More precisely, for
each $U\in\pi_0(K_f\ssm J_f)$ there exists $d(U)\in \IN$ such that the
diagram
\begin{equation}
  \label{dia:varphiU}
  \xymatrix{
    U\ar[d]_{\varphi_U}\ar[r]^{f|_{U}} & f(U)\ar[d]^{\varphi_{f(U)}}
    \\
    \Delta\ar[r]_{z\mapsto z^{d(U)}} & \Delta
  }
\end{equation}
commutes. 

The collection of all maps $\varphi_U$ is unique up-to finite amount
of ambiguity; see Section IV.2~\cite{DH:etudedynI}. Roughly speaking,
two choices differ by a factor that is on the unit circle. In
particular, the \textit{center} $\mathfrak{c}_U=\varphi_U^{-1}(0)\in
U$ of $U$ is intrinsically defined and satisfies
$f(\mathfrak{c}_U) = \mathfrak{c}_{f(U)}$.
The set of centers of $f$ is $\{\mathfrak{c}_U : U\in \pi_0(K_f\ssm
J_f)\}$.
By (\ref{dia:varphiU}) the set of centers of $f$ is fully invariant.  

By a suitable version of Carath\'eodory's Theorem, see Th\'eor\`eme
II.3.1~\cite{DH:etudedynI},  the 
biholomorphic map $\varphi_U \colon U\rightarrow \Delta$ extends to a
homeomorphism $\overline U \rightarrow\overline{\Delta}$.
See also the proof of Proposition~II.4.3a)~\cite{DH:etudedynI}
and Douady's work~\cite{Douady:compact}.
We also denote this extension by $\varphi_U$.

Before defining the Hubbard tree we need to discuss allowable arcs, see
Chapter II.6~\cite{DH:etudedynI}. 
A \textit{ray} in $\overline\Delta$ is a set  $\{t\zeta : t\in [0,1]\}$ where
$|\zeta|=1$. An \textit{arc} in $\IC$ is the image of an injective continuous map
$[0,1]\rightarrow\IC$. An \textit{allowable arc} for $f$ is an arc
$\Gamma\subset K_f$ such that $\varphi_U(\Gamma\cap \overline
U)\subset\overline\Delta$ is contained in the union of two rays for
all $U\in\pi_0(K_f\ssm J_f)$. Any two distinct points $x,y\in K_f$ are
endpoints of a unique allowable arc $[x,y]_f$ by Proposition
II.6.6~\cite{DH:etudedynI}. It is convenient to define $[x,x]_f = \{x\}$
for $x\in K_f$

The \emph{Hubbard tree} $H_f$ of $f$ is the finite union of allowable
arcs $[z',z'']_f$ where $z'$ and $z''$ range over the finitely many
members of $\pco_{\ge 0}(f) = \{f^{(n)}(z) : n\in\IN_0 \text{ and }
f'(z)=0\}$. The Hubbard tree $H_f$ is a finite topological tree. Let
$z\in H_f$. We call $z$ a branching point of $H_f$ if $H_f\ssm\{z\}$
has more than two connected components. Let $V_f$ denote union of
$\pco_{\ge 0}(f)$ and the set of branching points of $H_f$. Then $V_f$
is a finite subset of $H_f$ and its elements are called vertices of
$H_f$.

The complement $H_f\ssm V_f$ is homeomorphic to a finite disjoint
union of copies of $(0,1)$. For each connected component $C$ of
$H_f\ssm V_F$ we fix a homeomorphism $e\colon (0,1)\rightarrow C$. The
map $e$ extends to a continuous injective map $[0,1]\rightarrow H_f$ for which
we also use $e$. Note that $e(0),e(1)\in V_f$ and these points
are distinct since $H_f$ is a tree. We call $e$ an edge of $H_f$. Let
$E_f$ denote the finite set of maps $e$ thus obtained.\footnote{The
  maps $e$ are not uniquely determined, but their images $e([0,1])$ are.}

Let $U\in \pi_0(K_f\ssm J_f)$. By Sullivan's Nonwandering Theorem, the
forward orbit $U,f(U),f^{(2)}(U),\ldots$ is finite, see
Lemma~\ref{lem:fatoucomponents}(iv). So there exists $k\in\IN_0$ such
that $f^{(k)}(U)$ is periodic under $f$ of period length
$\mathrm{per}(U)\in\IN$.

\begin{lemma}
  \label{lem:centerprops}
  Let $f\in\IC[T]$ be of degree at least $2$. Suppose that $f$ is
  postcritically finite. Suppose $U\in \pi_0(K_f\ssm J_f)$ is periodic.
  The center $\mathfrak{c}_U$ of $U$ is a periodic point, and it is the
  only periodic point contained in $U$. Moreover, the forward orbit of
  $\mathfrak{c}_U$ contains a critical point and $\mathfrak{c}_U \in
  \pco_{\ge 1}(f)$.
\end{lemma}
\begin{proof}
  We recall that $f^{(k)}(\mathfrak{c}_U)$ is the center of
  $f^{(k)}(U)$ for all $k\ge 0$. Therefore, $\mathfrak{c}_U$ is fixed by
  $f^{(\mathrm{per}(U))}$ as, by hypothesis, $f^{(\mathrm{per}(U))}$
  maps $U$ to itself. In particular, $\mathfrak{c}_U$ is $f$-periodic.
  Conversely, by iterating (\ref{dia:varphiU}) we
  find that $\varphi_{f^{(k)}(U)}\circ f^{(k)}(z) = \varphi_U(z)^{D(k)}$ for
  all $z\in U$ and all $k\ge 0$ where $D(k) = d(U)\cdots
  d(f^{(k-1)}(U))\in\IN$. The period of any periodic point in $U$ is a
  multiple of $\mathrm{per}(U)$. Thus $\mathfrak{c}_U$ is the only
  periodic point in $U$. As $\mathrm{per}(U)\ge 1$, the restriction of
  $f^{(\mathrm{per}(U))}$ to $U$ cannot be the identity map. So
  $D(\mathrm{per}(U))\ge 2$ and thus $\mathfrak{c}_U$ is a critical point of
  $f^{(\mathrm{per}(U))}$. By the chain rule, $f^{(k)}(\mathfrak{c}_U)$
  is a critical point of $f$ for some $k\in
  \{0,\ldots,\mathrm{per}(U)-1\}$. Thus
  $\mathfrak{c}_U=f^{(\mathrm{per}(U)-k)}(f^{(k)}(\mathfrak{c}_U)) \in
  \pco_{\ge 1}(f)$.
\end{proof}

On the other hand,
an element in $\pco_{\ge 1}(f)$ need not be a center.
This is the case if $K_f=J_f$ and there are  no centers,
for example for $f=T^2-2$.

In the next lemma we collect some facts about the Hubbard
tree due to Douady and Hubbard.

\begin{lemma}
  \label{lem:hubbardbasic}
  Let $f\in\IC[T]$ be of degree at least $2$. Suppose that $f$ is
  postcritically finite.
  \begin{enumerate}
  \item [(i)] The Hubbard tree $H_f$ is a finite union of allowable
    arcs $\bigcup_i [x_i,y_i]_f$ where $x_i,y_i\in \pco_{\ge 0}(f)$ and where $[x_i,y_i]_f$
    contains no critical points of $f$ expect possibly $x_i$ or
    $y_i$. 
  \item[(ii)] If $x,y\in K_f$ and $[x,y]_f$ contains no critical points
    of $f$ expect possible $x$ and $y$, then
    $f([x,y]_f)=[f(x),f(y)]_f$
    and in particular, $f([x,y]_f)$ is an allowable arc.
  \item [(iii)] We have $f(H_f)\subset H_f$.
  \item[(iv)] We have $f(V_f) \subset V_f$.
  \item[(v)] The set of centers of $f$ equals  the set of
    preperiodic points in $K_f\ssm J_f$. 
  \end{enumerate}
\end{lemma}
\begin{proof}
  Part (i) follows from the definition of the Hubbard and from the
  Section~2~\cite{poirier:hubbardtrees}.

  Regarding part (ii), 
  Lemme~IV.3.1~\cite{DH:etudedynI} implies that the restriction
  $f|_{[x,y]_f}$ is injective and $f([x,y]_f)$ is
  an allowable arc. The latter must equal $[f(x),f(y)]_f$ because
  there is only one allowable arc connecting $f(x)$ and $f(y)$. 
  
  Part (iii) follows from Proposition~IV.3.3~\cite{DH:etudedynI}.

  For part (iv) follows from the more general
  Lemma~2.13~\cite{poirier:hubbardtrees} when $M=\mathrm{PCO}_{\ge
    0}(f)$.
  
  We come to part (v). Let $z\in K_f\ssm J_f$ be preperiodic.
  Some forwards orbit of $z$ is periodic and lies in a periodic element of
  $\pi_0(K_f\ssm J_f)$.  
  Lemma~\ref{lem:centerprops} implies that the
  forward orbit of $z$ contains a center of $f$. 
  But the preimage of a center is a center. Therefore, $z$ is a center. 
  Conversely,
  suppose $z$ is the center of $U\in\pi_0(K_f\ssm J_f)$, say.
  The forward orbit of $U$ under $f$ is
  finite. Moreover, any iterate of $z$ is again a center. Therefore,
  the forward orbit of $z$ is finite and $z$ is preperiodic. 
\end{proof}

\begin{example}
  We compute the Hubbard tree when $f\in\IR[T]$ is postcritically
  finite of degree $d\ge 2$ such that all critical points are real.
  In this case $\pco_{\ge 0}(f)\subset \IR$.

  Any arc in $\IC$ that is stable under complex
  conjugation and contains a real end point, is contained in $\IR$.
  From this observation and since $\pco_{f\ge 0}(f)$ is stable under
  complex conjugation,  one may show that the Hubbard tree $H_f$ lies
  in $\IR$.  Moreover, $H_f$ has
  no branching points.  Its set of vertices is 
  $\pco_{\ge 0}(f) = \{c_1,c_2,\ldots,c_n\}$ with $c_1<c_2<\cdots <
  c_n$. The  edges are just the standard intervals
  $[c_1,c_2],\ldots,[c_{n-1},c_n]$.

  \label{ex:hubbardtrees}
  \begin{enumerate} 
  \item[(i)] Let $f = T^2$, then $H_f = \{0\}$. There is a single vertex
    $V_f = \{0\}$ and no edges.
  \item[(ii)] Let $f = T^2-1$, then $\pco_{\ge 0}(f) = \{-1,0\}$.
    Here there are two
    vertices $V_f = \{-1,0\}$, both centers, and
    a single edge $H_f=[-1,0]$.
  \item[(iii)] Let $f = T^2-2$, then $\pco_{\ge 0}(f) = \{-2,0,2\}$.
    Here $V_f =
    \{-2,0,2\}$ and the Hubbard tree has
    two edges  $[-2,0]$ and $[0,2]$. In this case $J_f=K_f=[-2,2]$ and
    the set of centers is empty. 
  \item[(iv)] Let $f = T^2+c$ with $c$ the unique real number such that
    $c^3+2c^2+c+1=0$, so $c = -1.75487\ldots$. Then $\pco_{\ge 0}(f) =
    \{0,c,c^2+c\}$ as $f(f(c))=0$. Here $c^2+c = 1.32471\ldots$.
    So $V_f = \{0,c,c^2+c\}$ and $E_f =
    \{[c,0],[0,c^2+c]\}$.
  \end{enumerate}
\end{example}

The next compactness argument is useful for later reference. 

\begin{lemma}
  \label{lem:distarclemma}
  Let $f\in\IC[T]$ be of degree at least $2$. Suppose that $f$ is
  postcritically finite. For all $\epsilon > 0$ there exists
  $\delta=\delta(f,\epsilon)>0$ with the following property. Suppose
  $z\in K_f$ with $\mathrm{dist}(z,J_f)\ge \epsilon$. Let $U$ be the
  connected component of $K_f\ssm J_f$ that contains $z$. Then
  $\mathrm{dist}(z',J_f) \ge\delta$ for all $z' \in
  [z,\mathfrak{c}_U]_f$.
\end{lemma}
\begin{proof}
  We prove the lemma
  by contradiction. We assume that there is
  $\epsilon > 0$ and sequences $(z_n)_{n\in\IN},(z'_n)_{n\in\IN}$ of
  complex numbers with
  the following property.
  Each $z_n$ lies in a connected component $U_n$
  of $K_F\ssm J_f$ and satisfies $\mathrm{dist}(z_n,J_f)\ge\epsilon$.
  Each $z'_n$ lies in the allowable arc
  $[z_n,\mathfrak{c}_{U_n}]_f\subset {U_n}$ and satisfies
  $\mathrm{dist}(z'_n,J_f)<1/n$.

  The compact set $\{z\in K_f
  :\mathrm{dist}(z,J_f)\ge\epsilon\}$ is covered by finitely many
  connected components of $K_f\ssm J_f$. So after passing to a
  subsequence, there exists a connected component $U$ of $K_f\ssm J_f$
  such that $z_n\in U_n=U$ for all $n\in\IN$.
  
  By the definition of allowable arcs, the image
  $\varphi_{U}([z_n,\mathfrak{c}_{U}]_f)$ is a line segment in
  $\Delta$ through $0$ and $\varphi_{U}(z_n)$. As $z'_n$ lies on
  $[z_n,\mathfrak{c}_U]$ we have 
  $\varphi_U(z'_n) = \lambda_n \varphi_{U}(z_n)$ for some
  $\lambda_n\in [0,1]$.
  
  Recall that all $z_n,z'_n$ lie in the bounded set ${U}$. We again pass to a
  subsequence and assume that $\lim_{n\rightarrow \infty} z_n=z$
  and  $\lim_{n\rightarrow \infty} z'_n=z'$ exist, they both lie in
  $\overline {U}$.
  Similarly, we may also assume
  that $\lambda_n$ converges to $\lambda\in [0,1]$.
  Recall that $\varphi_{U}$
  extends  to a continuous map  $\overline
  {U}\rightarrow\overline{\Delta}$.
  Taking the limit
  of $\varphi_{U}(z_n')=\lambda_n \varphi_{U}(z_n)$ leads to
  $\varphi_{U}(z')=\lambda \varphi_{U}(z)$. Moreover,
  $\mathrm{dist}(z'_n,J_f)<1/n$ for all $n$ implies $z'\in J_f$. Thus
  $|\lambda \varphi_{U}(z)|=1$ and finally $\lambda=|\varphi_U(z)|=1$.
  But then $z_n$ converges to  $z\in\partial {U}\subset J_f$. This
  contradicts  $\mathrm{dist}(z_n,J_f)\ge\epsilon$. 
\end{proof}

Being hyperbolic is a purely analytic property of a complex
polynomial.
On the other hand, being postcritically finite is expressible in
algebraic terms, it is preserved under field automorphisms. 
We conclude this section with the observation that set of
postcritically finite hyperbolic polynomials is stable under
field automorphisms.
We thank Laura DeMarco for pointing out the following lemma and
providing us with a proof.

For a subfield $F\subset\IC$, we let $\mathrm{Aut}(F)$ denote the set
of field automorphisms $F\rightarrow F$. 
For $\sigma\in \mathrm{Aut}(F)$ and
 $f\in F[T]$, we let $\sigma(f)$ denote the polynomial in $F[T]$
obtained by applying $\sigma$ to the coefficients of $f$.

\begin{lemma}
  \label{lem:pcfhyperbolicgalois}
  Let $f\in F[T]$ be of degree at least $2$. Suppose that $f$ is
  postcritically finite and hyperbolic.
  Then $\sigma(f)$ is postcritically finite and hyperbolic for all
  $\sigma\in \mathrm{Aut}(F)$. 
\end{lemma}
\begin{proof}
  Let $f$ be postcritically finite and hyperbolic. Let $z\in \pco_{\ge
    0}(f)$ be arbitrary. By a characterization of being hyperbolic,
  Lemma~\ref{lem:hyperboliccharacterization}, the forward orbit of $z$
  does not meet the Julia set of $f$. This forward orbit is finite by
  hypothesis. So it contains a periodic point $w$, say. We have $w\in
  K_f\ssm J_f$. The connected component $U\in \pi_0(K_f\ssm J_f)$
  containing $w$ is periodic. Lemma~\ref{lem:centerprops} implies that
  $w$ is the center of $U$. Moreover, the forward orbit of $w$ contains
  a critical point of $f$.
  
  We have proved that the forward orbit of any point in $\pco_{\ge
    0}(f)$ contains a periodic critical point.
  In other words,  for all $z\in \pco_{\ge 0}(f)$ there exists $m_z\in \IN_0$ such that
  $f^{(m_z)}(z)$ is periodic of period $n_z\ge 1$, say, and critical.
  Thus $f^{(m_z+n_z)}(z)=f^{(m_z)}(z)$ and $f'(f^{(m_z)}(z))=0$. 

  Let $\sigma\in \mathrm{Aut}(F)$ and let $z$ be in $\pco_{\ge
    0}(\sigma(f))$.  Applying $\sigma$ commutes with differentiation.
  So $\sigma^{-1}(z)$ lies in $\mathrm{PCO}_{\ge 0}(f)$.  The
  argument above implies
  $f^{(m+n)}(\sigma^{-1}(z)) =
  f^{(m)}(\sigma^{-1}(z))$ and
  $f'(f^{(m)}(\sigma^{-1}(z)))=0$
  with $m=m_{\sigma^{-1}(z)}\ge 0$ and $n=n_{\sigma^{-1}(z)}>0$.
  Now we apply $\sigma$ to find
  $$
  \sigma(f)^{(m+n)}(z)= \sigma(f)^{(m)}(z) \quad\text{and}\quad
  \sigma(f)'(\sigma(f)^{(m)}(z))=0.
  $$
  In other words,  $z$ is a preperiodic point of $\sigma(f)$ and its forward orbit
  under $\sigma(f)$ contains a critical point of $\sigma(f)$.
  Recall that  $z$ was an arbitrary  point of $\pco_{\ge
    0}(\sigma(f))$.
  So by Lemma~\ref{lem:critperimplieshyp}(i) we
  conclude that $\sigma(f)$ is postcritically finite and hyperbolic. 
\end{proof}

\subsection{The Core Entropy of $f$}
\label{sec:entropy}

In this section we survey Thurston's notion of \textit{core
entropy}~\cite{thurstonetal:degreed} of a postcritically finite
polynomial.

Let $f\in \IC[T]$ be of  degree $d \ge 2$ and
postcritically finite. 
Let $H_f,V_f,$ and $E_f$ be as in
Section~\ref{sec:hubbard}. We use $\IZ[E_f]$ to denote the free abelian group generated by the finite set
of edges $E_f$ of $H_f$. 

Let us verify that $f$ induces a \textit{Markov map} on the Hubbard
tree $H_f$, see Appendix~A~\cite{thurstonetal:degreed}. This property
is stated in Lemma 13.7~\textit{loc.cit.}

Suppose $e\in E_f$, considered as usual as an injective continuous map
$e\colon [0,1]\rightarrow H_f$.
We recall
Lemma~\ref{lem:hubbardbasic}.
By part (iii) $f\circ e$ takes values in $H_f$,
by part (iv) we have $f(e(0)),f(e(1)) \in V_f$, and
by parts (i) and (ii)
$f\circ e$ is a path in $H_f$. 

Let us give some details on the final statement which is important for
the things to come.
We claim that $f\circ e$ represents a simple path in the tree $H_f$.
Indeed, $e((0,1))$ does not meet $V_f$ by construction. In particular,
$e((0,1))$  does
not meet any critical point of $f$. So $f$ is locally injective at any
point of $e((0,1))$. Therefore, as $f\circ e$ moves along a fixed edge
of $H_f$, the edge coordinate is strictly monotone. Moreover, $f\circ e$
traverses edges without backtracking. As $H_f$ is a tree, no vertex in
$V_f$ is visited more than once. This implies our claim.

We define $L(e) = e_1+\cdots + e_m$
where $e_1,\ldots,e_m$ are the pairwise distinct edges of the Hubbard
tree 
traversed by
$f\circ e$.
A non-zero element of $\IZ[E_f]$ is called
\textit{reduced} if it is the sum of pairwise different basis elements
in $E_f$.
In particular, $L(e)$ is reduced.
We extend $L$ by linearity and obtain a group homomorphism
\begin{equation}
  \label{def:L}
  L \colon \IZ[E_f] \rightarrow \IZ[E_f].
\end{equation}

If $E_f\not=\emptyset$ then $L$ is not nilpotent. Indeed, for any
$e\in E_f$ the image $L(e)$  is non-zero and
lies in the cone $\IZ[E_f]_{\ge 0}$, the set of all possible
$\sum_{e\in E_f} n(e)e$ with $n(e)\ge 0$. 
The same is true for $L^k(e)$ for all $k\ge 1$. Hence $L^k$ is non-zero for
all $k\ge 1$.

We write $L\otimes \IC$ for  endomorphism of the
$\IC$-vector space $\IC[E_f]$  induced by $L$.
By the Cayley--Hamilton Theorem, the characteristic
polynomial of $L\otimes\IC$ is not a power of $T$. The characteristic
polynomial lies in $\IZ[T]$ and is monic. Hence $L\otimes\IC$
has a complex eigenvalue
of modulus at least $1$.

Let $A$ be an endomorphism of a finite dimensional $\IC$-vector space
$V$. 
If $V\not=0$ we let
$\rho(A)$ denote the spectral radius of $A$, \textit{i.e.}, $\rho(A)$
is the largest absolute value of an eigenvalue of $A$. If $V=0$, we
set $\rho(A)=1$.

\begin{defi}
  \label{def:entropy}
  We keep the notation above.  The exponential core entropy of $f$ is
  $\rho(f)=\rho(L\otimes\IC)\ge 1$. 
  Moreover, $\log \rho(f)\ge 0$ is called the core entropy of $f$.
\end{defi}

By Proposition~A.6~\cite{thurstonetal:degreed} the core entropy of $f$
agrees with the  notation given in Definition 13.8.~\textit{loc.cit.}

Let $A=(a_{ij})_{i,j}\in \mat_{m,n}(\IC)$ with $m,n\in\IN$. The
row-sum norm of $A$ is $\|A\|_\infty = \max_{1\le i \le m}
\sum_{j=1}^n |a_{ij}|$. If $v\in\IC^m$ is a column vector,
\textit{i.e.}, an element of $\mat_{m,1}(\IC)$, then $\|v\|_\infty$ is
the maximum-norm on the coordinates. Moreover, $\|Av\|_\infty \le
\|A\|_\infty \|v\|_\infty$ by a simple application of the triangle
inequality. Suppose $m=n$. Then we consider $A$ as an endomorphism of
$\IC^m$. It is well-known that the spectral radius satisfies
\begin{equation}
  \label{eq:specradbound}
  \rho(A) \le \|A^s\|_\infty^{1/s}
\end{equation}
for all $s\in\IN$. 

\begin{lemma}
  \label{lem:entropybound}
  We have $1\le \rho(f) \le d=\deg f$. 
\end{lemma}
\begin{proof}  
  The inequality $\rho(f)\ge 1$ follows from definition.
  We may assume $E_f\not=\emptyset$. 
  Let $e_1,\ldots,e_m$ be the edges in $E_f$.
  There exist $a_{ij}\in \IZ$ with
  $L(e_j) = \sum_{i=1}^m a_{ij} e_i\in\IZ[E_f]$ for all $j$. 
  So $A= (a_{ij})_{i,j}\in \mat_{m}(\IZ)$ represents $L$ with respect
  to
  the edge basis.  The entries of $A$ are zeros or ones since $L(e_j)$
  is reduced. Moreover, each row of
  $A$ contains at most $d$ ones; indeed, any edge of $E_f$ has at most
  $d$ preimages under $f$.
  This implies $\|A\|_\infty\le d$.
  So $\rho(f) \le d$ by (\ref{eq:specradbound}) in the case $s=1$.
\end{proof}

Let us compute the core entropy of each polynomial that appears in
Example \ref{ex:hubbardtrees}.

\begin{example}
  \label{ex:spectralradius}
  \begin{enumerate}
  \item[(i)] The Hubbard tree of $f = T^2$ has no edges, thus $\rho(f)=1$
    by definition. 
  \item[(ii)] If $f = T^2-1$, then $f$ restricts to a bijection of
    the only edge $[-1,0]$ onto itself. Again we find $\rho(f)=1$.  In other words, the core
    entropy of $f$ is $0$. 
  \item[(iii)] If $f = T^2-2$, then $f([-2,0]) = f([0,-2]) =
    [-2,0]\cup [0,2]$. With respect to the basis $([-2,0],[0,2])$ the
    linear map $L$ is represented by
    \begin{equation*}
      \left(
        \begin{array}{cc}
          1 & 1 \\ 1 & 1
        \end{array}
      \right).
    \end{equation*}
    Thus $\rho(f) = 2$ attains the largest possible value in degree $2$.
  \item[(iv)] Let $f = T^2+c$ with $c$ as in
    Example~\ref{ex:hubbardtrees}(iv).
    Since $f([c,0]) = [c,c^2+c] = [c,0]\cup [0,c^2+c]$ and
    $f([0,c^2+c]) = [c,0]$ the linear map $L$ is represented by
    \begin{equation*}
      \left(
        \begin{array}{cc}
          1 & 1 \\
          1 & 0
        \end{array}
      \right)
    \end{equation*}
    with respect to $([c,0],[0,c^2+c])$. Therefore, $\rho(f) = (1+\sqrt 5)/2$. 
  \end{enumerate}
\end{example}

\subsection{Convex Hull in a Tree}
\label{sec:growth}

Let $f\in \IC[T]$ be of  degree $d \ge 2$ and
postcritically finite. Let $H_f,V_f,$ and $E_f$ be as in
Section~\ref{sec:hubbard}.

For any $s\in\IN_0$ we define
\begin{equation}
  \label{def:hubbardpreimage}
 H_f^{(s)} = \bigl(f^{(s)}\bigr)^{-1}(H_f) \subset\IC.
\end{equation}
In particular, $H_f^{(0)} = H_f$.

We will endow $H_f^{(s)}$ with the structure of
a finite topological tree. By
the chain rule, a critical value of $f^{(s)}$ lies in $\pco_{\ge
1}(f)\subset H_f$. Thus $H_f^{(s)}$ is path-connected as $H_f$
contains all critical values of $f^{(s)}$. 

Let us begin by constructing the vertices of $H_f^{(s)}$. 
We define
\begin{equation}
  \label{def:Vfs}
  V_f^{(s)} = \bigl(f^{(s)}\bigr)^{-1}(V_f) 
\end{equation}
for all $s\in\IN_0$ and observe $V_f^{(s)} \subset H_f^{(s)}$. 

Suppose $v\in V_f^{(s)}$ with $s\in\IN_0$. Then $f^{(s+1)}(v)
=f(f^{(s)}(v)) \in f(V_f)$ by definition. Lemma~\ref{lem:hubbardbasic}(iv) implies
$f^{(s+1)}(v)\in V_f$ and hence $v \in V_f^{(s+1)}$. We thus have 
\begin{equation}
  \label{eq:Vfchain}
  V_f = V_f^{(0)} \subset V_f^{(1)}\subset V_f^{(2)}\subset \cdots. 
\end{equation}

Let $e \in E_f$ be an edge. Then $e((0,1)) \cap V_f=\emptyset$. Recall
that $f^{(s)}(\crit{f^{(s)}})\subset \pco_{\ge 1}(f) \subset V_f$. In
particular, $e((0,1))$ does not meet any critical values of $f^{(s)}$.
So $e(1/2)$ has precisely $d^s$ different preimages under $f^{(s)}$.
Each preimage induces a lift of 
$e|_{(0,1)}$  to a continuous map
$\tilde e \colon (0,1)\rightarrow H_f^{(s)}$ with $f^{(s)} \circ
\tilde e = e|_{(0,1)}$. Any such lift extends to a continuous and
injective map on $[0,1]$. Observe that $\tilde e(0),\tilde e(1) \in
V_f^{(s)}$. We define $E_f^{(s)}$ to be the set of maps thus obtained.

We have $E_f^{(0)}=E_f$. Note that $\# E_f^{(s)} = d^s \# E_f$. The
map $f^{(s)}\colon \IC\rightarrow\IC$ is unramified above $\IC\ssm
V_f\subset \IC\ssm \pco_{\ge 1}(f)$. The Riemann--Hurwitz formula
implies $\# V_f^{(s)} = \# E_f^{(s)}+1$. As $H_f^{(s)}$ is connected,
it
follows that $H_f^{(s)}$
is a finite topological tree with vertices $V_f^{(s)}$ and edges the
images of elements of $E_f^{(s)}$. Of course, we recover $H_f,V_f,$
and $E_f$ if $s=0$.

The \textit{convex hull} of a set $\cV$ of vertices of $H_f^{(s)}$ is the
subtree generated by $\cV$. In other words, it is the intersection of all subtrees of
$H_f^{(s)}$ whose vertex set contains $\cV$. The convex hull of $\cV$
is the smallest subtree of
$H_f^{(s)}$ containing $\cV$.

Since $H_f^{(s)}$ has $d^s \# E_f $ edges, the convex hull of $\cV$
has at most $ d^s \# E_f$ edges. We consider this as the trivial bound. 

Our aim is to improve on this trivial bound when $\rho(f)<d$, compare
with Lemma~\ref{lem:entropybound}. We now state our main result
regarding the size of the convex hull of a collection of points in
$H_f^{(s)}$. The constant implicit in $O_f(\cdots)$ below may depend on $f$ but not
on $s$ or $\cV$.

\begin{propo}
  \label{prop:treegrowthestimate}
  We keep the notation from above.
  Let $\lambda\ge 1$ be  the size of the largest Jordan block
  among all eigenvalues of $L\otimes\IC$ of maximal modulus. 
  Let $s\in\IN$ and  $\cV$ be a subset of $V_f^{(s)}$. Then for all
  $t\in \{1,\ldots,s\}$, the convex
  hull of $\cV$ in $H_f^{(s)}$ has 
  $O_f(t^{\lambda-1} \rho(f)^t (d^{s-t}+t \#\cV))$ edges.
\end{propo}

Let us focus on $f$ with non-maximal core entropy. 
We illustrate the previous proposition with a corollary 
when $\rho(f)<d$.
Recall that $H_f^{(s)}$ has $O_f(d^s)$ vertices and edges.
Being vague, we may think of a subset of  vertices as
sparse, if its cardinality is at most  $\kappa d^s$ for some fixed
$\kappa>0$ as $s$ grows.
Similarly, a subtree of $H_f^{(s)}$ is sparse if its number of edges
is at most $\kappa d^s$ as $s$ grows. 
Roughly speaking, the  corollary below states that 
the
convex hull of a sparse subset of vertices is sparse.
The condition $\rho(f)<d$ is met when $f$ is postcritically finite
and hyperbolic, see
Proposition~\ref{prop:rhoflessthand} below.

\begin{corollary}
  \label{cor:convexhull}
  We keep the notation from above and suppose $\rho(f)<d$. For all
  $\nu > 0$ there exists $\kappa(f,\nu)>0$ which depends on $f$ and $\nu$ and
  has the following properties. Let $s\in\IN_0$ and let $\cV$ be a subset
  of $V_f^{(s)}$ with $\# \cV \le \kappa(f,\nu) d^s$. Then the convex hull of
  $\cV$ in $H_f^{(s)}$ has at most $\nu d^s$ edges.
\end{corollary}
\begin{proof}
  It suffices to handle the case $\cV\not=\emptyset$ and $s\ge 1$.
  Given $\nu>0$ we will see how to fix $\kappa(f,\nu)$.
  For brevity, we will write $\kappa=\kappa(f,\nu)$.
  By Proposition~\ref{prop:treegrowthestimate} there exists $c=c(f)>0$ such
  that
  the convex hull of $\cV$ in $H_f^{(s)}$ has at most
  \begin{equation*}
    ct^{\lambda-1}\left(
    \left(\frac{\rho(f)}{d}\right)^t + t \rho(f)^t \frac{\#\cV}{d^s} 
  \right)d^s
  \le \left(
    ct^{\lambda-1}\left(\frac{\rho(f)}{d}\right)^t  + ct^{\lambda}
    \rho(f)^t \kappa
  \right)d^s
  \end{equation*}
  edges for all $s\in\IN$ and all $t\in \{1,\ldots,s\}$; here $\lambda
  \ge 1$ depends only on $f$. 

  As $\rho(f)/d<1$ we max fix an integer $t\ge 1$ in terms of $c,f,$ and $\nu$  such that $ct^{\lambda-1}
  (\rho(f)/d)^t \le \nu/2$. For $\kappa>0$ sufficiently small in
  terms of $c,t,f,$ and $\nu$ we have $ct^{\lambda}\rho(f)^t
  \kappa\le\nu/2$.  We may also assume that $\kappa <d^{-t}$. 

  Thus if $s\ge t$, the number of edges in the convex hull is at most
  $\frac{\nu}{2}d^s + \frac{\nu}{2} d^s = \nu d^s$.

  But what if $s<t$? Our hypothesis
  implies $1\le \#\cV\le \kappa d^s<\kappa d^t$. This contradicts our
  choice of $\kappa$.
\end{proof}

We will prove Proposition~\ref{prop:treegrowthestimate} after some
preliminary work.

Let $s\in\IN$. If $e\in E_f^{(s)}$, then we define $f^{(s)}_*(e) =
f^{(s)} \circ e \in E_f$. This determines a map $f^{(s)}_* \colon
E_f^{(s)} \rightarrow E_f$. We let $\IZ[E_f^{(s)}]$ denote the free
abelian group generated by the finitely many elements of $E_f^{(s)}$.
We extend $f^{(s)}_*$ linearly to get  a group homomorphism
\begin{equation*}
  f^{(s)}_* \colon \IZ[E_f^{(s)}]\rightarrow \IZ[E_f],
\end{equation*}
for which we use the same symbol. It is convenient to let $f^{(0)}_*$
denote the identity map on $\IZ[E_f]$. 

A non-zero element of $\IZ[E_f^{(s)}]$ is called \textit{reduced} if
it is the sum of pairwise different basis elements in $E_f^{(s)}$.
This extends the homonymous notion when $s=0$.

Let $\IZ[E_f^{(s)}]_{\ge 0}$ denote all non-negative $\IZ$-linear
combinations of elements in $E_f^{(s)}$. The \textit{support} of
$\sum_{e\in E_f^{(s)}} n(e) e \in \IZ[E^{(s)}_f]_{\ge 0}$ is
$\bigcup_{e:n(e)\not=0} e([0,1])$.

Suppose $e\in E_f^{(s)}$ is an edge with  $s\in\IN_0$. By construction, we have
$e(0),e(1)\in V_f^{(s)}$. 
Recall $V_f^{(s)}\subset V_f^{(s+1)}$ by (\ref{eq:Vfchain}). So $e(0),e(1)\in
V_f^{(s+1)}$.
The image  $f^{(s)}(e([0,1]))$ is back in the Hubbard tree  $H_f$.
As $f(H_f)\subset H_f$ as sets, see
Lemma~\ref{lem:hubbardbasic}(iii), we find $e([0,1]) \subset
H_f^{(s+1)}$.
But $e([0,1])$ may not be the image of an edge
of $H_f^{(s+1)}$. Indeed, there may be  points of $V_f^{(s+1)}$
along $e((0,1))$.
As $H_f^{(s+1)}$ is a tree,
there is exactly one simple path in
$H_f^{(s+1)}$ connecting $e(0)$ to $e(1)$. It consists of a finite
number of pairwise distinct edges $e_1,\ldots,e_m \in E_f^{(s+1)}$,
say. We say that the path $(e_1,\ldots,e_m)$ \textit{refines} the path $(e)$ in
$H_f^{(s+1)}$. We define the refinement map
\begin{equation*}
  r_s \colon \IZ[E_f^{(s)}]\rightarrow \IZ[E_f^{(s+1)}]
\end{equation*}
by imposing $r_s(e) = e_1+\cdots + e_m$ and extending it by linearity. 
Here $r_s(e)$ is reduced as the $e_1,\ldots,e_m$ are pairwise
distinct. Finally, the support of $r_s(e)$ equals $e([0,1])$, the
support of $e$.

By a simple induction we find that 
\begin{equation}
  \label{eq:rsiterated}
  (r_{s+t} \circ \cdots \circ r_{s+1}\circ r_s)(e) \in
  \IZ[E_f^{(s+t+1)}]_{\ge 0}
\end{equation}
is reduced for all $s,t\in\IN_0$ and all $e\in E_f^{(s)}$. Moreover,
the support of the left-hand side of (\ref{eq:rsiterated}) is the
support of $e$.

It is sometimes convenient to drop the subscript and
consider $r$ as a self-map of the disjoint  union
$\dot\bigcup_{s\in\IN_0} \IZ[E_f^{(s)}]$.
Thus it makes sense to consider iterates $r^{(0)}=\mathrm{id},r^{(1)}=r,r^{(2)} =
r\circ r,\ldots$ of $r$.

\begin{lemma}
  \label{lem:Lrcommutativediag}
  We keep the notation from above. In particular,  $L$ is as in (\ref{def:L}).
  \begin{enumerate}
  \item  [(i)] For all $s\in\IN$ we have $L\circ f_*^{(s-1)} =
    f_*^{(s)} \circ r_{s-1}$ on $E_f^{(s-1)}$.
  \item  [(ii)] For all $s\in\IN_0$ and all $t\in \{0,\ldots,s  \}$ we have
    $L^t \circ f_*^{(s-t)} = f_*^{(s)} \circ r^{(t)}$ on
  $\IZ[E_f^{(s-t)}]$.  We follow here the convention $r^{(t)} = r_{s-1}\circ \cdots \circ
  r_{s-t}$ on   $\IZ[E_f^{(s-t)}]$.
  \end{enumerate}
\end{lemma}
\begin{proof}
  The second claim follows from the first claim by a simple induction on $s$.
  
  To prove the first claim, it 
  suffices to verify $L(f_*^{(s-1)}(e)) = f_*^{(s)}(r_{s-1}(e))$ for all
  edges $e\in E_f^{(s-1)}$. We evaluate both sides of the equality.

  First, recall that $f_*^{(s-1)}(e)$ is an
  edge, say  $\widetilde e$,
  of the Hubbard tree $H_f$. Therefore, $L(f_*^{(s-1)}(e))$ is 
  reduced as $L$ maps
  edges to reduced elements. The support of $L(f_*^{(s-1)}(e))$ equals
  the support of $L(\widetilde e)$, which equals 
  $f(\widetilde e([0,1]))=f^{(s)}(e([0,1]))$.

  Second,
  say $r_{s-1}(e) = e_1+\cdots+e_m$ with $e_i\in E_f^{(s)}$.
  As we have seen
  during the construction of $L$ at the beginning of this section,
  the composition $f\circ \widetilde e
  = f^{(s)}\circ e$ is a
  simple path in the tree $H_f$. Traversing this path is the same as
  traversing the $f^{(s)}_*(e_i)$ with $i\in \{1,\ldots,m\}$ in the
  appropriate
  order (at a different
  speed). So $f^{(s)}_*(e_i)$ are pairwise distinct edges of $H_f$ as
  $i\in\{1,\ldots,m\}$. Thus $f^{(s)}_*(r_{s-1}(e))$ is reduced. 
  Finally, the support of $f_*^{(s)}(r_{s-1}(e))$ also equals
  $f^{(s)}(e([0,1]))$.

  Because $L(f_*^{(s-1)}(e))$ and $f_*^{(s)}(r_{s-1}(e))$ are reduced and
  have the same support they must be equal.   
\end{proof}

Let $s\in \IN_0$. 
A path $p$ in $H_f^{(s)}$ naturally defines an element $p_*\in
\IZ[E_f^{(s)}]_{\ge 0}$: the coefficient of $p_*$ at the edge $e\in
E_f^{(s)}$ is the number of times $p$ traverses $e$. We say that
$p_*$ represents $p$. Note that we do not account for orientation.

We define a group homomorphism $\ell\colon
\IZ[E_f^{(s)}]\rightarrow\IZ$, called \textit{length}, by imposing
$\ell(e)=1$ for all $e\in E_f^{(s)}$. This homomorphism takes
non-negative values on $\IZ[E_f^{(s)}]_{\ge 0}$. The \textit{length}
of the path $p$ is $\ell(p_*)\ge 0$. If $e\in E_f^{(s)}$, then
$f^{(s)}_*(e)$ is an edge of the Hubbard tree. So
$\ell(f^{(s)}_*(e))=1=\ell(e)$. By linearity, we find
\begin{equation}
  \label{eq:ellisometryHfs}
  \ell(f_*^{(s)}(\gamma)) =\ell(\gamma)\quad\text{for all}\quad \gamma\in
  \IZ[H_f^{(s)}].
\end{equation}

\begin{lemma}
  \label{lem:connectionlemma}
  We keep the notation from above.
  Let $s\in \IN_0, v\in V_f^{(s)},$ and $w\in V_f$.
  For each $i\in \{0,\ldots,s\}$ there exists
  $\gamma_i \in \IZ[E_f^{(i)}]_{\ge 0}$ with 
  $\ell(\gamma_i)\le \mathrm{diam}(H_f)$ such that the following
  holds. For all $j\in \{0,\ldots,s\}$, the
  element $\gamma_s + r(\gamma_{s-1}) +
  r^{(2)}(\gamma_{s-2})+\cdots + r^{(j)}(\gamma_{s-j})$ represents a path
  in $H_f^{(s)}$. Moreover, $v$ lies in the
  support of $\gamma_s$ and $w$ lies in the support of
  $\gamma_s + r(\gamma_{s-1}) +
  r^{(2)}(\gamma_{s-2})+\cdots + r^{(s)}(\gamma_{0})$.
\end{lemma}
\begin{proof}
  Recall that $V_f \subset
  V_f^{(s)}$ by (\ref{eq:Vfchain}). So $v,w\in V_f^{(s)}$.

  The proof is by induction on $s$.
  
  Suppose $s=0$. We fix the unique
  simple path $p$ in the Hubbard tree $H_f$  connecting $v$ to
  $w$. Therefore, 
  $\ell(p_*)  \le \mathrm{diam}(H_f)$ and we take $\gamma_0=p_*$.

  Assume $s\ge 1$. We again fix $p$ to be the unique simple path in
  $H_f^{(s)}$ connecting $v$ to $w$. So $p_* = e_1+\cdots + e_m$ where
  $e_1,\ldots,e_m$ are pairwise distinct
  edges of $E_f^{(s)}$ traversed by $p$ in given
  order. In particular, $v$ lies in the support of $e_1$. 
  
  Suppose $m\le \mathrm{diam}(H_f)$. Then we may take $\gamma_s = p_*$ and
  $\gamma_{s-1}=0, \ldots,\gamma_0=0$. The lemma follows.
  
  So let us assume
  $m>\mathrm{diam}(H_f)$. Each of $f_*^{(s)}(e_1), \ldots,
  f_*^{(s)}(e_m)$ is an edge of the Hubbard tree. Following these
  edges
  in order
  produces a path in $H_f$ whose length is greater than the
  diameter of $H_f$. As $H_f$ is a tree, it contains no cycles.  So
  there must be backtracking, \textit{i.e.}, there is an $i\le \mathrm{diam}(H_f)$
  such that $f_*^{(s)}(e_i)=f_*^{(s)}(e_{i+1})$.

  Let  $z\in V_f^{(s)}$
  be the endpoint of $e_i$
  below which backtracking occurs in $H_f$. Then $f^{(s)}$ is not
  locally injective near $z$. So ${f^{(s)}}'(z)=0$. This implies
  $f'(z)f'(f(z))\cdots f'(f^{(s-1)}(z))=0$ by the chain rule.
  Among $z,f(z),\ldots,f^{(s-1)}(z)$ is a critical point of $f$.
  Therefore, $f^{(s-1)}(z) \in \pco_{\ge 0}(f) \subset V_f$. In particular,
  $z\in V_f^{(s-1)}$ by (\ref{def:Vfs}).  
  We set $\gamma_s=e_1+\cdots +e_i \in\IZ[E_f^{(s)}]_{\ge 0}$. 
  It represents
  the truncated path $(e_1,\ldots,e_i)$ that begins at $v$ and ends at
  $z$.
  Note that $\ell(\gamma_s) = i \le \mathrm{diam}(H_f)$.
  
  We apply induction to $z\in V_f^{(s-1)}$ and $w\in V_f$. So there  exist
  $\gamma_{s-1},\ldots,\gamma_0$ all of length at most 
  $\mathrm{diam}(H_f)$ such that 
  $\gamma_{s-1}+r(\gamma_{s-2}) + \cdots
  +r^{(j)}(\gamma_{s-1-j})\in\IZ[E_f^{(s-1)}]$
  represents a path in $H_f^{(s-1)}$ for all $j\in \{0,\ldots,s-1\}$.
  The support of $\gamma_{s-1}$ contains $z$ and for $j=s-1$ the
  support of the sum contains $w$. 

  We refine and apply $r_{s-1}$ to find that
  $r(\gamma_{s-1})+r^{(2)}(\gamma_{s-2}) + \cdots
  +r^{(s)}(\gamma_0)\in\IZ[E_f^{(s)}]$ represents a path in
  $H_f^{(s)}$ connecting $z$ and $w$. The support does not change on
  applying $r_{s-1}$. 
  Finally, adding  $\gamma_s$ concatenates the path from $v$ to $z$
  with the path from $z$ to $w$. This completes the proof.
\end{proof}

\begin{proof}[Proof of Proposition~\ref{prop:treegrowthestimate}]
  If $\cV$ is empty, then the convex hull is a tree with no zero edges
  and our claim is
  trivial. So assume $\cV\not=\emptyset$.
  Similarly, we may also assume $E_f\not=\emptyset$. 

  Let $t\in \{1,\ldots,s\}$. As a
  set, the tree $H_f^{(s-t)}$ is contained in $H_f^{(s)}$. Recall that
  an edge of $H_f^{(s-t)}$ need not be an edge of $H_f^{(s)}$.
  However, the image $r^{(t)}(E_f^{(s-t)})$ can be identified with the
  set of edges of a subtree $\cT$ of $H_f^{(s)}$.
  Each vertex in $V_f^{(s-t)}$ is a vertex of $\cT$.

  Recall that
  $\ell(f_*^{(s)}\circ r^{(t)}(e)) = \ell(r^{(t)}(e))$ by
  (\ref{eq:ellisometryHfs}). So the
  number of edges of $\cT$ equals
  \begin{equation*}
    \sum_{e\in E_f^{(s-t)}} \ell(r^{(t)}(e)) = \sum_{e\in E_f^{(s-t)}}
    \ell(f_*^{(s)}\circ r^{(t)}(e))
    =\sum_{e\in E_f^{(s-t)}}
    \ell(L^{t}\circ f_*^{(s-t)}(e))
  \end{equation*}
  where we used Lemma~\ref{lem:Lrcommutativediag}(ii).
  Recall 
  that $E_f^{(s-t)}$ has $d^{s-t} \#E_f$ edges. So the number of edges
  of $\cT$ is at most
  \begin{equation}
    \label{eq:nredgesofTpre}
    \max_{e\in E_f^{(s-t)}}\ell(L^{t}\circ f_*^{(s-t)}(e))d^{s-t} \#E_f.
  \end{equation}

  Next we will bound $\ell(L^{t}\circ f_*^{(s-t)}(e))$ from above as $t$ varies.
  We fix  a norm $\|\cdot\|$ on the
  endomorphisms of $\IZ[E_f]\otimes\IC$. 
  Consider the Jordan normal form of $L$. By definition, the largest
  absolute value of an eigenvalue is the exponential core entropy
  $\rho(f)\ge 1$. Moreover, 
  the largest Jordan block of an eigenvalue of maximal modulus is
  $\lambda \ge 1$, by hypothesis. 
  By basic linear algebra, we have 
  $\| L^t\| = O_f(t^{\lambda-1}\rho(f)^t)$ for all  $t\in\IN$.
  
  The image  $f_*^{(s-t)}(e)$ is an edge of $H_f$ for all $e\in
  E_f^{(s-t)}$.  Thus $\| f_*^{(s-t)}(e)\| = O_f(1)$, with implicit
  constant independent of $s$ and $t$.
  Then length $\ell(\cdot)$ is bounded from above linearly in terms of
  the norm $\|\cdot\|$. 
  Therefore, $\ell(L^{t}\circ f_*^{(s-t)}(e)) = O_f(t^{\lambda-1}\rho(f)^t)$.
  We recall (\ref{eq:nredgesofTpre}) and
  conclude that the number of edges of  $\cT$ is
  \begin{equation}
    \label{eq:nredgesofT}
    O_f( t^{\lambda-1} \rho(f)^{t} d^{s-t}). 
  \end{equation}

  The construction up till now did not involve the set of vertices
  $\cV$. In next step we successively connect each element in $\cV$ to
  a vertex of $\cT\subset H_f^{(s)}$ using edges from $H_f^{(s)}$.
  This will lead to a
  subtree of $H_f^{(s)}$ containing $\cT$ and $\cV$. This new subtree
  must contain the  convex hull of $\cV$.   

  To fill in the details
  we fix a base point $w\in V_f$ from $\pco_{\ge 0}(f)$. In
  particular, $w$ is a vertex of $\cT$.

  Let $v\in \cV$.
  
  By Lemma~\ref{lem:connectionlemma}, for $j=s$, there is $\gamma_s +
  r(\gamma_{s-1})+\cdots + r^{(s)}(\gamma_0)$ that represents a path $p$
  in $H_f^{(s)}$ starting at  $v$ and ending at $w$. We have 
  $\gamma_i \in \IZ[E_f^{(i)}]_{\ge 0}$ and $\ell(\gamma_i)\le
  \mathrm{diam}(H_f)$ for all $i\in \{0,\ldots,s\}$.
  The support of $\gamma_s$ contains $v$. 
  The partial sum $q_* = \gamma_s + r(\gamma_{s-1})+\cdots +
  r^{(t)}(\gamma_{s-t})$ represents
  the path $p$ truncated.
  The truncated path starts at $v$ and ends at a vertex of
   $V_f^{(s-t)}$, hence at a vertex of $\cT$.
  We insert $q_*$ into  $f_*^{(s)}$ and
  apply
  Lemma~\ref{lem:Lrcommutativediag}(ii) to find
  \begin{equation*}
    f_*^{(s)}(q_*) = \sum_{i=0}^t f_*^{(s)}(r^{(i)}(\gamma_{s-i}))=  \sum_{i=0}^{t} L^{i} \circ f_*^{(s-i)}(\gamma_{s-i}).    
  \end{equation*}
  We take  the length and use (\ref{eq:ellisometryHfs}) to conclude
  \begin{equation}
    \label{eq:ellpstareq}
    \ell(q_*) = \ell(f_*^{(s)}(q_*)) =  \sum_{i=0}^{t}
      \ell\left(L^i \circ f_*^{(s-i)}(\gamma_{s-i})\right).
  \end{equation}

  Now $\ell(f_*^{(s-i)}(\gamma_{s-i})) = \ell (\gamma_{s-i}) \le
  \mathrm{diag}(H_f) = O_f(1)$.
  So $\|f_*^{(s-i)}(\gamma_{s-i})\| = O_f(1)$. 
  As before, some basic linear algebra yields
  $\ell\left(L^i \circ f_*^{(s-i)}(\gamma_{s-i})\right) =
  O_f(\| L^i \circ f_*^{(s-i)}(\gamma_{s-i})\|)=
  O_f(\max\{1,i\}^{\lambda-1}\rho(f)^i)$.
  We  sum over $i\in \{0,\ldots,t\}$ and use (\ref{eq:ellpstareq})
  to get
  \begin{equation}
    \label{eq:connecttoT}
    \ell(q_*) = O_f(t^{\lambda} \rho(f)^{t}). 
  \end{equation}
  
  Recall that $q_*$  represents a path in
  $H_f^{(s)}$ that connects $v\in\cV$  to a vertex of the
  subtree $\cT\subset H_f^{(s)}$.

  To  attach all points
  in $\cV$ to $\cT$
  we repeat the procedure above at most $\#\cV$ times. At each iteration we
  adjoin at most
  (\ref{eq:connecttoT}) edges. 
  Not forgetting that the  number of edges of $\cT$ is bounded by  (\ref{eq:nredgesofT}),
  it follows that $\cV$ lies in a subtree of
  $H_f^{(s)}$ with at most
  \begin{equation*}
    O_f \left(t^{\lambda-1} \rho(f)^{t} d^{s-t}+ t^{\lambda} \rho(f)^t
      \#\cV \right)
  \end{equation*}
  edges.
\end{proof}

\subsection{Non-maximal Core Entropy}
\label{sec:nonmaxentropy}

Let $f\in \IC[T]$ be of degree $d \ge 2$ and postcritically finite.
Recall that $\rho(f)$ is the exponential core entropy of $f$, see
Definition~\ref{def:entropy}.

Lemma~\ref{lem:entropybound} implies  $ \rho(f) \le d$.
Our goal is to improve this upper bound to a strict inequality
in the hyperbolic case. 

For an integer $d\ge 2$, we let $\Psi_d$ denote the degree $d$ monic
\textit{Chebyshev polynomial}, normalized as in Problem
7-c~\cite{Milnor} and Section B.1~\cite{Milnor:pasting}.
For example $\Psi_2 = T^2-2$ and $\Psi_3=T^3-3T$. We have the functional
equation
\begin{equation}
  \label{eq:chebynorm}
  \Psi_d (2\cos \theta) = 2\cos(d \theta) 
\end{equation}
for all $\theta\in\IC$. 
Moreover,  the Julia set of $\pm \Psi_d$ equals $[-2,2]$, see
Problem 7-c~\cite{Milnor}.
This exercise also shows that both $\pm 2$ are critical values of
$\Psi_d$ for $d\ge 3$. Certainly $-2$ is a critical value for
$\Psi_2$. So $\Psi_d$ is not hyperbolic by Lemma~\ref{lem:hyperboliccharacterization}.
The same argument shows that $-\Psi_d$ is not hyperbolic.

Two polynomials $f,g\in\IC[T]$ are
called \textit{linearly conjugate}
if there exist $a\in\IC^\times$ and $b\in\IC$ with $f = (g(aT+b)-b)/a$.
Clearly, two linearly conjugate polynomials are both hyperbolic if one
of them is. 

\begin{propo}
  \label{prop:rhoflessthand}
  Let $f\in\IC[T]$ be of degree $d \ge 2$ and postcritically finite.
  Suppose $f$ is not linearly conjugate to $\pm \Psi_d$. This is the
  case, for example, if
  $f$ is hyperbolic. Then $1\le \rho(f)<d$.
\end{propo}

This proposition is based on two results. The first one is the
following theorem of Milnor~\cite[Lemma B.3]{Milnor:pasting}.

\begin{thm}[Milnor]
  \label{thm:chebpropn}
  Let $f\in\IC[T]$ be of degree at least 2. Suppose that
  the filled Julia set of $f$ is a finite tree. Then $f$ is linearly
  conjugate to $\pm\Psi_d$ for an integer $d\ge 2$.
\end{thm}

The second result is a statement closely related to
Alsed\`a--Fagella's Theorem~D~\cite{AlsedaFagella} which characterizes
Misiurewicz polynomials of maximal core entropy. The following lemma
is a close adaptation of their proof to our setting.

\begin{lemma}
  \label{lem:JfsubsetHf}
  Let $f\in \IC[T]$ be of degree $d \ge 2$ and postcritically finite.
  Suppose $J_f\not\subset H_f$. Let $z\in H_f$. There exists
  $s_0=s_0(f,z)\ge 0$ such that 
  $\#\left(f^{(s)}|_{H_f}\right)^{-1}(z)<d^s$ for all 
  integers $s\ge s_0$.
\end{lemma}
\begin{proof}
  We follow the line of argumentation in the proof of 
  Theorem~D~\cite{AlsedaFagella}.
  
  Let $z\in H_f$. We begin by proving the \textit{a priori}
  weaker conclusion that
  there exists $s\in\IN_0$ with
  $\#(f^{(s)}|_{H_f})^{-1}(z)<d^s$.
  We aim at a contradiction and assume 
  $\# (f^{(s)}|_{H_f})^{-1}(z) \ge d^{s}$ for all $s\in\IN_0$.
  Since $(f^{(s)})^{-1}(z)$ has at most $\deg f^{(s)} = d^s$ elements
  for all $s$, we find
  $(f^{(s)}|_{H_f})^{-1}(z) = (f^{(s)})^{-1}(z)$ for all $s\in\IN_0$.
  In particular,
  $(f^{(s)})^{-1}(z)  \subset H_f$ for all $s\in\IN_0$.
  In other words, 
  the backward orbit  $B = \bigcup_{s\in\IN_0} (f^{(s)})^{-1}(z)$ is
  contained in  $ H_f$.
  The topological closure $\overline B$ is also contained in the
  closed set $H_f$. 
  Observe that $f^{-1}(B)\subset B$
  and hence $\overline{f^{-1}(B)}\subset \overline B$.
  
  As $f$ is continuous, we have $\overline{f^{-1}(B)}\subset
  f^{-1}(\overline B)$. But $f$ is also an open mapping
  $\IC\rightarrow\IC$, since it is a non-constant polynomial. This
  implies
  $f^{-1}(\overline B)\subset\overline{f^{-1}(B)}$.
  Hence $f^{-1}(\overline B)=\overline{f^{-1}(B)}$. Together with the
  previous paragraph we find $f^{-1}(\overline B)\subset \overline B$.

  Let $U$ denote the open set $\IC\ssm
  \overline B$.
  Say $w\in \IC$ with $f(w)\not\in U$. Then $f(w)\in \overline B$ and
  hence
  $w\in f^{-1}(\overline B)\subset \overline B$. So $w\not\in U$. In
  other words,  $f(U) \subset U$.

  Observe that $B$ is an infinite set as it contains the set
  $(f^{(s)})^{-1}(z)$ of cardinality $d^s$ for all $s\ge 0$. 
  In particular, $\#\overline B\ge 2$. By Montel's Theorem, see
  Theorem~3.7~\cite{Milnor}, the family $\{f^{(s)}|_\Omega : s\in\IN_0\}$ is
  normal for each connect component $\Omega$ of $U$.
  So $U=\IC\ssm \overline B$ is contained in the Fatou set of
  $f$. As the Fatou set is the complement of the Julia set we find
  $J_f\subset \overline B$. Therefore, $J_f \subset  H_f$. This
  contradicts the hypothesis.

  To summarize, we have shown that given $z\in H_f$, there exists
  $s_0\in\IN_0$, which may depend on $f$ and $z$, with $\#
  (f^{(s_0)}|_{H_f})^{-1}(z) < d^{s_0}$.

  Recall that $f(H_f)\subset H_f$, so $f|_{H_f}\colon H_f\rightarrow
  H_f$ is a self-map. For all  $m\ge 0$ we thus find
  \begin{equation*}
    \# (f^{(m+s_0)}|_{H_f})^{-1}(z) = \#
    (f^{(m)}|_{H_f})^{-1}\left((f^{(s_0)}|_{H_f})^{-1}(z)\right)
    < d^{m}d^{s_0}=d^{m+s_0}.\qedhere
  \end{equation*}
\end{proof}

\begin{lemma}
  \label{lem:JfnotsubsetHfrhoflessd}
  Let $f\in \IC[T]$ be of degree $d \ge 2$ and postcritically finite.
  Suppose $J_f\not\subset H_f$. Then $\rho(f)<d$.   
\end{lemma}
\begin{proof}
  If the set  of  edges $E_f$ of the Hubbard tree is empty, then
  $\rho(f)=1$ by definition. In this case, the lemma is true. So let us assume
  $E_f\not=\emptyset$. 
  Let $L$ be as in (\ref{def:L}). Using the basis $E_f$ of $\IZ[E_f]$ we
  represent $L$ by a square
  matrix with integer coefficients.
  We also use $L$ to denote this matrix. 
  
  We fix a point $z_e$ on the relative interior of
  each edge $e\in E_f$. That is, $z_e$ is not in the set of vertices
  $V_f$ and $z_e\in H_f$. 
  
  Let $s\in\IN$. This parameter will be fixed later in terms of the
  $z_e$.
  
  We apply Lemma~\ref{lem:Lrcommutativediag}(ii) in the case $t=s$.
  So $L^s e = f_*^{(s)}(r^{(s)}(e))$.

  On the other hand, we have $L^s e = \sum_{e' \in E_f} l(e',e) e'$
  with $l(e',e)\in\IZ$. Here $l(e',e)$ is the entry of
  $L^s$ in row $e'$, column $e$.

  Recall that $r^{(s)}(e)$ is reduced and has the same support as $e$
  by the arguments around (\ref{eq:rsiterated}). We write $r^{(s)}(e)
  = \sum_{e''} n(e,e'') e''$ where $e''$ ranges over $E_f^{(s)}$ and $n(e,e'')
  \in \{0,1\}$.
  We have $n(e,e'')=1$ if and only if
  the support of $e''$ is contained in the support of $e$.
  In particular,  $n(e,e'')=1$ implies that
  the support of $e''$   is contained in 
  Hubbard tree $H_f$.
  
  Therefore,
  \begin{equation*}
    \sum_{e'\in E_f} l(e',e) e' = L^s e =
    f_*^{(s)} (r^{(s)}(e)) = \sum_{e''\in E_f^{(s)}} n(e,e'')
    f_*^{(s)}(e'') 
  \end{equation*}
  where  $f_*^{(s)}(e'')$ is an edge in $E_f$ for all 
  $e''$. Comparing coefficients we find 
  \begin{equation*}
    l(e',e)  =  \sum_{\substack{ e''\in E_f^{(s)} \\
        f_*^{(s)}(e'')=e'}} n(e,e'')
    \quad\text{for all}\quad e'\in E_f. 
  \end{equation*}

  Suppose $e''\in E_f^{(s)}$ with $n(e,e'')=1$. T hen $e''([0,1])\subset
  e([0,1])$. 
  If $f^{(s)}(e'')=e'$, then $e([0,1])$ contains a
  preimage under $f^{(s)}$ of the point $z_{e'}$ we fixed in advance.
  As $n(e,e'')\in \{0,1\}$
  for all $e''$ we conclude
  \begin{equation*}
    0\le l(e',e)  \le \#\{ z\in e([0,1]) : f^{(s)}(z) = z_{e'} \}
    \quad\text{for all}\quad e,e'\in E_f. 
  \end{equation*}

  Two different edges $e,\widetilde e\in E_f$ intersect on the set of
  vertices $V_f$ of $H_f$. But vertices are mapped to vertices under
  the action of $f$ by Lemma~\ref{lem:hubbardbasic}(iv).
  As $z_{e'}$ is
  an interior point of $e'$, there can be no  $z$ in the support of $e$ and
  $\widetilde e$ for which $f^{(s)}(z) = z_{e'}$.
  We sum over $e\in E_f$ to find
  \begin{equation*}
    \sum_{e\in E_f} l(e',e) \le \#\{ z\in H_f : f^{(s)}(z) = z_{e'}
    \} = \# (f^{(s)}|_{H_f})^{-1}(z_{e'})
    \quad\text{for all}\quad e'\in E_f.
  \end{equation*}

  By Lemma~\ref{lem:JfsubsetHf}, we may fix $s\in\IN$ large enough such
  that $\# (f^{(s)}|_{H_f})^{-1}(z_{e'}) < d^s$ for all the
  finitely many ${e'}\in E_f$. Therefore,
  \begin{equation*}
    \sum_{e\in E_f} l(e',e) < d^s    \quad\text{for all}\quad e'\in E_f.
  \end{equation*}
  
  Recall that the $l(e',e)\ge 0$ are the entries of $L^s$.
  Thus, the row-sum norm satisfies $\|L^s\|_\infty < d^s$. We take the
  $s$-th root to find $\|L^s\|_\infty^{1/s} < d$.  By
  (\ref{eq:specradbound}) we conclude $\rho(L) < d$.
  Thus $\rho(f)<d$ by the definition of the exponential core entropy.
\end{proof}

\begin{proof}[Proof of Proposition~\ref{prop:rhoflessthand}]
  The polynomials $\pm \Psi_d$ are not hyperbolic. So any hyperbolic
  polynomial is not linearly conjugate to $\pm \Psi_d$. This settles
  one claim of the proposition.
  
  We have $\rho(f)\ge 1$ by definition of the exponential core entropy.

  Let us assume that $\rho(f) \ge d$. So $J_f \subset H_f$ by
  Lemma~\ref{lem:JfnotsubsetHfrhoflessd}. 

  We claim that $J_f = K_f$. Indeed, we assume $J_f \subsetneq K_f$
  and will derive a contradiction. We fix a connected component $U$
  of $K_f\ssm J_f$. Then $U$ is open in $\IC$ as $J_f$ is the
  topological boundary of $K_f$. Observe that $\partial U \subset J_f$.
  Therefore, $\partial U \subset H_f$ by what we showed above.
  
  As $H_f$ is a finite topological tree, it cannot contain the open
  subset $U$ of $\IC$.
  Fix any $z\in U\ssm H_f$ and any $w\in \IC\ssm (H_f\cup\overline{U})$, the
  latter exists since $H_f\cup \overline U$ is bounded.
  As $H_f$ is a tree, the complement $\IC\ssm H_f$ is connected and open.
  So there exists a path $\gamma\colon [0,1]\rightarrow\IC \ssm H_f$
  with $\gamma(0)=z\in U$ and $\gamma(1) = w\in \IC\ssm \overline U$.
  There is $t\in [0,1]$
  with $\gamma(t) \in \partial U$. But this contradicts $\partial
  U\subset H_f$.

  This contradiction implies our claim
  \begin{equation}
    \label{eq:KfeqJf}
    J_f = K_f.   
  \end{equation}

  As $J_f \subset H_f\subset K_f$ we conclude that $J_f = H_f$ is a finite
  topological tree. Milnor's Lemma~B.3~\cite{Milnor:pasting} implies
  that $f$ is linearly conjugate to $\pm \Psi_d$.

  This last paragraph can be replaced by the following direct
  argument if we assume \textit{a priori}
  that $f$ is hyperbolic. Indeed, 
  fix a critical point $z_0$ of $f$. Then $z_0\in K_f$ as $f$ is 
  postcritically finite. Then (\ref{eq:KfeqJf}) implies $z_0\in
  J_f$.
  So $f$ cannot be hyperbolic
  by Lemma~\ref{lem:hyperboliccharacterization}.
\end{proof}


\section{The Dynamical Hedgehog}
\label{sec:dynhedgehog}

Let $f\in \IC[T]$ be monic of degree $d\ge 2$, hyperbolic, and postcritically
finite.

Suppose we are presented with a finite set of complex numbers that have
small local canonical height (\ref{def:lambdafcomplex}). In this section we
construct a finite topological tree containing this set and with
controlled transfinite diameter. The construction relies on the
Hubbard tree of $f$, whose construction we reviewed in
Section~\ref{sec:hubbard}.

Our main technical result in the section is the following proposition.
We will use it to bound the transfinite diameter
of the tree further
down. Recall that $K_f$ and $J_f$ are the filled Julia set and
Julia set of $f$, respectively. 

\begin{proposition}
  \label{prop:growtree}
  Let $f\in \IC[T]$ be monic of degree $d\ge 2$, hyperbolic, and
  postcritically finite. Let $r > 0$ be sufficiently small in
  terms of $f$. Let $\nu>0$. There exist
  $c_{1}(f,r)>0,c_2(f,r,\nu)>0,c_3(f,r)>0,$ and
  $c_4(f,r)>0$ with the following property. Let $s\in\IN_0$ and
  $z_1,\ldots,z_D\in \IC$ satisfy
  \begin{equation}
    \label{eq:propDdshypo}
    \max_{1\le j\le D} \lambda_f(z_j) < \frac{c_1(f,r)}{d^s}
    \quad\text{and}\quad D\le c_2(f,r,\nu) d^s. 
  \end{equation}
  There exists a finite topological tree $\cT\subset\IC$ with the
  following properties.
  \begin{enumerate}
  \item[(i)] We have $\{z_1,\ldots,z_D\}\subset \cT$.
  \item[(ii)]
    There is an integer $n$ with
    \begin{equation*}
      0 \le n \le D+c_4(f,r) \nu d^s
    \end{equation*}
    and open disks $\cD_1,\ldots,\cD_n$ of radius at most
    $r$ with the following properties.
    For all $j$, the disk $\cD_j$ is centered at a point of $J_f$ and
    there is a
    connected component $V_j$ of $(f^{(s)})^{-1}(\cD_j)$ such that
    \begin{equation*}
      \cT \subset V_1\cup\cdots\cup V_n \cup \{z\in K_f :
      \mathrm{dist}(f^{(s)}(z),J_f) \ge c_3(f,r)\}.
    \end{equation*}   
  \end{enumerate}
\end{proposition}

Let us emphasize that $c_1(f,r), c_3(f,r),$ and $c_4(f,r)$ in this
proposition are independent of $\nu$. Only $c_2(f,r,\nu)$ is allowed
to depend on $\nu$. 

We will use the description to $\cT$ in part (ii) and the results from
Sections~\ref{sec:transfinitei} and \ref{sec:trees}  to bound its
logarithmic transfinite diameter  $\mathsf{d}_\infty(\cT)$ from above.

\begin{corollary}
  \label{cor:growtree}
  Let $f\in \IC[T]$ be monic of degree $d\ge 2$, hyperbolic, and
  postcritically finite.
  There exists $c(f)>0$ with the following property.
  Let $D\ge 1$ and $z_1,\ldots,z_D\in\IC$ satisfy
  \begin{equation}
    \label{eq:growtreeshyp}
    \max_{1\le j\le D} \lambda_f(z_j) <\frac{c(f)}{D}. 
  \end{equation}
  There exists a finite topological tree $\cT\subset\IC$ with the
  following properties.
  \begin{enumerate}
  \item [(i)] We have $\{z_1,\ldots,z_D\}\subset \cT$.

  \item[(ii)] We have $\mathsf{d}_\infty(\cT)\le -c(f)/D$.
  \end{enumerate}
\end{corollary}
\begin{proof}
  We fix $r > 0$ small enough in terms of $f$ such that
  Proposition~\ref{prop:transfiniteapplication} applies to $f$. We also
  assume that $r$ is small enough for
  Proposition~\ref{prop:growtree} to hold.
  Let
  $c_1(f,r),c_2(f,r,\nu),c_3(f,r),$ and
  $c_4(f,r)$ be as in Proposition~\ref{prop:growtree}. Here $\nu$ is to be chosen
  later in terms of $f$ and $r>0$.

  Later, we will apply Proposition~\ref{prop:transfiniteapplication}
  to $\delta=c_3(f,r)$. Let $c_5(f,r)>0$ be
  $c_2(f,\delta)$ from the said proposition.
  We fix
  \begin{equation}
    \label{def:nu}
    \nu=\frac{c_5(f,r)}{2c_4(f,r)}.
  \end{equation}
  It is crucial
  that $c_3(f,r)$ and $c_4(f,r)$ do
  not depend on $\nu$ as $c_5$ was defined
  in terms of $c_3$.
  
  It will be convenient to set
  $c_6(f,r,\nu)=\min\{1,c_2(f,r,\nu)\}\min\{1,c_5(f,r)/2\}$.
  We ultimately set the value $c(f)>0$ from the conclusion of the
  corollary in function of $f,r,$ and $\nu$.

  We choose
  \begin{equation*}
    s = \left\lceil
    \log^+\left(\frac{D}{c_6(f,r,\nu)}\right)/(\log
  d)\right\rceil \in \IN_0. 
  \end{equation*}
  
  It follows that  $D\le c_6(f,r,\nu)d^s\le
  c_2(f,r,\nu)d^s$. So the  condition on the right
  in (\ref{eq:propDdshypo}) is satisfied.
  If $D\ge c_6(f,r,\nu)$, then  $d^s <
  {dD}/{c_6(f,r,\nu)}$ by our choice of $s$. 
  Thus
  $$\max_{1\le j\le D} \lambda_f(z_j)< \frac{c(f)}{D} \le
  \frac{c(f)d}{c_6(f,r,\nu)d^s} \le \frac{c_1(f,r)}{d^s}$$
  by (\ref{eq:growtreeshyp}) and
  as we may assume $c(f) \le c_1(f,r)c_6(f,r,\nu)/d$.
  So the  condition on the left of (\ref{eq:propDdshypo}) is satisfied.
  If $D<c_6(f,r,\nu)$, then $s=0$ by definition. Then
  (\ref{eq:propDdshypo})
  also follows from (\ref{eq:growtreeshyp}) and
  since we may assume $c(f)\le c_1(f,r)$.
  
  We may apply
  Proposition~\ref{prop:growtree} to $s$ and $z_1,\ldots,z_D$ and
  obtain a tree $\cT$ containing $z_1,\ldots,z_D$. We also know that
  \begin{equation*}
    \cT\subset V_1\cup\cdots\cup V_n \cup \{z\in K_f :
    \mathrm{dist}(f^{(s)}(z),J_f)\ge c_3(f,r)\}    
  \end{equation*}
  with  $V_1,\ldots,V_n$  as in the
  said proposition. Above, we
  saw $D\le c_6(f,r,\nu) d^s$. Therefore,
  $$n\le D+c_4(f,r) \nu d^s 
  \le (c_6(f,r,\nu)+c_4(f,r)\nu)d^s
  \le
  \left(\frac{c_5(f,r)}{2}+c_4(f,r)\nu\right)d^s =
  c_5(f,r)d^s$$
  where we used the choice of $\nu$ made in (\ref{def:nu}).

  It remains to bound $\mathsf{d}_\infty(\cT)$ from above. This is where   
  Proposition~\ref{prop:transfiniteapplication} comes in. We already
  announced the choice 
  $\delta = c_3(f,r)$
  and we showed  $n\le c_5(f,r)d^s$. Recall that
  $c_5(f,r)$ is $c_2(f,\delta)$ from Proposition~\ref{prop:transfiniteapplication}.
  Thus Proposition~\ref{prop:transfiniteapplication}
  implies $\mathsf{d}_\infty(\cT)\le -c_7(f,r)/d^s$
  where $c_7(f,r)$ is $c_1(f,\delta)$ from the said
  proposition.
  
  Recall that $D\ge c_6(f,r,\nu)$ implies $d^s
  <dD/c_6(f,r,\nu)$. In this case, $\mathsf{d}_\infty(\cT)\le
  -c_6(f,r,\nu)c_7(f,r)/(dD)$ and the corollary follows for
  $c(f)$ small enough. If $D<c_6(f,r,\nu)$, then $s=0$ and
  $\mathsf{d}_\infty(\cT)\le -c_7(f,r)\le -c_7(f,r)/D$. Again,
  we are done.  
\end{proof}

\begin{remark}
  \begin{enumerate}
  \item [(i)]
    We explain how this corollary can be understood as a dynamical
    version of Dubinin's Theorem~\cite{Dubinin}.
    So let us consider for the moment $f=T^2$. The Julia set is the
    complex unit circle and the local canonical height is $\lambda_f(z) =
    \log^+|z|$. Suppose we are given $z_1,\ldots,z_D\in\IC$ with $D\ge 1$.
    The hedgehog
    \begin{equation*}
      \cT = \bigcup_{i=1}^D [0,1]z_i
    \end{equation*}
    is a finite topological tree containing $z_1,\ldots,z_D$.
    Dubinin's Theorem~\cite{Dubinin}, implies    
    \begin{equation*}
      \mathsf{d}_\infty(\cT) \le -\frac{\log 4}{D}+\log\max_{1\le i\le
        D}\log|z_i|;
    \end{equation*}
    we refer also to the discussion in Dimitrov's work \cite{dimitrov:SZ}.
    If we assume $\max_{1\le i\le D}\log^+|z_i| \le (\log 2)/D$,
    then
    \begin{equation*}
      \mathsf{d}_\infty(\cT)\le -\frac{\log 2}{D} <0.
    \end{equation*}
    Up-to a constant, the logarithmic transfinite diameter of $\cT$ satisfies a
    similar bound as in (ii) of Corollary~\ref{cor:growtree}.

    \item[(ii)] In Corollary~\ref{cor:growtree} we cannot drop the
      hypothesis that $f$ is hyperbolic. Indeed, $f=T^2-2$ is
      postcritically finite with Julia set $[-2,2]$. The local canonical height of
      $\pm 2$ with respect to $f$ is zero.      

      Any connected and compact subset $\cT$ of $\IC$ containing $-2$ and
      $2$ has diameter at least $4$. Thus $\mathsf{d}_\infty(\cT)\ge \log(4/4)=0$ by
      Theorem~5.3.2(a)~\cite{Ransford}.  On the other hand,
      the logarithmic
      transfinite diameter of any finite topological tree produced by
      Corollary~\ref{cor:growtree}
      is negative.

      Observe that   $T^2-2$ is not hyperbolic, as the critical point lies in
      the Julia set, see
      Lemma~\ref{lem:hyperboliccharacterization}.
    \item[(iii)] In Corollary~\ref{cor:growtree} we cannot drop the
      hypothesis that $f$ is postcritically finite. Indeed, let us
      consider $f=T^2+c$ with $|c|>2$. We reproduce here a well-known
      computation. Observe that $|f(z)|-|z|=|z^2+c|-|z|\ge
      |z|^2-|z|-|c|$. The map $t\mapsto t^2-t-|c|$ is increasing on
      $[1/2,\infty)$. With $t=|c|$ we find  $|f(z)|-|z|\ge
      |c|^2-2|c|>0$ if $|z|\ge |c|$. Thus the critical orbit of $f$ is
      unbounded. In particular, $f$ is not postcritically finite. But it is
      hyperbolic by Lemma~\ref{lem:hyperboliccharacterization}.

      The points $z_{\pm} = (1\pm \sqrt{1-4c})/2$ are fixed points of $f$. We
      have $|f'(z_{\pm})|=|2z_{\pm}|=|1\pm \sqrt{1-4c}|\ge \sqrt{|1-4c|}-1 \ge
      \sqrt{4|c|-1}-1>\sqrt 7 -1>1$ as $|c|>2$. Thus $z_{\pm}$ are repelling
      fixed points and thus lie in the Julia set $J_f$. 

      Any compact and connected set $\cT$ containing both $z_{\pm}$
      has diameter at least $|z_+-z_-|=|\sqrt{1-4c}|\ge
      \sqrt{4|c|-1}$. If $|c|\ge 17/4$, then the diameter is at least
      $4$. As in part (ii) we find $\mathsf{d}_\infty(\cT)\ge
      0$. 
  \end{enumerate}
\end{remark}

\subsection{Growing the Tree}

Let $f\in \IC[T]$ be monic of degree $d\ge 2$, hyperbolic, and
postcritically finite.

Attached to $f$ is the Hubbard tree of $H_f$ from
Section~\ref{sec:hubbard}. Let $s\in\IN_0$ be a parameter. We
considered preimages $H_f^{(s)}$ of $H_f$ under $f^{(s)}$ given by
(\ref{def:hubbardpreimage}). Recall that $H_f^{(s)}\subset\IC$ is a
finite topological tree, see Section~\ref{sec:growth} for details.

Fix $r > 0$ small enough so that
Proposition~\ref{prop:transfiniteapplication} holds as does
Lemma~\ref{lem:fninvuniform_new1} with $\epsilon=r$.

\begin{lemma}
  \label{lem:quill}
  We keep the notation from above. There are
  $s_0=s_0(f,r)\in\IN$ and  $c(f,r)>0$
  with the following property. For all $z\in \IC$ and all $s\in\IN_0$ with
  $\lambda_f(z)<c(f,r)/d^s$ there is $\tilde z\in \IC$ such that the
  following holds true.
  \begin{enumerate}
  \item [(i)]  The iterate  $f^{(s+s_0)}(\tilde z)$ lies in a periodic bounded Fatou component of
   $f$ and $\mathrm{dist}(f^{(s)}(\tilde z),J_f)\ge c(f,r)$.
   In particular, $\tilde z\in K_f\ssm J_f$.
  \item[(ii)]   If $\mathrm{dist}(f^{(s)}(z),J_f)<r/2$, then  $z$
    and $\tilde z$ lie in the same
    connected component  of $(f^{(s)})^{-1}(B_{r}(w))$
    for some $w\in J_f$.
  \item[(iii)]   If $\mathrm{dist}(f^{(s)}(z),J_f)\ge r/2$,
    then $z=\tilde z$.
  \end{enumerate}
\end{lemma}
\begin{proof}
  We will fix $c(f,r)>0$ in function of $f$ and $r$
  during this proof.
  The ``in particular'' statement in (i)
  follows from the preceding claim as the preimage under $f$ of a point
  in $K_f\ssm J_f$ lies in $K_f\ssm J_f$.

  Let us continue with some preliminary remarks.

  As $J_f$ is compact, there exist $l\ge 1$ and $w_1,\ldots,w_l\in J_f$ with 
  \begin{equation}
    \label{eq:Jfcover}
    J_f \subset  B_{r/2}(w_1)\cup\cdots\cup B_{r/2}(w_l).
  \end{equation}

  By Lemma~\ref{lem:KfJfdenseinKf}(ii) the complement $K_f\ssm J_f$ lies
  dense in $K_f$.
  So for each $i\in \{1,\ldots,l\}$ we fix $\tilde w_i \in
  B_{r}(w_i)\cap (K_f\ssm J_f)$. We may assume that
  \begin{equation}
    \label{eq:condctildew}
    0 < c(f,r) \le \min_{1\le i\le l} \mathrm{dist}(\tilde w_i,J_f). 
  \end{equation}
  
  The connected component of $K_f\ssm J_f$ that contains $\tilde w_i$
  is preperiodic by Lemma~\ref{lem:fatoucomponents}(iv). We may fix  
  $s_0=s_0(f,r)\in\IN$ such that $f^{(s_0)}(\tilde w_i)$ lies in a periodic
  bounded Fatou component of $f$ for all $i$. Here $s_0$ depends on $f$ and $r$.

  The compact set $\{w\in K_f : \mathrm{dist}(w,J_f)\ge r/2 \}$ is
  contained in a necessarily finite union of connected components of
  $K_f\ssm J_f$. These connected components depend only on $f$ and the
  choice of $r$. So, again by Lemma~\ref{lem:fatoucomponents}(iv)
  and 
  after possibly increasing $s_0$, we may assume that $f^{(s_0)}(w')$
  lies in a periodic bounded Fatou component for all $w'\in K_f$ with
  $\mathrm{dist}(w',J_f)\ge r/2$. We will continue to use this $s_0$, it
  depends only on $f$ and $r$.

  The local canonical height $\lambda_f\colon\IC\rightarrow[0,\infty)$
  is continuous and vanishes precisely on the compact set $K_f$. It
  differs from $w\mapsto \log^+|w|$ by a bounded function. If
  $c(f,r)>0$ is small enough, then $\lambda_f(w)<c(f,r)$
  implies $\mathrm{dist}(w,K_f)< r / 2$.

  Let $z\in \IC$ and $s\in\IN_0$ with $\lambda_f(z) < c(f,r)/d^s$.
  Then $\lambda_f(w)<c(f,r)$ for $w = f^{(s)}(z)$. We split up into two
  cases.

  \textbf{Case 1:} We have $\mathrm{dist}(w,J_f)<r/2$. By
  (\ref{eq:Jfcover}) there is $i$ with $|w-w_i|<r$. Fix the connected
  component $V$ of $(f^{(s)})^{-1}(B_{r}(w_i))$ that contains $z$. By
  Lemma~\ref{lem:fninvuniform_new1} 
  and the choice of $r$ we have
  $f^{(s)}(V) = B_{r}(w_i)$. As $B_{r}(w_i)$ contains $\tilde w_i$ there
  exists $\tilde z\in V$ with $f^{(s)}(\tilde z) = \tilde w_i$.
  
  Then $f^{(s+s_0)}(\tilde z) = f^{(s_0)}(\tilde w_i)$ is in a periodic
  bounded Fatou component of $f$ by the choice of $s_0$. This is the
  first claim in (i).
  
  We have $\mathrm{dist}(f^{(s)}(\tilde z),J_f) =
   \mathrm{dist}(\tilde w_i,J_f)$.
  So the second claim in (i) follows from (\ref{eq:condctildew}). Conclusion (ii) holds as $V$ contains $z$ and $\tilde
  z$ by construction.
    
  \textbf{Case 2:} We have $\mathrm{dist}(w,J_f)\ge r/2$.
  Note that $w\in K_f$ as we have $\mathrm{dist}(w,K_f)< r/2$.  
  By our choice of $s_0$, the
  iterate $f^{(s_0)}(w)$ lies in a periodic bounded Fatou
  component of $f$.  In this case we set $\tilde z = z$.
  Thus $f^{(s+s_0)}(\tilde z) = f^{(s+s_0)}(z) = f^{(s_0)}(w)$
  lies in a periodic bounded Fatou component of $f$.
  Moreover, $\mathrm{dist}(f^{(s)}(\tilde z),J_f) = \mathrm{dist}( w,J_f)\ge
  r /2\ge c(f,r)$ as we may assume $c(f,r)
  \le r/2$. So parts (i) and (iii)
  follow. 
\end{proof}

We continue to work with $f$ and $r$ fixed and as above.
Let $s_0(f,r)$ be as in the
previous lemma.

We continue with a simple compactness argument. Recall that
$H_f^{(s_0)}$ is a finite topological tree. We treat its edges as arcs
$[0,1]\rightarrow H_f^{(s_0)}$.

\begin{lemma}
  \label{lem:smallepsilonpcfhyp}
  The following holds for all $\delta >
  0$ small enough in terms of $f,r,$ and $s_0$. Let $e$ be an edge of
  $H_f^{(s_0)}$ and $z,z'\in e$ be members of $B_\delta(w)$ for some
  $w\in J_f$. Then the segment on $e$ between $z$ and $z'$ lies in an
  open disk of radius $\ge \delta$ around $w$ that is disjoint from
  $\pco_{\ge 1}(f)$.  
\end{lemma}
\begin{proof}  
  The minimum
  \begin{equation}
    \label{eq:choose}
    \kappa = \min_{c\in \pco_{\ge 1}(f)} \mathrm{dist}(c,J_f)
  \end{equation}
  is well-defined as $f$ is postcritically finite. Note that $\kappa>0$;
  indeed $f$ is hyperbolic by assumption and so 
  $J_f$ contains no  point
  in the postcritical orbit $\pco_{\ge 1}(f)$, see
  Lemma~\ref{lem:hyperboliccharacterization}.
  
  \textbf{Claim:} For all sufficiently small  $\delta > 0$ the
  following holds.  Let $e$ be an edge of 
  $H_f^{(s_0)}$, let $z,z'\in e$ such that
  $z,z'\in B_\delta(w)$ for some $w\in J_f$. 
  Then all points $z''$ on $e$ between $z$ and $z'$ satisfy
  $|z''-w|<\kappa$.

  We prove the claim by contradiction. Suppose there exist sequences
  $(\delta_n)_{n\in\IN},(t_n)_{n\in\IN},$
  $(t'_n)_{n\in\IN},(s_n)_{n\in\IN}$ in $\IR$, $(w_n)_{n\in\IN}$ in
  $J_f$, and edges $(e_n)_{n\in \IN}$ of $H_f^{(s_0)}$ with the
  following properties. We have $\lim_{n\rightarrow \infty}\delta_n=0$
  and for all $n\ge 1$ we have 
  $0\le t_n\le s_n\le t_n'\le 1$ such that
  \begin{equation*}
    |e_n(t_n)-w_n|<\delta_n, \quad |e_n(t'_n)-w_n|<\delta_n, \quad
    |e_n(s_n)-w_n| \ge \kappa.
  \end{equation*}

  Recall that $H_f^{(s_0)}$ has only finitely many edges.
  By passing to a subsequence we may assume that $e$ is independent
  of $n$. 
  Moreover, as $[0,1]$ and $J_f$ are compact we can pass to a further
  subsequence and assume that
  $\lim_{n\rightarrow\infty} t_n=t,  \lim_{n\rightarrow\infty}
  t'_n=t',\lim_{n\rightarrow\infty} s_n=s,$ and that $w_n$ converges
  to $w\in J_f$. 
  In the limit, the first two displayed inequalities above imply $e(t) = e(t')=w$. However, $e$ is an
  arc, and so $t=t'$. Thus $s=t=t'$ and we conclude $e(s)=w$. 
  The final displayed inequality above now yields $0=|e(s)-w| \ge \kappa$, a
  contradiction. So our claim follows. 

  We may fix
  $\delta$ small in terms of $f$. So we may assume
  $\delta \le \kappa$.
  
  Let $e$ and $z,z'\in e,w\in J_f$ be as in the lemma.   
  We have $B_\kappa(w)\cap \pco_{\ge 1}(f)=\emptyset$
  by the choice of $\kappa$ in (\ref{eq:choose}).
  By the claim, all points $z''\in e$ between $z$ and $z'$ lie in
  $B_\kappa(w)$.
  The lemma follows with the disk $B_\kappa(w)$ whose radius is at
  least $\delta$.
\end{proof}

For the following steps we fix
$\delta>0$ small enough in terms of $f,r,$ and $s_0$ as in the last
lemma. We may assume $\delta\le r$.
Recall that $s_0$ depends on $f$ and $r$. So $\delta$
depends only on $f$ and $r$.
\begin{lemma}
  \label{lem:subtreecover}
  Let $s\in\IN_0$ and let $\cT$ be a subtree of $H_f^{(s+s_0)}$ with
  $\tau$ edges. Then there exist open disks $\cD_1,\ldots,\cD_n$ of radius
  $\delta$ centered at points of $J_f$, and for each $i\in \{1,\ldots,n\}$ a connected
  component $V_i$ of $(f^{(s)})^{-1}(\mathcal{D}_i)$ such that
  \begin{equation}
    \label{eq:subtreecoversubset}
    \cT\subset V_1\cup\cdots\cup V_n\cup \{ z\in K_f
    : \mathrm{dist}(f^{(s)}(z),J_f)\ge \delta/2\}.
  \end{equation}
  Moreover,  $n= O_{f,\delta}(\tau)$.
\end{lemma}
\begin{proof}
  We recall that the Hubbard tree $H_f$ is contained in $K_f$ and so
  is $H_f^{(s+s_0)}$.
  
  \textbf{Claim:}
  Let $\tilde e$ be an edge of the tree $H_f^{(s+s_0)}$. Let $\cD$ be
  an open disk of radius $\delta$ centered at a point in $J_f$. Then
  \begin{equation}
    \label{eq:preimageedisk}
    \{z\in \tilde e
    : \mathrm{dist}(f^{(s)}(z),J_f)<\delta/2 \text{ and }f^{(s)}(z)\in
    \cD\}  
  \end{equation}
  is contained in single connected component of $(f^{(s)})^{-1}(\cD)$.
  
  Suppose $z$ and $z'$ are in (\ref{eq:preimageedisk}).
  Let $V$
  and $V'$ be the connected components
  of $(f^{(s)})^{-1}(\mathcal{D})$ containing $z$ and $z'$,
  respectively. We need to show $V=V'$.

  Recall that $w=f^{(s)}(z)$ and $w'=f^{(s)}(z')$ are both on the
  edge
  $f^{(s)}(\tilde e)$ of $H_f^{(s_0)}$.
  Lemma~\ref{lem:smallepsilonpcfhyp}, applied to these two points
  and the disk $\mathcal{D}$,  
  implies that the
  segment $[w,w']$ of $f^{(s)}(\tilde e)$ lies in an open disk $\mathcal{D}'$ disjoint
  from $\pco_{\ge 1}(f)$ and with
  $\mathcal{D}\subset \mathcal{D}'$.

  By Lemma~\ref{lem:hyperboliccovering0}, $f^{(s)}$ restricts to a covering map 
  $(f^{(s)})^{-1}(\mathcal{D}')\rightarrow \mathcal{D}'$ such that each connected
  component of the domain maps homeomorphically to $\mathcal{D}'$. 
  We lift $[w,w']$ via $f^{(s)}|_{(f^{(s)})^{-1}(\mathcal{D}')}$ starting at $z$, recall $f^{(s)}(z)=w$.
  By uniqueness, the lifted path coincides with the segment $[z,z']$ of $\tilde
  e$ between $z$ and $z'$. Thus $[z,z']\subset (f^{(s)})^{-1}(\mathcal{D}')$
  and $[z,z']$ is contained in a connected
  component $\mathcal{D}^0$ of $(f^{(s)})^{-1}(\mathcal{D}')$.
  
  As $\mathcal{D}\subset \mathcal{D}'$  each of $V,V'$ is
  contained in some connected component of  $(f^{(s)})^{-1}(\mathcal{D}')$.
  But as we just saw, $z$ and $z'$ are both contained in $\mathcal{D}^0$. Since
  $z\in V$ and $z'\in V'$, we get $V\subset \mathcal{D}^0$ and $V'\subset
  \mathcal{D}^0$. Finally, $f^{(s)}$ is injective on $\mathcal{D}^0$
  by Lemma~\ref{lem:hyperboliccovering0}(ii). As
  $f^{(s)}(V)=f^{(s)}(V')=\mathcal{D}$ we conclude
  $V=V'$.
  The claim follows.

  Recall that $J_f$ is compact.
  So there is a finite subset $\mathcal{C}$ of $J_f$ such
  that any point on $J_f$ lies in an open disk of radius $\delta/2$
  centered at a point in $\mathcal{C}$. The set $\mathcal{C}$ and its
  cardinality depend only on $f$ and $\delta$. 
 
  To  complete the proof, let 
  $z\in  \cT$. Recall that $\cT\subset H_f^{(s+s_0)}\subset K_f$.
  To show that $z$ lies in the right-hand side of
  (\ref{eq:subtreecoversubset}) we may assume
  $\mathrm{dist}(f^{(s)}(z),J_f) <\delta / 2$.
  Then $f^{(s)}(z)$ lies
  in an open  disk $\cD$ of radius $\delta$ centered at a point in
  $\mathcal{C}$.
  Moreover,
  $z$ lies on some edge $\tilde e$ of $\cT$. 
  For $\tilde e$ and $\cD$ fixed, 
  the claim implies that $z$ lies in a single connected component of 
  $(f^{(s)})^{-1}(\cD)$.
  
  The open disks in the conclusion are centered at the points in
  $\mathcal{C}$. For each such point we take the corresponding disk
  $\tau$ times, once for each edge of $\cT$. So $n = \tau \#\mathcal{C}$, 
  the number of possible pairs $(\tilde  e,\tau)$.
  Recall that $\#\mathcal{C}$ depends only on $f$ and $\delta$. Therefore, $n =
  O_{f,\delta}(\tau)$.  This completes
  the proof. 
\end{proof}

We are now ready to prove Proposition~\ref{prop:growtree}.

We assume that  $r>0$ is small enough in terms of $f$ to
satisfy the needs stated before Lemma~\ref{lem:quill}. Let us fix
$s_0\in\IN_0$, which depends on $f$ and $r$, as in the said
lemma. We also fix  $\delta \in (0,r]$
as before Lemma~\ref{lem:subtreecover}.
Then $\delta$ ultimately depends  only on $f$ and $r$.
Let $\nu > 0$.
The positive values $c_1(f,r),
c_2(f,r,\nu),c_3(f,r),$ and $c_4(f,r)$ will be
fixed throughout the argument. 

Let $z_1,\ldots,z_D\in\IC$ and $s\in\IN_0$ with
\begin{equation*}
  \max_{1\le j\le D} \lambda_f(z_j)< \frac{c_1(f,r)}{d^s}.
\end{equation*}
We will assume {$c_1(f,r)\le c(f,r)$} where
$c(f,r)$ is from Lemma~\ref{lem:quill}.

We will carefully track the dependency of expressions in our main
parameter $s\in\IN_0$.

For all $j\in \{1,\ldots,D\}$ we attach to each $z_j$ a point $\tilde
z_j$ as in Lemma~\ref{lem:quill}. So $\mathrm{dist}(f^{(s)}(\tilde
z_j),J_f)\ge c(f,r)>0$ and $\tilde z_j\in K_f\ssm J_f$
for all $j$. Let $U_j$ denote the bounded Fatou
component of $f$ containing $\tilde z_j$. So $f^{(s+s_0)}(U_j)$ is
also a bounded Fatou component of $f$, see
Lemma~\ref{lem:fatoucomponents}(iii). 
It is periodic under the action of
$f$ by Lemma~\ref{lem:quill}(i).

Recall that $\mathfrak{c}_{U_j}\in U_j$ denotes the center of $U_j$ as
defined in Section~\ref{sec:hubbard}.
We abbreviate $\mathfrak{c}_j = \mathfrak{c}_{U_j}$.
Then $f^{(s+s_0)}(\mathfrak{c}_{j})$ is the center of the periodic
bounded Fatou component
$f^{(s+s_0)}(U_j)$. So it lies in $\pco_{\ge 1}(f)$ by
Lemma~\ref{lem:centerprops}. In particular, $f^{(s+s_0)}(\mathfrak{c}_j)$
is a vertex of $H_f$. We conclude
that $\mathfrak{c}_{j}$ is a vertex of  $H_f^{(s+s_0)}$ for all
$j\in \{1,\ldots,D\}$.

As in Definition~\ref{def:entropy}, let $\rho(f)\ge 1$ denote the
exponential core entropy of $f$. We have $\rho(f)<d$ by
Proposition~\ref{prop:rhoflessthand} and as $f$ is hyperbolic.

Thus we may apply Corollary~\ref{cor:convexhull} to $f$ and $s+s_0$.
For $\cV$ we take the  vertices 
$\{\mathfrak{c}_{1},\ldots,\mathfrak{c}_{D}\}$ of $H_f^{(s+s_0)}$.
We need to check
$D\le \kappa(f,\nu) d^{s+s_0}$.
So if we take $c_2(f,r,\nu) = \kappa(f,\nu)$
then (\ref{eq:propDdshypo}) gives what we desire. Recall that $c_2$ is
allowed to depend on $\nu$. 
The convex hull $\cT_0$ of $\cV$ is a subtree of $H_f^{(s+s_0)}$ with
at most $\nu d^{s+s_0}$ edges.

We have constructed a subtree $\cT_0$ of $H_f^{(s+s_0)}$ to which
Lemma~\ref{lem:subtreecover} applies with $\tau\le \nu d^{s+s_0}$. 
Thus
\begin{equation}
  \label{eq:T0inclusion}
  \cT_0 \subset V_1\cup\cdots\cup V_m \cup \{z\in K_f :
  \mathrm{dist}(f^{(s)}(z),J_f) \ge \delta/2\}
\end{equation}
with $m = O_{f,\delta}(\nu d^{s+s_0})$.
So
\begin{equation}
  \label{eq:mc4nuds}
  m \le c_4(f,r) \nu d^{s} 
\end{equation}
where {$c_4(f,r)>0$};
recall that $\delta$ and $s_0$ were chosen in
function of $f$ and $r$. 
Each $V_j$ is a connected component of the preimage under
$f^{(s)}$ of some open disk of radius $\delta\le r$ centered
at a point of $J_f$. 

Our
tree needs to grow further as
$z_1,\ldots,z_D$ are not necessarily on $\cT_0$.   Our process is inductive.
To summarize, we will obtain a chain of finite topological trees
\begin{equation*}
  \cT_0 \subset \cT_1 \subset \cdots \subset \cT_{D}
\end{equation*}
as follows. 
Suppose we have constructed $\cT_{j-1}$ with $j\in \{1,\ldots,D-1\}$.
We will attach at most $2$ edges to $\cT_{j-1}$ and obtain
$\cT_{j}$ to achieve $z_j\in \cT_j$.
The final tree $\cT_D$ will contain $\cT_0$ as a subtree and also
all points $z_1,\ldots,z_D$.
But it may no longer be a subset of $K_f$. 

We load the induction with one additional piece of information.
Recall that $\cT_0$ is covered as in (\ref{eq:T0inclusion}). 
At each step we keep track of additional sets required to cover $\cT_j$ by showing
\begin{equation}
  \label{eq:Tjcover}
  \cT_j \subset \cT_{j-1} \cup W_j\cup \{ z\in K_f :
  \mathrm{dist}(f^{(s)}(z),J_f)\ge c_3(f,r)\}
\end{equation}
where each $W_j$ is either empty or a connected component of
$(f^{(s)})^{-1}(B_{r}(w_j))$ for some $w_j\in J_f$.
We may and will assume that {$c_3(f,r)\le \delta/2$} since
$\delta$ is a function of $f$ and $r$.
Note that $c_3(f,r)$ is independent of $z_1,\ldots,z_D$ and of $s,\nu$.

Let us begin the construction of the $\cT_j$ with $j\in
\{1,\ldots,D\}$.

Recall from above that $\tilde z_j \in U_j, f^{(s)}(\tilde z_j)\in K_f,$ and
$\mathrm{dist}(f^{(s)}(\tilde z_j),J_f)\ge c(f,r)$.

We are now in one of two cases corresponding to the conclusions (ii)
and (iii) of Lemma~\ref{lem:quill}.

\textbf{Case 1:} We have $\mathrm{dist}(f^{(s)}(z_j),J_f)<r/2$.

The center $\mathfrak{c}_j$ of $U_j$ is already a vertex of $\cT_0$.
Thus it is a vertex of $\cT_{j-1}$. We use an allowable
arc inside $U_j$ to connect $\tilde z_j$ to $\mathfrak{c}_j$. If the
allowable arc hits $\cT_{j-1}$ before reaching $\mathfrak{c}_j$, we
stop.

If we pick up a new edge, say $e$, then $e\subset U_j$ by the
definition of allowable arcs.
As $e$ is contained in a ray of $U_j$, the image $f^{(s)}(e)$ is an
allowable arc. 
 By
Lemma~\ref{lem:distarclemma} applied to $c(f,r)$ and $f^{(s)}(\tilde z_j)$,
all points of
$f^{(s)}(e)$ have distance at least $c_3(f,r)$ to $J_f$. Here we allow
{$c_3(f,r)$}
to be small in terms of $c(f,r)$. So this new edge is contained
in the very right-hand side of (\ref{eq:Tjcover}).

While we have now connected $\tilde z_j$ to $\cT_{j-1}$, we still need
to deal with $z_j$. So we add a second edge in order to connect $\tilde
z_j$ to $z_j$. For this we use the fact, afforded by
Lemma~\ref{lem:quill}(ii), that $z_j,\tilde z_j\in W$ where $W$ is a
connected component of the open set
$(f^{(s)})^{-1}(B_{r}(w_j))$ for some $w_j\in J_f$.
We connect $z_j$ to $\tilde z_j$ by any arc
inside $W$. Again, we stop if the arc hits $\cT_{j-1}$ or the
allowable arc from above, before reaching $\tilde z_j$. This
construction leads to an additional $W_j=W$ as in (\ref{eq:Tjcover}).
If $\tilde z_j$ happens to lie in $\cT_{j-1}$ there is nothing to do, and
we set $W_j=\emptyset$. 

\textbf{Case 2:} We have $\mathrm{dist}(f^{(s)}(z_j),J_f)\ge r/2$.

Here $\tilde z_j = z_j$ by Lemma~\ref{lem:quill}(iii). So we need only
connect $\tilde z_j$ to $\mathfrak{c}_j\in \cT_{j-1}$.  
Thus we need only repeat the first step of the construction in
case 1. 
This leads to at most one new edge with is contained in the very
right-hand side of (\ref{eq:Tjcover}). In this case we take $W_j=\emptyset$. 

We have completed our construction. Indeed, the inclusion
(\ref{eq:Tjcover}) yields
\begin{equation*}
  \cT \subset \cT_0 \cup W_1\cup\cdots \cup W_D \cup
  \{z\in K_f : \mathrm{dist}(f^{(s)}(z),J_f)\ge c_3(f,r)\}
\end{equation*}
where $\cT = \cT_D$. Moreover, by (\ref{eq:T0inclusion}) 
and $c_3(f,r)\le \delta/2$ we find
\begin{equation*}
  \cT \subset V_1\cup\cdots\cup V_m \cup W_1\cup\cdots \cup W_D \cup
  \{z\in K_f : \mathrm{dist}(f^{(s)}(z),J_f)\ge c_3(f,r)\}.
\end{equation*}

By the construction, all $z_1,\ldots,z_D$ lie on $\cT$. This is the
property stated in
Proposition~\ref{prop:growtree} part (i).

For part (ii) we  take  $V_1,\ldots,V_n$ to
be  $V_1,\ldots,V_m,W_1,\ldots,W_D$ deprived of the empty set.
So $n\le m+D$. 
To bound $m$ from above we use (\ref{eq:mc4nuds}), we find
$n\le m+D \le D+c_4(f,r) \nu d^s$, as desired. \qed

\subsection{Growing a Tree in the Filled Julia Set}

We conclude this section with a qualitative result for trees that are
contained entirely in the
filled Julia set of a complex polynomial. We relax the conditions on
the polynomial, it no longer needs to be postcritically finite or hyperbolic.

\begin{proposition}
  \label{prop:cappreper}
  Let $f\in\IC[T]$ be monic of degree $d \geq 2$. Suppose that 
  $K_f$ is path-connected and that $f$ is
  not linearly conjugated to a Chebyshev polynomial or a negative Chebyshev
  polynomial. Let $S \subset K_f$ be a finite set. There exists a
  finite topological tree  $\cT\subset K_f$
  with $S\subset \cT$ and  $\mathsf{d}_\infty(\cT) < 0$.
\end{proposition} 

The proof of Proposition~\ref{prop:cappreper} needs the following
fact about the transfinite diameter. We were unable to locate a
reference for this elementary fact.
Let $\widehat\IC$ denote the Riemann sphere $\IC\cup\{\infty\}$. 

\begin{lemma}
  \label{lem:capacity} Let $K' \subsetneq K\subset \mathbb{C}$ be two
  compact subsets, such that $\widehat{\IC}\ssm K$ and
  $\widehat\IC\ssm K'$ are simply connected and such that
  $\mathsf{d}_\infty(K)>-\infty$.  Then
  $$\mathsf{d}_\infty(K') < \mathsf{d}_\infty(K).$$
\end{lemma}
\begin{proof}
  We have
  $\mathsf{d}_\infty(K')\le \mathsf{d}_\infty(K)$ since $K'\subset K$.
  The burden is to prove strict inequality. 

  After rescaling, we may assume  $\mathsf{d}_\infty(K) = 0$.
  If $K'$ is a singleton, then $\mathsf{d}_{\infty}(K')=-\infty$ and
  we are done.

  Now assume that $K'$ contains at least $2$ points.
  Both $\widehat{\mathbb{C}}\ssm K$ and $\widehat{\mathbb{C}}\ssm K'$
  are simply connected and their complements contained at least $2$
  points.
  So the Riemann Mapping Theorem provides biholomorphisms $r\colon
  \widehat{\mathbb{C}}\ssm K \rightarrow \Delta, r' \colon
  \widehat{\mathbb{C}}\ssm K' \rightarrow \Delta$,
  where $\Delta = \{z\in \IC : |z|<1\}$.
  
  For $z$ large enough we have $r(z) = z + O(1)$
  by the Fekete--Szeg\H{o} Theorem,
  Theorem 5.5.2~\cite{Ransford}, and as $K$ has transfinite diameter $1$.
  
  If  $\mathsf{d}_\infty(K')=0$ as well,
  then $r'(z) = z + O(1)$. In this case, $f= r\circ (r')^{-1}$
  maps $\Delta$ to $\Delta$ and satisfies $f(0) = 0$ and $f'(0) = 1$. By the
  Schwarz Lemma,  $f$ is the identity. Thus $r$ is the
  restriction of $r'$ to $\widehat{\mathbb{C}}\ssm K$.
  This is impossible; indeed, 
  let $z\in K\ssm K'$. There is $w\in \widehat
  \IC\ssm K$ with $r(w)=r'(z)$.
  But $r(w)=r'(w)$ by the argument above. So $z=w$ as $r'$ is
  injective, a contradiction. 
\end{proof}

\begin{proof}[Proof of Proposition~\ref{prop:cappreper}]
  Since $K_f$ is in particular connected, Theorem 9.5~\cite{Milnor}
  implies that $\widehat\IC\ssm K_f$ is simply connected.

  Every finite set $S\subset K_f$ is contained in a finite  topological
  tree contained entirely in $K_f$. 
  If $f$ is postcritically finite, the tree
  can be chosen to be the allowable hull of the points in $S$, see
  Section~\ref{sec:trees}. 
  For general $f$, we fix a point $s_0\in S$ as a root  of the
  tree. Then 
  we continue by successively attaching the points in $S\ssm\{s_0\}$ using
  arcs inside $K_f$ to $s_0$. If an arc hits part of the tree
  under constructing we stop  and designate the intersection point as  a
  new   vertex.
  The resulting graph is connected. Its number of vertices equals
  the number of edges plus one. So it is a tree, say $\cT$, containing
  $S$ and with $\cT\subset K_f$. 

  Theorem \ref{thm:chebpropn} implies $\cT \subsetneq K_f$. Since
  $f$ is monic, we have $\mathsf{d}_\infty(K_f)=0$, see
  Theorem~6.5.1~\cite{Ransford}. Hence $\mathsf{d}_\infty(K_S)< 0$ by
  Lemma \ref{lem:capacity}, as desired.
\end{proof}


\section{Constructing the Power Series}
\label{sec:powerseries}
\subsection{Integrality and Postcritically Bounded Maps}

Let $F$ be a number field and let $p$ be a prime number.
Let $f\in F[T]$ be monic of degree $d\ge 2$.

Fix any  finite field  extension $F'/F$ such that derivative $f'$ splits
completely in $F'[T]$. We write $c_1,\ldots,c_{d-1}\in F'$ for the
critical points of $f$, with multiplicities. The barycenter of $f=T^d+
f_1T^{d-1}+\cdots $,
given by (\ref{def:barycenter}), equals
\begin{equation}
  \label{def:barycenter2}
  -\frac{f_1}{d}=  \frac{c_1+\cdots+c_{d-1}}{d-1}.
\end{equation}

We say that $f$ is \textit{postcritically bounded} above $p$
if for all places $v$ of $F'$ above $p$ and all $i\in
\{1,\ldots,d-1\}$ the forward orbit
$\{f^{(n)}(c_i) : n\in\IN_0\}$ is bounded with respect to
$|\cdot|_v$. This condition depends only on $f$ and is independent of
the choice of $F'$. 

For example, $f$ is postcritically bounded above $p$ if it is
postcritically finite. That is, if $\{f^{(n)}(c_i):n\in\IN\}$ is
finite for all $i$.

Let $R$ be a subring of $F$ suppose $f\in R[T]$.
Let $A\in R[T]$ be a monic polynomial.
Suppose $A$ splits in a field extension $K$ of $F$.
Hence $A=(X-x_1) \cdots (X-x_D)$ with $x_1,\dots,x_D\in K$. For all
$k\in\IN_0$ we define
\begin{equation}
  \label{def:Ak}
  A_k = (X-f^{(k)}(x_1))\cdots (X-f^{(k)}(x_D)).
\end{equation}
The fundamental theorem on elementary symmetric polynomials implies
that the coefficients of $A_k$ lie in $R$. Moreover, $A_k$ is
independent of the choice of $K$.

The goal of this section is to generalize a congruence
condition by Dimitrov~\cite{dimitrov:SZ} going back to older work of
Smyth~\cite{smyth:coloringproof} and Arnold~\cite{arnold:eulerfermat}.
The authors later proved a  version~\cite{hs:2021lower} suitable for
 dynamical systems coming from certain polynomials. 

Our main new tool is an integrality result of
Epstein~\cite{epstein:12} on postcritically bounded
polynomials of prime power degree. So we will assume that $d$, the
degree of $f$, is $p^e$ for some $e\in\IN$. 

Let us define
\begin{equation}
  \label{def:OFp}
  \cO_{F,p} = \{x\in F : |x|_v\le 1 \text{ for all }v\in M_F\text{
    with }v\mid p\}.
\end{equation}
Let $\cO_F$ denote the ring of algebraic integers of $F$. 
Then $\cO_{F,p}$ is the localization of 
$\cO_F$  outside the (finite) union of all
prime ideals of $\cO_F$ containing $p$.

In order to state the congruence condition we will require the
following technical notion. 

\begin{defi}
  \label{def:Ffadmissible}
  Let $f\in F[T]$ have degree $d\ge 2$ which is a power of $p$.
  We say that the pair $(k,l)\in\IN^2$ is $(F,f)$-admissible if
  the following two properties hold true.
  \begin{enumerate}
  \item [(i)] For all $a\in \cO_{F,p}$ we have $a^{d^{k-1}} \equiv
    a^{d^{l-1}}\imod {p\cO_{F,p}}$.
  \item[(ii)] We have $f^{(k-1)}(0) - f^{(l-1)}(0) \in p\cO_{F,p}$. 
  \end{enumerate}
\end{defi}

All diagonal pairs are $(F,f)$ admissible.
But such pairs are useless in our application.
We will require an admissible pair with distinct entries. Further
down we exhibit the existence of such pairs, see
Lemmas~\ref{lem:conglemma_new_cor} and \ref{lem:conglemma_unicritical}.

\begin{propo}
  \label{prop:congruence}
  Let $F$ be a number field and $p$ a prime number. Let $d=p^e$ for
  some $e\in\IN$. Let $f\in F[T]$ be monic and of degree $d$. Suppose
  that $f$ is postcritically bounded above $p$ and that its barycenter
  is in $\cO_{F,p}$. Then the following hold true.
  \begin{enumerate}
  \item [(i)] We have $f\in \cO_{F,p}[T]$. Moreover,
    $A_k\in \cO_{F,p}[X]$ for all monic $A\in \cO_{F,p}[X]$ and
    all $k\in\IN_0$.
  \item[(ii)] Let $(k,l)\in\IN^2$ be an $(F,f)$-admissible pair. Then
    $A_k\equiv A_l \imod{dp \cO_{F,p}[X]}$
    for all monic $A\in \cO_{F,p}[X]$.  
  \end{enumerate}  
\end{propo}

\subsection{Congruence Conditions}

In this section we prove Proposition~\ref{prop:congruence}. We take a
somewhat different approach from the authors's earlier special
case~\cite{hs:2021lower}. Apart from systematically using Epstein's
work~\cite{epstein:12}, we emphasize computations that are reminiscent
to the theory of Witt vectors.

Let $p$ denote a prime number.

We repeat here a well-known congruence relation which we use freely
throughout this section.
\begin{lemma}
  \label{lem:congruencedp}
  Let $R$ be a ring and let $n\in\IN_0$. Then
  \begin{equation*}
    (x+py)^{p^{n}} \equiv x^{p^{n}} \imod {p^{n+1} R}
  \end{equation*}
  for all $x,y\in R$. 
\end{lemma}
\begin{proof}
  We claim that $(1+p T)^{p^n} \equiv 1 \imod{p^{n+1}\IZ[T]}$ for all $n\ge
  0$. Granted this claim,  $p^{n+1}$
  divides ${p^n\choose k} p^k$ for all $k\in \{1,\ldots,p^n\}$ and this
  implies the lemma. The claim clearly holds for $n=0$.
  If it holds for $n$, then $(1+pT)^{p^n} = 1+ p^{n+1} B$ for some
  $B\in\IZ[T]$. Then it holds for $n+1$ as
  $(1+pT)^{p^{n+1}}  = (1+p^{n+1}B)^p \equiv 1
  \imod{p^{n+2}\IZ[T]}$ by Fermat's Little Theorem. 
\end{proof}

We come to a lemma that is reminiscent of calculations involving the ghost
components of Witt vectors. Let $D\ge 1$ be an integer and
let $X_1,\ldots,X_D$ denote independents.
We abbreviate the tuple $(X_1,\ldots,X_D)$ with the symbol
$\boldsymbol X$ and set ${\boldsymbol X}^k = (X_1^k,\ldots,X_D^k)$ for
all $k\in\IN_0$.

\begin{lemma}
  \label{lem:witt}
  Let $Q\in \IZ[{\boldsymbol X}]$. There
  exist $Q_0,Q_1,\ldots \in \IZ[{\boldsymbol X}]$ with the following
  property. For all $k\in \IN_0$ we have 
  \begin{equation}
    \label{eq:witt}
    Q({\boldsymbol X}^{p^k}) = \sum_{i=0}^k p^i Q_i^{p^{k-i}}.
  \end{equation}
  Moreover, if $Q$ is a symmetric polynomial, then we may assume
  that $Q_0,Q_1,\ldots$ are symmetric symmetric polynomials.
\end{lemma}
\begin{proof}
  We set $Q_0=Q$ and obtain (\ref{eq:witt}) for $k=0$.
  Assume $k\ge 0$ and that $Q_0,\ldots,Q_k$ have been constructed and satisfy
  (\ref{eq:witt}). We  will construct $Q_{k+1}$ and establish   (\ref{eq:witt})
  for $k+1$.
  We substitute ${\boldsymbol X}^p$ into (\ref{eq:witt}) and find 
  \begin{equation*}   
    Q({\boldsymbol X}^{p^{k+1}}) = \sum_{i=0}^k p^i Q_i({\boldsymbol X}^p)^{p^{k-i}}.
  \end{equation*}
  
  By Fermat's Little Theorem,  there is $R_i\in
  \IZ[{\boldsymbol X}]$ with $Q_i({\boldsymbol X}^p) = Q_i^p + pR_i$
  for all $i\in \{0,\ldots,k\}$.
  Lemma~\ref{lem:congruencedp} yields $Q_i({\boldsymbol X}^p)^{p^{k-i}} =
  Q_i^{p^{k+1-i}} + p^{k+1-i} S_{i}$ for some
  $S_{i}\in\IZ[{\boldsymbol X}]$. 
  Thus
  \begin{equation*}   
    Q({\boldsymbol X}^{p^{k+1}}) 
    = 
    \sum_{i=0}^k p^i Q_i^{p^{k+1-i}} + \sum_{i=0}^k p^{k+1}S_{i}.
  \end{equation*}
  We set $Q_{k+1} = \sum_{i=0}^k S_{i} \in \IZ[{\boldsymbol X}]$ and
  find that (\ref{eq:witt}) holds for $k+1$. 
  
  We have established that $Q_0,Q_1,\ldots$ as in (\ref{eq:witt})
  exist with integral coefficients. The
  existence is unique since 
  $Q_k = \frac{1}{p^k} (Q({\boldsymbol X}^{p^k})- (Q_0^{p^k}+ p
  Q_1^{p^{k-1}}+\cdots + p^{k-1} Q_{k-1}^p))$ for $k\in\IN_0$.
  Moreover, this equality and an induction on $k$ implies that all $Q_k$ are
  symmetric if $Q$ is symmetric. 
\end{proof}

For  $j\in \{0,\ldots,D\}$ we let
$e_j\in\IZ[X_1,\ldots,X_D]$
denote the elementary symmetric polynomial of degree $j$. 

\begin{lemma}
  \label{lem:ppowermodulp}
  Let $R$ be an integral domain.
  Let $k,l\in\IN$ such that
  \begin{equation}
    \label{eq:apkapl}
    a^{p^{k}}\equiv a^{p^{l}} \imod {pR} \quad\text{for all}\quad
    a\in R.
  \end{equation}
  Let $A\in R[X]$ be monic with $A = \prod_{i=1}^D (X-x_i)$
  where $x_1,\ldots,x_D$ lie in a splitting field of $A$ over the field of
  fractions of $R$.
  Then for all symmetric  $Q\in\IZ[{\boldsymbol X}]$ and all $j\in\IN_0$ we
  have 
  $    Q(x_1^{p^{k+j}},\ldots,x_D^{p^{k+j}}),
  Q(x_l^{p^{l+j}},\ldots,x_D^{p^{l+j}})\in R$ and 
  \begin{equation*}
    Q(x_1^{p^{k+j}},\ldots,x_D^{p^{k+j}}) \equiv
    Q(x_1^{p^{l+j}},\ldots,x_D^{p^{l+j}}) \imod{p^{j+1}R}
  \end{equation*}
  for all $j\in\IN_0$. 
\end{lemma}
\begin{proof}
  We apply Lemma~\ref{lem:witt} to $Q$ and obtain symmetric polynomials
  $Q_0,Q_1,\ldots \in \IZ[{\boldsymbol X}]$.

  Let $i\in \IN_0$ and let
  us define $a_i = Q_i(x_1,\ldots,x_D)$, member of the said
  splitting field of $A$. 
  Each $Q_i$ is an integral polynomial in $e_1,\ldots,e_D$ by the fundamental
  theorem on symmetric polynomials.
  By hypothesis, $A\in R[X]$ is monic so its coefficients are up-to a
  sign equal to 
  $e_j(x_1,\ldots,x_D)\in R$ for all $j$.  Thus we find
  $a_i\in R$ for all $i\in\IN_0$.

  By hypothesis, $a_i^{p^{k}} \equiv a_i^{p^{l}} \imod
  {p R}$ for all $i\ge 0$. Thus Lemma~\ref{lem:congruencedp} yields
  $a_i^{p^{k+j-i}} \equiv a_i^{p^{l+j-i}} \imod
  {p^{j-i+1} R}$ for all integers $i,j$ with $0\le i\le j$.
  We multiply with
  $p^{i}$ and find
  $p^{i} a_i^{p^{k+j-i}} \equiv p^{i} a_i^{p^{l+j-i}} \imod
  {p^{j+1}R}$ which implies  the middle congruence in
  \begin{alignat}1
    \label{eq:applyghost}
  \sum_{i=0}^{k+j} p^i     a_i^{p^{k+j-i}}
    \equiv \sum_{i=0}^j p^i a_i^{p^{k+j-i}} 
    \equiv \sum_{i=0}^j p^i a_i^{p^{l+j-i}} 
    \equiv \sum_{i=0}^{l+j} p^i a_i^{p^{l+j-i}} \imod{p^{j+1}R}
  \end{alignat}
  for all  $j\ge 0$. 
  We use the identity  Lemma~\ref{lem:witt} and specialize
  ${\boldsymbol X}$ to $(x_1,\ldots,x_D)$. Thus sum on the left on
  (\ref{eq:applyghost}) equals
  $Q(x_1^{p^{k+j}},\ldots,x_D^{p^{k+j}})$ while the sum
  on the right is
  $Q(x_1^{p^{l+j}},\ldots,x_D^{p^{l+j}})$. This completes the
  proof. 
\end{proof}

\begin{lemma}
  \label{lem:symmetricppower}
  Let $R$ be an integral domain of characteristic unequal to $p$ and
  let  $e\in\IN_0$. 
  Let $Q \in R[\boldsymbol X]$ such that $\frac{\partial Q}{\partial
    X_j} \in p^e R[\boldsymbol X]$ for all $j\in \{1,\ldots,D\}$.
  There exist $\tilde Q_0,\ldots,\tilde Q_e \in
  R[\boldsymbol X]$ such that
  \begin{equation}
    \label{eq:diffQpowerpconclusion}
    Q = \sum_{i=0}^e p^i \tilde Q_i(\boldsymbol{X}^{p^{e-i}}).
  \end{equation}
  Moreover, if $Q$ is a symmetric polynomial, then we may assume that
  $\tilde Q_0,\ldots,\tilde Q_e$ are symmetric polynomials. 
\end{lemma}
\begin{proof}
  Our proof is by induction on $e$.
  If $e=0$, then the lemma follows with $\tilde Q_0=Q$. 
 
  Now suppose $e\ge 1$.
  Let 
  $q_\iota X_1^{\iota_1}\cdots
  X_D^{\iota_D}$
  be a term of $Q$ whose
  corresponding
   exponent vector
  $\iota=(\iota_1,\ldots,\iota_D)$  is not an integral multiple of $p$.
  Our hypothesis implies $\iota_j q_\iota\in p^e R$ for all $j$. As the tuple
  $p^e,\iota_1,\ldots,\iota_D$ is coprime,
  we have $a_0 p^e+ a_1 \iota_1+\cdots +a_D \iota_D =1$ for some
  $a_0,\ldots,a_D\in\IZ$.
  Multiplying by $q_\iota$
  gives $q_\iota \in p^e R$.

  So we can write $Q =\tilde Q(\boldsymbol{X}^p) + p^e \tilde Q_e$
  where $\tilde Q,\tilde Q_e\in
  R[\boldsymbol X]$ and all exponent vectors $\iota$ that arise in
  $\tilde Q_e$
  satisfy $\iota\not\in p\IZ^D$.
  If $Q$ is symmetric, then so is $\tilde Q_e$
  because if $\iota\not\in p\IZ^D$, then no
  permutation of $\iota$ is a multiple of $p$. Furthermore, $\tilde Q$ is
  then also symmetric.

  Our hypothesis
  implies that all partial derivatives
  $\frac{\partial (Q-p^e\tilde Q_e))}{\partial X_j}=\frac{\partial \tilde Q(\boldsymbol{X}^p)}{\partial X_j}$ lie in $p^e
  R[\boldsymbol X]$. The chain rule implies
  \begin{equation*}
    p\frac{\partial \tilde Q}{\partial X_j}(\boldsymbol{X}^p) X_j^{p-1}
    \in p^e R[\boldsymbol X]
  \end{equation*}
  for all $j$. As $R$ is a domain of characteristic unequal to $p$, we
  may cancel out $p$. Hence
  $\frac{\partial \tilde Q}{\partial X_j}(\boldsymbol{X}^p) \in p^{e-1}
  R[\boldsymbol X]$  and thus
  $\frac{\partial \tilde Q}{\partial X_j} \in p^{e-1}
  R[\boldsymbol X]$ for all $j$. We apply induction on $e$ to $\tilde
  Q$ and find
  \begin{equation*}
    \tilde Q = \sum_{i=0}^{e-1} p^i \tilde Q_i (\boldsymbol
    X^{p^{e-1-i}})
  \end{equation*}
  for some $\tilde Q_0,\ldots, \tilde Q_{e-1}\in R[\boldsymbol X]$. The
  induction argument also implies that all $\tilde Q_i$ are symmetric
  if $\tilde Q$ is symmetric. We
  substitute $\boldsymbol X^p$ and find
  \begin{equation*}
    Q = \tilde Q(\boldsymbol{X}^p) + p^e \tilde Q_e =
    \sum_{i=0}^{e-1}  p^i \tilde Q_i(\boldsymbol{X}^{p^{e-i}}) + p^e
    \tilde Q_e
  \end{equation*}
  which is just (\ref{eq:diffQpowerpconclusion}).
\end{proof}

Let $F$ be a number field and $p$ a prime number. 
For the remainder of this section $f\in F[T]$ denotes a polynomial as
in  Proposition~\ref{prop:congruence}.
So $f$ is monic of degree $d=p^e$ for some
 $e\in\IN$. Moreover, $f$
is postcritically bounded above $p$ with barycenter
$\overline c\in \cO_{F,p}$ as in (\ref{def:barycenter2}).

The next lemma relies on the work of  Epstein~\cite{epstein:12}.

\begin{lemma}
  \label{lem:epstein}
  The following hold true.
  \begin{enumerate}
  \item [(i)]  We have $f\in \cO_{F,p}[T], f(\overline c)\in
    \cO_{F,p},$     and $f\equiv T^d+ f(\overline c)- \overline c^d
    \imod {p \cO_{F,p}[T]}$.
  \item[(ii)] The derivative satisfies $f' \in d \cO_{F,p}[T]$. 
  \item[(iii)] Let $x,y\in \cO_{F,p}[T]$. Then $f(x+py) \equiv f(x) \imod{dp
      \cO_{F,p}[T]}$. 
  \item[(iv)] We have
    $A_k\in \cO_{F,p}[X]$ for all monic $A\in \cO_{F,p}[X]$ and
    all $k\in\IN_0$.
  \end{enumerate}
\end{lemma}
\begin{proof}
  We fix a finite field extension $F'/F$ over which $f'$ splits,
  \textit{i.e.}, $f' = d(T-c_1)\cdots(T-c_{d-1})$ with
  $c_1,\ldots,c_{d-1}\in F'$ the critical points of $f$.
  Let us abbreviate $\mathbf{c}=(c_1,\ldots,c_{d-1})$ and
  $b=f(\overline{c})$.
  
  We use the notation of Epstein's Section 2~\cite{epstein:12}. 
  There
  \begin{equation}
    \label{eq:Fcb}
    F_{\mathbf{c},b} = T^d + \sum_{k=1}^{d-1} (-1)^{d-k}
    \frac{d}{k}s_{d-k} T^k +
    b - \overline{c}^d - \sum_{k=1}^{d-1} (-1)^{d-k}
    \frac{d}{k}s_{d-k} \overline{c}^k
  \end{equation}
  where $s_{k}$ are elementary symmetric polynomials in
  $c_1,\ldots,c_{d-1}$ of  degree $k$.
  Then $f = F_{\mathbf{c},b}$.
  Indeed, the derivatives of both polynomials
  are equal. So their difference is
  constant. Equality then follows from  $F_{\mathbf{c},b}(\overline
  c)=b=f(\overline c)$.
  
  By hypothesis, the barycenter $\overline c$ of $f$
  is integral with respect to all places
  of $F$ above $p$. Moreover, the forward
  orbit under $f$ of each $c_i$ is bounded with
  respect all  places of $F'$ above $p$.
  By Epstein's Theorem~1~\cite{epstein:12},
  $b\in \cO_{F',p}$ and the critical points satisfy
  $c_1,\ldots,c_{d-1}\in\cO_{F',p}$.

  We observe that $\cO_{F',p}\cap F = \cO_{F,p}$. Therefore, $b\in
  \cO_{F,p}$ and all elementary symmetric polynomials in
  $c_1,\ldots,c_{d-1}$ lie in $\cO_{F,p}$. In other words,
  $s_1,\ldots,s_{d-1}\in\cO_{F,p}$.  The expression for
  $f=F_{\mathbf{a},b}$ in (\ref{eq:Fcb}) and $\overline c\in \cO_{F,p}$
  imply
  $f\in \cO_{F,p}[T]$;
  note that $d/k\in\cO_{F,p}$ for all $k\in \{1,\ldots,d\}$ as $d$ is
  a power of $p$.
  Moreover, if $k\in \{1,\ldots,d-1\}$,
  then $d/k$ is a multiple of $p$ in $\cO_{F,p}$. So $f\equiv T^d + b-
  \overline c^d \imod{p\cO_{F,p}[T]}$ and all claims in (i) hold true.

  Claim  (ii) follows as
  $f=F_{\mathbf{c},b}$, since $s_1,\ldots,s_{d-1}\in \cO_{F,p},$
  and by (\ref{eq:Fcb}).

  For claim (iii) we apply the congruence condition in
  Lemma~\ref{lem:congruencedp} and (\ref{eq:Fcb}). Suppose
  $x,y\in\cO_{F,p}[T]$. Let $k\in
  \{1,\ldots,d\}$. We write $k=p^n m$
  with $0\le n\le e$ and $p\nmid m$. So 
  $(x+py)^{p^n} \equiv  x^{p^n} \imod{ p^{n+1} \cO_{F,p}[T]}$.
  By taking  the $m$-th power we find
    $(x+py)^{k} \equiv  x^k \imod{ p^{n+1} \cO_{F,p}[T]}$. The quotient
    $\frac dk$ lies in $p^{e-n}\cO_{F,p}$, and so 
  $\frac dk (x+py)^{k} \equiv  \frac dk x^k \imod{ p^{e+1}
    \cO_{F,p}[T]}$.
  So the equivalence holds modulo $dp=p^{e+1}$. 
  Recall that all $s_{d-k}$ in (\ref{eq:Fcb}) lie in $\cO_{F,p}$.
  Thus
  $\frac dk s_{d-k} (x+py)^{k} \equiv \frac  dk s_{d-k} x^k \imod{ dp
    \cO_{F,p}[T]}$ for all $k\in \{1,\ldots,d-1\}$. Moreover,
  $(x+py)^d \equiv x^d \imod{dp\cO_{F,p}[T]}$.
  Taking the appropriate sum over all
  $k$  and using (\ref{eq:Fcb}) implies part (iii).

  Part (iv) follows from part (i) and the comment around (\ref{def:Ak}). 
\end{proof}

For all $k\in\IN$ we define
\begin{equation}
  \label{def:alphak}
  \alpha_k = f^{(k)}(0) \in \cO_{F,p} \quad\text{and}\quad \alpha_0 = 0.
\end{equation}
By Lemma~\ref{lem:epstein}(i) we find
\begin{equation*}
  f(0)\equiv   \alpha_1 \equiv f(\overline c) - \overline c^d \imod{p\cO_{F,p}}. 
\end{equation*}

\begin{lemma}
  \label{lem:fkcongpsqr}
  The following hold true.
  \begin{enumerate}
  \item [(i)]
    We have $f^{(k)} \equiv T^{d^{k}}+\alpha_k
    \imod{p \cO_{F,p}[T]}$ for all $k\in\IN_0$.
  \item[(ii)] We have $\alpha_k \equiv f(0)^{d^{k-1}}+\cdots +
    f(0)^d + f(0) \imod{p\cO_{F,p}}$ for all $k\in\IN$. 
  \item[(iii)]
    We have $f^{(k)} \equiv f(T^{d^{k-1}}+\alpha_{k-1})
    \imod{dp\cO_{F,p}[T]}$ for all $k\in\IN$.
  \item[(iv)] Suppose $k,l\in \IN$ satisfies $k\le l$
    and $\alpha_{k-1} \equiv
    \alpha_{l-1} \imod{p\cO_{F,p}}$. Then
    $$f^{(l)}(T)\equiv f^{(k)}(T^{d^{l-k}}) \imod{dp \cO_{F,p}[T]}.$$
  \end{enumerate}
\end{lemma}
\begin{proof}
  We prove (i) by induction on $k$. The claim is clear for $k=0$.
  For $k\ge 1$ we have
  \begin{equation*}
    f^{(k)} \equiv f(f^{(k-1)}) \equiv (f^{(k-1)})^d + \alpha_1 \imod {p\cO_{F,p}[T]}
  \end{equation*}
  by Lemma~\ref{lem:epstein}(i). The induction hypothesis states
  $f^{(k-1)}\equiv T^{d^{k-1}} + \alpha_{k-1} \imod {p\cO_{F,p}[T]}$.
  As $d$ is a power of $p$, we conclude 
  \begin{equation*}
    f^{(k)} \equiv  T^{d^k} + \alpha_{k-1}^d + \alpha_1 \imod {p\cO_{F,p}[T]}.
  \end{equation*}
  By definition of $\alpha_k$, we have $f^{(k)}(0)= \alpha_k$. 
  Therefore, part (i) follows.

  This  argument also yields $\alpha_k \equiv \alpha_{k-1}^d + \alpha_1
  \imod{p\cO_{F,p}}$.
  A short induction leads to claim (ii).  

  Part (iii) follows from part (i) in the case $k-1\ge 0$ together with
  Lemma~\ref{lem:epstein}(iii).

  To prove part (iv) we first observe
  \begin{equation*}
    f^{(l)}(T)\equiv f(T^{d^{l-1}}+\alpha_{l-1}) \imod {dp\cO_{F,p}[T]}    
  \end{equation*}
  by part (iii) applied to $l$. Second, the hypothesis $\alpha_{k-1}\equiv
  \alpha_{l-1}\imod{p\cO_{F,p}}$ and   Lemma~\ref{lem:epstein}(iii)
  imply
  \begin{equation*}
    f^{(l)}(T)\equiv
    f(T^{d^{l-1}}+\alpha_{l-1}) \equiv     f(T^{d^{l-1}}+\alpha_{k-1}) \imod {dp\cO_{F,p}[T]}.
  \end{equation*}
  Third,  after substituting $T^{d^{l-k}}$ into
  (iii) of the current lemma, we find
  $f^{(k)}(T^{d^{l-k}}) \equiv
  f(T^{d^{l-1}}+\alpha_{k-1})\imod{dp\cO_{F,p}[T]}$.
  The chain of equivalences implies (iv).
\end{proof}

Below we use Definition~\ref{def:Ffadmissible} on admissible pairs.

\begin{lemma}
  \label{lem:conglemma_newnew}
  Suppose $(k,l)\in\IN^2$ is an $(F,f)$-admissible pair.
  Then all monic  $A\in \cO_{F,p}[X]$ satisfy
  \begin{equation*}
    A_k \equiv A_l \imod {dp \cO_{F,p}[X]}.
  \end{equation*}
\end{lemma}
\begin{proof}
  Without loss of generality we have $k\le l$.
  By Lemma~\ref{lem:epstein} we have $f\in \cO_{F,p}[T]$ and
  $f' \in d\cO_{F,p}[T]$.
  The chain rule implies ${f^{(k)}}' 
    = f'(f^{(k-1)}) \cdots f'(f) f' \in d^k \cO_{F,p}[T]$. 
  
  Let $s\in
  \{1,\ldots,D\}$ be fixed for the moment. Recall
  $\boldsymbol X = (X_1,\ldots,X_D)$ and that $e_s$ is an elementary
  symmetric polynomial in $X_1,\ldots,X_D$. We define the symmetric polynomial
  \begin{equation}
    \label{def:Qmainconglemma}
    Q  =  e_s(f^{(k)}(X_1),\ldots,f^{(k)}(X_D))
    \in \cO_{F,p}[\boldsymbol X].
  \end{equation}
  For all $j\in \{1,\ldots,D\}$ we have
  \begin{equation*}
    \frac{\partial Q}{\partial X_j} = \frac{\partial e_s}{\partial
      X_j}(f^{(k)}(X_1),\ldots,f^{(k)}(X_D)) {f^{(k)}}'(X_j) \in d^k
    \cO_{F,p}[\boldsymbol X]. 
  \end{equation*}
  By Lemma~\ref{lem:symmetricppower}
  there exist $\tilde Q_0,\ldots,\tilde Q_{ek} \in \cO_{F,p}[\boldsymbol
  X]$ such that
  \begin{equation*}
    Q = \sum_{i=0}^{ek} p^i \tilde Q_i(\boldsymbol X^{d^k/p^i});
  \end{equation*}
  recall that $d=p^e$. Furthermore, we may take $\tilde
  Q_0,\ldots,\tilde Q_{ek}$ to
  be symmetric as $Q$ is symmetric.

  By the fundamental theorem on symmetric polynomials,
  there exist $P_0,\ldots,P_{ek} \in \cO_{F,p}[\boldsymbol X]$ such that
  $\tilde Q_i = P_i (e_1(\boldsymbol X),\ldots, e_D(\boldsymbol X))$ for all $i$.
  In particular,
  \begin{equation}
    \label{eq:QsumpnQn}
    Q 
    =\sum_{i=0}^{ek} p^i P_i(e_1(\boldsymbol
    X^{d^k/p^i}), \ldots e_D(\boldsymbol
    X^{d^k/p^i})). 
  \end{equation}

  Let $A\in \mathcal{O}_{F,p}[X]$ be  as in the hypothesis
  with $A = (X-x_1)\cdots(X-x_D)$
  where $x_1,\ldots,x_D$ lie in a splitting field of $A$ over $F$.
  Let $\boldsymbol x$ denote the tuple $(x_1,\ldots,x_D)$.
  
  By Definition~\ref{def:Ffadmissible}(i) we have
  $a^{p^{e(k-1)}}\equiv a^{p^{e(l-1)}} \imod {p\cO_{F,p}}$ for
  all $a\in \cO_{F,p}$. We apply Lemma~\ref{lem:ppowermodulp}  to the
  elementary symmetric polynomial $e_s$ and with $k,l$ chosen
  as $e(k-1),e(l-1)$, respectively, and $j\ge 0$ chosen as $e-i$.  Thus 
  \begin{equation*}
    e_s(\boldsymbol x^{p^{ek-i}})
    \equiv
    e_s(\boldsymbol x^{p^{el-i}})
    \imod{p^{e-i+1}\cO_{F,p}}
  \end{equation*}
  for all $i\in \{0,\ldots,e\}$. This is true for
  all $s\in \{1,\ldots,D\}$. We
  insert into $P_i$ and find
  \begin{equation*}
    P_i(e_1(\boldsymbol x^{p^{ek-i}}), \ldots e_D(\boldsymbol
    x^{p^{ek-i}}))
    \equiv
    P_i(e_1(\boldsymbol x^{p^{el-i}}), \ldots e_D(\boldsymbol
    x^{p^{el-i}}))
    \imod{p^{e-i+1}\cO_{F,p}}
  \end{equation*}
  for all $i\in \{0,\ldots,e\}$. We multiply with $p^i$, so
  \begin{equation*}
    p^i P_i(e_1(\boldsymbol x^{p^{ek-i}}), \ldots e_D(\boldsymbol
    x^{p^{ek-i}}))
    \equiv
    p^iP_i(e_1(\boldsymbol x^{p^{el-i}}), \ldots e_D(\boldsymbol
    x^{p^{el-i}}))
    \imod{p^{e+1}\cO_{F,p}}
  \end{equation*}
  for all $i \in \{0,\ldots,e\}$.  If $i \in \{e+1,\ldots,ek\}$, then this congruence
   trivially holds modulo
   $p^{e+1}\cO_{F,p}=dp\cO_{F,p}$.
 
  We insert $\boldsymbol{x}$ in
  (\ref{eq:QsumpnQn}) and work modulo $dp$ to get
  \begin{equation}
    \label{eq:QxQxdlk}
    \begin{aligned}
    Q(\boldsymbol x) &\equiv
    \sum_{i=0}^{e} p^i P_i(e_1(\boldsymbol
    x^{d^k/p^i}), \ldots e_D(\boldsymbol
    x^{d^k/p^i})) \imod{dp \cO_{F,p}} \\
    &\equiv\sum_{i=0}^{e} p^i P_i(e_1(\boldsymbol
    x^{d^l/p^i}), \ldots e_D(\boldsymbol
      x^{d^l/p^i})) \imod{dp \cO_{F,p}} \\
    &\equiv Q(\boldsymbol x^{d^{l-k}})\imod{dp \cO_{F,p}}; 
    \end{aligned}
  \end{equation}  
  the final equivalence follows from substituting
  ${\boldsymbol x}^{d^{l-k}}$ into  (\ref{eq:QsumpnQn}).

  The choice (\ref{def:Qmainconglemma}) yields
  \begin{equation*}
    Q(\boldsymbol x) = e_s(f^{(k)}(x_1),\ldots,f^{(k)}(x_D)) \in \cO_{F,p}.
  \end{equation*}

  By Definition~\ref{def:Ffadmissible}(ii), 
  Lemma~\ref{lem:fkcongpsqr}(iv) implies $f^{(l)}(T) \equiv
  f^{(k)}(T^{d^{l-k}}) \imod{dp\cO_{F,p}[T]}$. Therefore,
  \begin{equation*}
    Q(\boldsymbol X^{d^{l-k}})
    \equiv  e_s(f^{(l)}(X_1),\ldots,f^{(l)}(X_D))
    \imod{dp\cO_{F,p}[\boldsymbol X]}.
  \end{equation*}
  We substitute $\boldsymbol x$ and retrieve
  \begin{equation*}
    Q(\boldsymbol x^{d^{l-k}}) \equiv
    e_s(f^{(l)}(x_1),\ldots,f^{(l)}(x_D))
    \imod{dp \cO_{F,p}}. 
  \end{equation*}
  Finally, (\ref{eq:QxQxdlk}) implies
  \begin{equation*}
    e_s(f^{(k)}(x_1),\ldots,f^{(k)}(x_D)) \equiv
    e_s(f^{(l)}(x_1),\ldots,f^{(l)}(x_D))
    \imod{dp \cO_{F,p}}
  \end{equation*}
  for all $s\in \{1,\ldots,D\}$. Hence $A_k\equiv A_l
  \imod{\cO_{F,p}[T]}$, as desired. 
\end{proof}

\begin{proof}[Proof of Proposition~\ref{prop:congruence}]  
  The first statement in part (i) follows from
  Lemma~\ref{lem:epstein}(i) and the second statement follows
  from   Lemma~\ref{lem:epstein}(iv).
  Part (ii) follows from Lemma~\ref{lem:conglemma_newnew}.
\end{proof}

As indicated in further up, we prove the existence of
$(F,f)$-admissible pairs outside the diagonal.

\begin{lemma}
  \label{lem:conglemma_new_cor}
  Let $F,p,$ and $f$ be as in Proposition~\ref{prop:congruence}.
  Let $\mathfrak{e}\in \IN$ be an upper bound for the ramification
  indices of all prime ideals of $\cO_F$ above $p$.
  Let $\mathfrak{f}\in\IN$ be a multiple
  of the residue degrees of all prime ideals of $\cO_F$ above $p$. We set
  \begin{equation}
    \label{eq:choicekl}
    k=1+\left\lceil\frac{\log \mathfrak{e}}{\log
        d}\right\rceil\quad\text{and}\quad
    l = k + p\mathfrak f,
  \end{equation}
  using the ceiling function.
  Then $(k,l)$ is an $(F,f)$-admissible pair.
\end{lemma}
\begin{proof}
  Let $\mathfrak{P}$ be a prime ideal of $\cO_F$ containing $p$. 
  The natural homomorphism $\cO_{F}/\mathfrak{P}\rightarrow
  \cO_{F,p}/\mathfrak{P}\cO_{F,p}$ is an isomorphism.
  
  We first verify part (i) of Definition~\ref{def:Ffadmissible}.
  By the choice of $\mathfrak f$, we
  have  $a^{p^{\mathfrak f}} - a\in \mathfrak P\cO_{F,p}$ for all
  $a\in\cO_{F,p}$. In particular, $a^{d^{\mathfrak f}}-
  a\in\mathfrak{P}\cO_{F,p}$ as $d$ is a power of $p$.

  Our choice of $k$ in (\ref{eq:choicekl}) satisfies
  $d^{k-1}\ge \mathfrak{e}$. Hence $(a^{d^{\mathfrak f}}-a)^{d^{k-1}}
  \in \mathfrak{P}^{\mathfrak{e}}\cO_{F,p}$ for all $a\in \cO_{F,p}$.
  This statement holds for all prime ideals $\mathfrak{P}\subset \cO_{F}$ above
  $p$.
  By hypothesis, the
  ramification index of every prime ideal of $\mathfrak{P}$ above
  $p$
  is at most $\mathfrak{e}$.
  So we conclude 
  $(a^{d^{\mathfrak f}}-a)^{d^{k-1}}\in {p\cO_{F,p}}$. As
  $d^{k-1}$ is a power of $p$, we find
  \begin{equation*}
    0\equiv (a^{d^{\mathfrak f}}-a)^{d^{k-1}}\equiv a^{d^{\mathfrak
        f+k-1}}+(-a)^{d^{k-1}}
    \equiv a^{d^{\mathfrak f+k-1}}-a^{d^{k-1}}\imod{p\cO_{F,p}}.
  \end{equation*}
  A simple induction that involves raising $a^{d^{\mathfrak
      f+k-1}}\equiv a^{d^{k-1}}\imod{p\cO_{F,p}}$ to the $d^{\mathfrak
    f}$-th power provides
  \begin{equation}
    \label{eq:adjpower}
    a^{d^{j\mathfrak f + k-1}} \equiv a^{d^{k-1}} \imod {p\cO_{F,p}}
  \end{equation}    
  for all $j\in\IN_0$. Taking $j=p$ and recalling
  $l = k+p\mathfrak f$ from (\ref{eq:choicekl}) yields
  \begin{equation*}
    a^{d^{l-1}} \equiv a^{d^{k-1}} \imod {p\cO_{F,p}}
  \end{equation*}
  for all $a\in\cO_{F,p}$.
  Thus condition (i) of Definition~\ref{def:Ffadmissible} holds for $(k,l)$.
    
  Now we verify part (ii) of Definition~\ref{def:Ffadmissible}.
  We already know from Proposition~\ref{prop:congruence}
  that $f$ has coefficients in
  $\cO_{F,p}$. In particular, $f^{(n)}(0)\in \cO_{F,p}$ for all $n\ge 0$.  
  By our notation (\ref{def:alphak}) and by Lemma~\ref{lem:fkcongpsqr}(ii) we have
  \begin{equation*}
    f^{(l-1)}(0) - f^{(k-1)}(0) =     \alpha_{l-1}-\alpha_{k-1} 
    \equiv \sum_{i=k-1}^{l-2} f(0)^{d^i} \imod{ p\cO_{F,p}}.
  \end{equation*}
  The number of terms in the final sum is $l-k = p\mathfrak f$. 
  We split this sum up into $p$ parts, each of length $\mathfrak f$,
  to find
  \begin{equation*}
    f^{(l-1)}(0) - f^{(k-1)}(0) \equiv \sum_{j=0}^{p-1}
    \sum_{i=0}^{\mathfrak f-1} f(0)^{d^{i + j\mathfrak f + k-1}}
    \imod{p\cO_{F,p}}.
  \end{equation*}
  It follows from  (\ref{eq:adjpower}) that, modulo $p$, the inner sum is
  independent of $j$. Thus
  \begin{alignat*}1
    f^{(l-1)}(0) - f^{(k-1)}(0) &\equiv \sum_{j=0}^{p-1}
                                \sum_{i=0}^{\mathfrak f-1} f(0)^{d^{i+k-1}}
                              \equiv p \sum_{i=0}^{\mathfrak f-1} f(0)^{d^{i+k-1}}
                              \equiv 0\imod{p\cO_{F,p}}.
  \end{alignat*}
  Therefore, Definition~\ref{def:Ffadmissible}(ii) 
  holds for our choice of $(k,l)$.
\end{proof}

Soon we will obtain a second criterion for the choice of $(k,l)$ for
quadratic polynomials. But first
we prove Lemma~\ref{lem:critperimplieshyp} from the
introduction.

\begin{proof}[Proof of Lemma~\ref{lem:critperimplieshyp}]
  By hypothesis, the forward orbit of any point in $\pco_{\ge 1}(f)$
  contains a critical point of $f$.
  As there are finitely many critical points, the forward orbit of any
  critical point is
  finite. Hence  $f$ is postcritically finite.
  Let $z\in \pco_{\ge 1}(f)$ be arbitrary. By what we just proved, the
  forward orbit of $z$ contains a periodic point of period $n$, say.
  And by hypothesis,
  the forward orbit of this periodic point contains a critical point
  $c$, say. Then $c$ is also periodic of period $n$.
  Moreover,
  $(f^{(n)})'(c) = f'(f^{(n-1)}(c))\cdots f'(c)=0$ by the chain rule.
  So $c$ lies in an attracting periodic orbit. Thus
  $c\not\in J_f$ by Lemma~4.6~\cite{Milnor} and  hence $z\not\in
  J_f$.
  We conclude that $f$
  is hyperbolic by Lemma~\ref{lem:hyperboliccharacterization}. This
  completes the proof of (i).

  It is well known that $b$ is an algebraic integer if $T^d+b$
  is postcritically finite; see the first paragraph in the 
  Appendix~\cite{epstein:12}.
  Suppose now that $0$ is periodic. 
  It follows from the proof of 
  Proposition~A.1~\cite{epstein:12} that
  $b$ is the root of a monic polynomial $\Gamma \in \IZ[X]$ (depending only
  on the period length of $0$)
  with derivative $\Gamma' \equiv 1 \imod d$.
  Therefore, 
  $\IQ(b)/\IQ$ is unramified above all prime divisors of $d$.
  This deduction is  attributed to Gleason and Adler if $d=2$.
  We
  conclude part (ii).  
\end{proof}

Next we consider  degree $2$ polynomials $T^2+b$ such that the unique
critical point $0$ is periodic. 
By Lemma~\ref{lem:critperimplieshyp}(ii),
the critical value $b$ is an algebraic integer 
and the number field $\IQ(b)/\IQ$ is
unramified above $2$. 

The next lemma uses a result of Buff, Floyd, Koch, and
Parry~\cite{BFKP:Gleason}.

\begin{lemma}
  \label{lem:conglemma_unicritical}
  Let $b$ be an algebraic number and $f=T^2+b\in F[T]$ with $F=\IQ(b)$.
  Suppose that the
  critical point $0$ of $f$ is $f$-periodic of period length $n\ge 1$.
  Then $(1,1+n)$ is an $(F,f)$-admissible pair.
\end{lemma}
\begin{proof}
  We briefly recall some properties of the $n$-th Gleason polynomial
  $G_n$, here $n\in\IN$, 
  see Section~1~\cite{BFKP:Gleason} for a definition.
   Its roots 
  are the  $b\in\IC$ for
  which $0$ is $(T^2+b)$-periodic of period length $n$.
  Each $G_n$ is monic, square-free, and has coefficients in $\IZ$.

  A folklore conjecture states that all Gleason polynomials are
  irreducible in $\IZ[X]$. By contrast, the reduction
  $\tilde{G}_n\in \IF_2[X]$ modulo $2$ usually factors.
  
  By the proof of Proposition~A.1~\cite{epstein:12}, the polynomial
  $G_n$ divides a polynomial whose derivative is equivalent to $1$
  modulo $2$. So the discriminant of $\tilde{G}_n$ is  $1\in \IF_2$
  and thus the discriminant of $G_n$ is odd. 
  It follows that  $\IZ[b]$ is an order in $\mathcal{O}_{F}$ with odd index
  $[\mathcal{O}_F : \IZ[b]]$. By the Dedekind--Kummer Theorem, the ideal
  $2\mathcal{O}_F$ factors into pairwise distinct prime ideals whose
  residue degrees equal the degrees of the irreducible factors of
  $\tilde{G}_n$ in $\IF_2[X]$.

  Buff, Floyd, Koch, and Parry computed these degrees  in their
  Theorem~1.6~\cite{BFKP:Gleason}. Indeed, if $n$ is odd, then all
  irreducible factors of $\tilde{G}_n$ in $\IF_2[X]$ have degree
  $n$. If $n$ is even, then any irreducible factor of $\tilde{G}_n$
  in $\IF_2[X]$ has degree $n$ or $n/2$.

  Let $\mathfrak{P}$ be a prime ideal of $\cO_F$ containing $2$.
  Then the residue degree divides
  $n$ by the Dedekind--Kummer Theorem.
  As in the first part of the proof of
  Lemma~\ref{lem:conglemma_new_cor},
  we find
  $a^{p^n} - a \in \mathfrak{P}  \mathcal{O}_{F,2}$ for all $a\in
  \mathcal{O}_{F,2}$
  and all
  prime ideals  $\mathfrak{P}\subset \cO_F$ containing $2$.
  As $F/\IQ$ is unramified above $2$ we conclude
  $a^{p^n} \equiv  a \imod{2\mathcal{O}_{F,2}}$ for all $a\in
  \mathcal{O}_{F,2}$.
  Thus (i) of Definition~\ref{def:Ffadmissible}
  holds for $(k,l)=(1,1+n)$.

  We have $f^{(n)}(0) - f^{(0)}(0) = f^{(n)}(0) = 0$ as $0$ has period
  length $n$. Therefore, Definition~\ref{def:Ffadmissible} part (ii)
  holds
  for $(1,1+n)$ as well.
\end{proof}

\subsection{Application to Power Series}
\label{sec:apppowerseries}

Let $F$ be a number field, let $p$ be a prime number, and let $d=p^e$
with $e\in\IN$.

Suppose $R$ is a subring of $F$.
Consider two polynomials $A,B\in R[X]$ 
with $A$ monic of degree $D$ and $\deg B\le D$.
The quotient $B/A$ lies in $F(X)$. However, we can
also consider $B/A$ as a power series in $1/X$. Indeed,
we write $A = X^D + a_{1}X^{D-1}+\cdots + a_D$ and $B = b_0 X^D +
\cdots + b_D$. Then
\begin{equation*}
  \frac{B}{A} = \frac{X^{-D}B}{X^{-D}A}= \frac{ b_0 + b_{1} X^{-1}+\cdots + b_D X^{-D}}
  { 1 + a_{1} X^{-1}+\cdots + a_D X^{-D}}.
\end{equation*}
Numerator and denominator are certainly elements of $R[[1/X]]$.
But the denominator has constant term $1$, so it is a unit in
$R[[1/X]]$. Thus we can consider $B/A\in R[[1/X]]$.
We have $b_0=0$, \textit{i.e.}, $\deg B < D$, if and only if
$B/A\in X^{-1}R[[1/X]]$. 

According to Newton, the power series
\begin{equation*}
  \Phi_d = \sum_{k=0}^\infty {1/d\choose k} X^k \in \IQ[[X]]
\end{equation*}
satisfies $\Phi^d = 1+X$.

The power seides $\Phi_d$ does not have integral
coefficients as its linear term is $1/d$. A key insight in Dimitrov's proof of the
Schinzel--Zassenhaus Conjecture was the integrality property
$\Phi_p(p^2 X) \in \IZ[[X]]$. We refer to
Proposition~2.6~\cite{dimitrov:SZ} for a proof. Dimitrov's integrality
property extends easily to prime powers.
\begin{lemma}
  \label{lem:Phidintegral}
  We have $\Phi_d(dpX)\in  1+X\IZ[[X]]$. 
\end{lemma}
\begin{proof}
  The coefficient of $\Phi_d(dpX)$ for $X^k$  
  equals
  \begin{equation*}
    {1/d\choose k}(dp)^k = \frac{d^{-1} (d^{-1} - 1) \cdots (d^{-1} -
      (k - 1))d^k}{k!} p^k
    = \frac{(1-d)(1-2d)\cdots (1-(k-1)d)}{k!}p^k
  \end{equation*}
  for all integers $k\ge 0$. It is well-known that the
  $p$-adic absolute value of the factorial
  satisfies $|k!|_p \ge p^{-k/(p-1)}$.
  So
  $\bigl|{1/d\choose k}(dp)^k \bigr|_p \le |p^k/k!|_p \le
  p^{-k+k/(p-1)}=p^{-k(p-2)/(p-1)}\le 1$.
  Let $\ell\not=p$ be a prime.
  So $\ell\nmid d$. 
  We fix $a\in\IZ$ with $da \equiv 1
  \imod{\ell^{n}}$
  where $n$ is the exponent of $\ell$ in $k!$.
  Then $(1-d)\cdots(1-(k-1)d)\equiv d^k a(a-1)\cdots(a-(k-1))\equiv 0\imod
  {\ell^n}$ as ${a\choose k}$ is an integer.
  Thus $\left|{1/d\choose k}(dp)^k\right|_\ell\le 1$ and
  hence $\Phi_d(dpX)\in\IZ[[X]]$. 
\end{proof}

For a finite set of places $S$ of $F$ containing all archimedean
places, we let
\begin{equation*}
  \cO_{F,S} = \{ x\in F : |x|_v \le 1\text{ for all places $v$ of $F$
    with $v\not\in S$}\}
\end{equation*}
denote the ring of $S$-integers of $F$. 

For example, if $S$ equals the set of all
archimedean places of $F$, then $\cO_{F,S}$ is the ring $\cO_F$ of algebraic
integers of $F$. If $S$ is a finite set of places of $F$
containing all archimedean places but no finite places above $p$, then
$\cO_{F,S}\subset \cO_{F,p}$.

Recall that admissible pairs were introduced in Definition~\ref{def:Ffadmissible}.
\begin{propo}
  \label{prop:maincongruence}
  Let $F$ be a number field and $p$ a prime number.
  Let $S$ be a finite set of places of $F$ containing all archimedean places
  and disjoint from all finite places dividing $p$.
  Let $f\in \cO_{F,S}[T]$ be monic and of degree a positive power of
  $p$.
  Suppose that $f$ is postcritically bounded
  above $p$ and that its barycenter is in $\cO_{F,p}$.
  Let $(k,l)$ be an $(F,f)$-admissible pair. 
  If $A\in \cO_{F,S}[X]$ is monic there
  exists $B\in \cO_{F,S}[X]$ with $\deg B < dD$ such that
  $\phi = \Phi_d\left(dp\frac{B}{A_k^d}\right) \in 1+X^{-1}\cO_{F,S}[[X^{-1}]]$
  and 
  \begin{equation}
    \label{eq:powerseriesequation}
    \phi^d = \frac{A_l}{A_k}.
  \end{equation} 
\end{propo}
\begin{proof}
  Let $A\in\cO_{F,S}[X]$ be monic of degree $D\ge 0$.
  By
  the second example before this proposition we have
   $\cO_{F,S}\subset \cO_{F,p}$. So $A\in \cO_{F,p}[X]$. 
 
  Proposition~\ref{prop:congruence}(ii)
  implies $A_l\equiv A_k \imod{dp\cO_{F,p}[X]}$. We multiply with
  $A_k^{d-1}$ and fix $B\in \cO_{F,p}[X]$ such that
  $A_l A_k^{d-1} = A_k^d + dp B$. 
  Certainly, $\deg B\le dD$. But we even have
  $\deg B <dD$, as the leading terms of the monic polynomials
  $A_lA_k^{d-1}$ and $A_k^d$,
  both in degree $dD$, cancel out when taking  the difference.
  Thus we established the desired degree bound for $B$.

  We claim that $B$ has coefficients in $\cO_{F,S}$.
  Our hypothesis $f\in\cO_{F,S}[T]$ implies 
  $A_k,A_l\in \cO_{F,S}[X]$, see the argument around (\ref{def:Ak}). 
  So $dpB\in \cO_{F,S}[X]$. Let $v$ be a  place  of $F$ with
  $v\not\in S$. If $v\mid p$, then $|B|_v\le 1$ since $B\in
  \cO_{F,p}[X]$; here $|\cdot|_v$ denotes the Gauss norm.
  If $v\nmid p$, then $|p|_v=1$ and so $dp B\in\cO_{F,S}[X]$ implies
  $|B|_v = |dp B|_v \le 1$. Therefore $|B|_v\le 1$ for all places $v$
  of $F$ with $v\not\in S$. Our claim follows. 
  
  We divide $A_l A_k^{d-1} = A_k^d + dp B$ by $A_k^d$ and get
  \begin{equation*}
    \frac{A_l}{A_k} = 1 + dp \frac{B}{A_k^d}.
  \end{equation*}
  As per the discussion before the  beginning of this proposition,
  we consider this identity to
  take place in 
  $\cO_{F,S}[[1/X]]$.
  Note that $\deg B < dD = \deg{A_k^d}$ implies $dpB/A_k^d \in X^{-1}
  \cO_{F,S}[[1/X]]$.  

  By Lemma~\ref{lem:Phidintegral},   $\Phi_d(dpB/A_k^d)$ is a
  well-defined element of 
  $1+X^{-1}\cO_{F,S}[[X^{-1}]]$.  The proposition now follows from
  \begin{equation*}
    \Phi_d\left(dp\frac{B}{A_k^d}\right)^d =
    1+dp\frac{B}{A_k^d}=\frac{A_l}{A_k}.\qedhere
  \end{equation*}
\end{proof}  

We retain the notation of the last proposition. Suppose $A =
(X-x_1)\cdots(X-x_D)$ splits in a given finite field extension $K/F$. Let
$\sigma \colon K\rightarrow\IC$ be a field embedding.
We may apply $\sigma$ to power series and rational functions in the
usual way. In particular, 
we consider $\sigma(\phi) \in \IC[[1/X]]$. 
Then (\ref{eq:powerseriesequation}) can be restated as
\begin{equation}
  \label{eq:phipproduct}
  \sigma(\phi)^d = \prod_{i=1}^D \frac{  1-\sigma(f^{(l)}(x_i))/X}{1-\sigma(f^{(k)}(x_i))/X}.
\end{equation}
The right-hand side of (\ref{eq:phipproduct}) is a
rational function in $\IC(X)$.

Recall that $\phi = \Phi_d(dpB/A_k^d)$.
If $z\rightarrow\infty$, then
$\sigma(\frac{B}{A_k^d})(z)\rightarrow 0$
since $\deg B < \deg A_k^d$. But $\Phi_d$ converges in a
neighborhood of $0$, and so  $\sigma(\phi)(z)$ converges
for $z\in\IC$ outside a sufficiently large closed disk. Thus
$\sigma(\phi)$ represents a holomorphic map outside a large enough
closed disk
with $\lim_{z\rightarrow\infty}\sigma(\phi)(z)=1$. 
Finally, by the Monodromy Theorem, $\sigma(\phi)$ extends to any
simply connected domain $\Omega\subset\IC$
that avoids the poles and zeros 
$\{\sigma(f^{(k)}(x_i)), \sigma(f^{(l)}(x_i)) : 1\le i\le D\}$
of (\ref{eq:phipproduct}).


\section{Proof of Main Results}
\label{sec:proofs}

We come to a central result of this paper. Recall that the barycenter
of a monic polynomial is given by (\ref{def:barycenter}) and that
admissibility is defined in Definition~\ref{def:Ffadmissible}.

\begin{theorem}
  \label{thm:main}
  Let $F$ be a number field, let $p$ be a prime number, and let
  $f\in\cO_F[T]$ be monic of degree $d=p^e$
  for some $e\in\IN$.  We suppose
  \begin{itemize}
  \item that $f$ is postcritically finite,
  \item that $\sigma_0(f) \in\IC[T]$ is hyperbolic for some $\sigma_0\colon
    F\rightarrow\IC$, and
  \item that the
    barycenter of $f$ is an algebraic integer.
  \end{itemize}
  Let $(k,l)$ be an $(F,f)$-admissible pair with $k<l$. 
  Then there exists $c=c(F,f,k,l)>0$  with the following property. 
  Let $\alpha\in \overline F$ be an algebraic integer, 
  then 
  \begin{enumerate}
  \item [(i)] either $\alpha$ is $f$-wandering with
    \begin{equation}
      \label{eq:thmmainbound}
      \max_{v\in M^\infty_{F(\alpha)}}
      \lambda_{f,v}(\alpha)\ge \frac{c}{[F(\alpha):F]^{k}} 
    \end{equation}
  \item[(ii)] or $\alpha$ is $f$-preperiodic with
    \begin{equation*}
       d^{\frac{\mathrm{preper}(\alpha)}{k}-1} 
      \le
       [F(\alpha):F]
      \quad\text{and}\quad
      {\mathrm{per}(\alpha)} \le      
      (l-k)[F(\alpha):F]. 
    \end{equation*}
  \end{enumerate}
\end{theorem}

In the notation of the theorem, if  $\sigma_0(f)$ is hyperbolic for one embedding
$\sigma_0\colon F\rightarrow\IC$, then it is hyperbolic for all
embeddings by Lemma~\ref{lem:pcfhyperbolicgalois}.

\begin{remark}
  \label{rmk:choicekl}
  Let $F,p,$ and $f$ be as in Theorem~\ref{thm:main}.
  Then Proposition~\ref{prop:congruence}
  and  Lemma~\ref{lem:conglemma_new_cor}
  apply to $f$.
  
  Lemma~\ref{lem:conglemma_new_cor} provides a possible choice of admissible pair $(k,l)$
  with $k<l$. We recall its conclusion.
  Let $\mathfrak{e}\in \IN$ be an upper bound for the ramification
  indices of all prime ideals of $\cO_F$ above $p$.
  Let $\mathfrak{f}\in\IN$ be a multiple
  of the residue degrees of all prime ideals of $\cO_F$ above $p$. Then
  \begin{equation*}
    k=1+\left\lceil\frac{\log \mathfrak{e}}{\log
        d}\right\rceil\quad\text{and}\quad
    l = k + p\mathfrak f,
  \end{equation*}
  are members of  an $(F,f)$-admissible pair $(k,l)$. 

  Let us consider some special cases. If $F/\IQ$ is unramified above
  $p$,  then we may take $\mathfrak e=1$.   
  With this choice we have
  $k=1+\left\lceil\frac{\log\mathfrak e}{\log d}\right\rceil =1$. Thus
  $(1,1+p\mathfrak f)$ is an $(F,f)$-admissible pair if $F/\IQ$ is unramified.
  
  If furthermore $F=\IQ$, then $\mathfrak f=1$ is a possible choice.
  In this case $(1,1+p)$ is an $(\IQ,f)$-admissible pair.
\end{remark}

We come to a variation of the theorem above for preperiodic points but
without the hyperbolic hypothesis. Let $F'/F$ be a field extension.
Two polynomials $f,g\in F[T]$ are called \textit{$F'$-linearly
conjugate} if there exist $a\in {F'}^\times$ and $b\in F'$ with $f =
(g(aT+b)-b)/a$.

\begin{theorem}
  \label{thm:mainpreper}
  Let $F$ be a number field, let $p$ be a prime number, and let
  $f\in\cO_F[T]$ be monic of degree $d=p^e$
  for some $e\in\IN$.  We suppose
  \begin{itemize}
  \item that $f$ is postcritically finite,
  \item that $f$ is not $\overline F$-linearly conjugate  to a
    Chebyshev polynomial or a
    negative Chebyshev polynomial, and 
  \item that the barycenter of $f$ is an algebraic integer.
  \end{itemize}
  Let $(k,l)$ be an $(F,f)$-admissible pair with $k<l$. 
  Let $\alpha\in \overline F$ be $f$-preperiodic, then 
  \begin{equation}
    \label{eq:thmpreperbounds}
    d^{\frac{\mathrm{preper}(\alpha)}{k}-1} \le [F(\alpha):F] 
    \quad\text{and}\quad
    \mathrm{per}(\alpha)\le (l-k)[F(\alpha):F].
  \end{equation}
\end{theorem}

In the special case $f=T^2+b$ below we use the work of Buff, Floyd,
Koch, and Parry~\cite{BFKP:Gleason}.

\begin{theorem}
  \label{thm:mainpreperdeg2}
  Let $b$ be an algebraic number and set $f=T^2+b\in F[T]$ where
  $F=\IQ(b)$.
  Suppose that $0$ is $f$-periodic of period length $n$. 
  Let $\alpha\in \overline F$ be $f$-preperiodic, then 
  \begin{equation}
    \label{eq:thmpreperbounds_deg2}
    2^{{\mathrm{preper}(\alpha)}-1} \le [F(\alpha):F] 
    \quad\text{and}\quad
    \mathrm{per}(\alpha)\le n [F(\alpha):F].
  \end{equation}  
\end{theorem}

Let us mention here the work of Anderson, Hamblen, Poonen, and
Walton~\cite{AHPW:18} on arboreal representations. In their Theorem
7.3, they show that, for fixed and suitable $x\in F$,  the normal closure of
$F((f^{(n)})^{-1}(a))$ over $F$ has
unbounded wild ramification above $p$ as $n$ increases.
Their $f$ has degree $p$, but is not assumed to be postcritically
finite. 
This induces a lower bound for the field degree in question.

\begin{remark}
  \begin{enumerate}
  \item[(i)] Let $f=T^2+b$ and assume that
    $0$ is $f$-periodic. Lemma~\ref{lem:critperimplieshyp}
    implies the following two statements. The polynomial $f$ is
    postcritically finite and $b$ is an algebraic integer in a number
    field $F$. Moreover,
    $\sigma(f)$ is hyperbolic for all $\sigma\colon F\rightarrow \IC$.
    
  \item [(ii)]   Neither theorem applies to $T^2-2 \in \IQ[T]$,
    the second Chebyshev polynomial $\Psi_2$ according to the
    normalization as around (\ref{eq:chebynorm}). The polynomial $T^2-2$ is
  postcritically finite and its barycenter is $0$. So
  $(k,l)=(1,3)$ is an $(\IQ,f)$-admissible pair by Remark~\ref{rmk:choicekl}.
  The critical point $\alpha=0$ is preperiodic with
  $\mathrm{preper}(0)=2$ and does not satisfy
  (\ref{eq:thmpreperbounds}). Note that $T^2-2$ is not hyperbolic as the
  critical point $0$ lies in the Julia set $[-2,2]$. So
  Theorem~\ref{thm:main} does not apply either.
  \item[(iii)] Neither theorem applies to
    $f=T^5+10T^4+45T^3+110T^2+145T+80$.
    Then $f$ is $\IC$-linearly conjugate to the
    monic Chebyshev polynomial
    $\Psi_5=T^5-5T^3+5T$ with the normalization as around
    (\ref{eq:chebynorm}). 
    Indeed, $f = (\Psi_5(aT+b)-b)/a$ with
    $a=\sqrt{-1}$ and $b=2\sqrt{-1}$.
    But $f$ is not $\IQ$-linearly conjugate to $\Psi_5$. 
  \end{enumerate}  
\end{remark}

We come to a conclusion on the canonical height. It holds for
all algebraic numbers and not just algebraic integers.

\begin{corollary}
  \label{cor:canhgtlb}
  Let $F,p,f,k,$ and $l$ be as in Theorem~\ref{thm:main}.
  Then there exists $c=c(F,f,k,l)>0$  with the following property. 
  Let $\alpha\in \overline F$,
  then 
  \begin{enumerate}
  \item [(i)] either  $\alpha$ is $f$-wandering with
    \begin{equation*}
      \hat h_f(\alpha) \ge \frac{c}{[F(\alpha):F]^{k+1}} 
    \end{equation*}
  \item[(ii)] or $\alpha$ is $f$-preperiodic with
    \begin{equation*}
      d^{\frac{\mathrm{preper}(\alpha)}{k}-1} \le [F(\alpha):F] 
      \quad\text{and}\quad
      {\mathrm{per}(\alpha)} \le      (l-k)[F(\alpha):F].
    \end{equation*}
  \end{enumerate}
\end{corollary}

We begin with the proof of these various results. 

Recall that $\widehat\IC$ is the Riemann sphere $\IC\cup\{\infty\}$
and that
$\lambda_f\colon\IC\rightarrow[0,\infty)$ was defined in (\ref{def:lambdafcomplex}). 
The following lemma is well-known.

\begin{lemma}
  \label{lem:constructKsmalllambda}
  Let $f\in \IC[T]$ be monic of degree $d\ge 2$. Suppose that all
  critical points of $f$ lie in $K_f$. Let $t\ge 0$. The set
  \begin{equation*}
    K = \{z \in \IC : \lambda_f(z)\le t\}
  \end{equation*}
  is compact, non-empty, and satisfies $\mathsf{d}_\infty(K)= t$.
  Moreover, $\widehat\IC\ssm K$ is simply connected.
\end{lemma}
\begin{proof}
  By Theorem 9.5~\cite{Milnor}, $J_f$ and $K_f$ are connected. 
  Moreover, there is a biholomorphic mapping $\phi\colon \IC\ssm
  K_f\rightarrow \IC\ssm\overline\Delta$; where $\overline\Delta$ is
  the closed unit disk.  Let
  $\psi=\phi^{-1}\colon \IC\ssm\overline\Delta\rightarrow\IC\ssm K_f$.
  The asymptotic expansion of
  $\psi(w)$ for $w$ near $\infty$ is of the form $w+O(1)$, see the
  discussion on pages 96 -  98 \textit{loc.cit.}
  Then  the discussion in Definition 9.6~\textit{loc.cit.} yields
  $\log|\phi(z)| = \lambda_f(z)$ for all $z\in\IC\ssm K_f$.
  Therefore, $K$ from the hypothesis equals
  $K_f\cup \{z\in \IC\ssm K_f : \log|\phi(z)|\le t \}$. Its complement
  $\IC\ssm K$
  equals $\psi(\{w\in\IC: |w|>e^{t}\})$.
  This is the same as $\psi_t(\IC\ssm\overline\Delta)$ where
  $\psi_t(w) = \psi(e^t w)$.
  Asymptotically, $\psi_t$ grows like $e^tw+O(1)$. 
  So the conformal radius of $K$ with respect to $\infty$ equals $e^t$.  
  Theorem~5.2.3~\cite{Ransford} combined with the
  Fekete--Szeg\H{o} Theorem, Theorem 5.5.2~\textit{loc.cit.}, yields  $\mathsf{d}_\infty(K)=t$.
  Finally, $\widehat\IC\ssm K$
  is simply connected as $\widehat\IC\ssm K$ is biholomorphic to 
  $\{w\in \IC : |w|>t \}\cup\{\infty\}$. 
\end{proof}

\begin{remark}
  Let $f$ be as in Lemma~\ref{lem:constructKsmalllambda} and
  $\overline f$ the polynomial after applying complex conjugation of
  $f$. Then $\lambda_{f}(z) = \lambda_{\overline f}(\overline z)$ for
  all $z\in\IC$. So
  $\{z\in \IC : \lambda_f(z)\le t\}$ is the image under complex
  conjugation of $\{z\in \IC : \lambda_{\overline f}(z)\le t\}$. These
  two sets have equal transfinite diameter.
\end{remark}

For the next lemma we recall the notation $A_k$ introduced
in (\ref{def:Ak}) for a monic polynomial $A$, it also depends on fixed
$f$. 

\begin{lemma}
  \label{lem:AkAlpreperiodic}
  Let $f\in F[T]$ be monic of degree at least $2$. Let
  $A\in F[X]$ be an irreducible and monic polynomial. Suppose
  $k,l\in\IN_0$ such that $k\le l$ and such that
  $A_k$ and $A_l$ are not coprime.
  \begin{enumerate}
  \item [(i)]  There exists $r\in
    \{1,\ldots,\deg A\}$ such that all roots $\alpha\in\overline F$ of
    $A$ satisfy $f^{(k)}(\alpha)=f^{((l-k)r+k)}(\alpha)$. 
  \item[(ii)] If $k<l$, then $\alpha$ is $f$-preperiodic. Moreover,
    $\mathrm{preper}(\alpha)\le k$ and
    $\mathrm{per}(\alpha)\le (l-k)[F(\alpha):F]$.
  \end{enumerate}
\end{lemma} 
\begin{proof}
  We follow the argument in Proposition 4.3~\cite{hs:2021lower}.
  
  By hypothesis and (\ref{def:Ak}), there exist roots 
  $\alpha,\alpha'\in\overline F$ of $A$ with
  $f^{(k)}(\alpha')=f^{(l)}(\alpha)$.
  As $A$ is irreducible in $F[X]$ there exists $\sigma\in
  \mathrm{Gal}(\overline F/F)$ with $\sigma(\alpha)=\alpha'$.
  Therefore,
  \begin{equation}
    \label{eq:fklsigma}
    f^{(k)}(\sigma(\alpha)) = f^{(l)}(\alpha).
  \end{equation}

  We claim that $\sigma^r (f^{(k)}(\alpha)) = f^{((l-k)r+k)}(\alpha)$
  for all $r\in\IN_0$.

  The proof is by induction on $r$, the case $r=0$ being trivial. We
  assume $r\ge
  1$. By the induction hypothesis we find
  \begin{equation*}
    \sigma^r (f^{(k)}(\alpha)) =
    \sigma\bigl(\sigma^{r-1}(f^{(k)}(\alpha))\bigr)
    = \sigma\bigl(f^{((l-k)(r-1)+k)}(\alpha)\bigr)
    =\sigma\bigl(f^{((l-k)(r-1))}(f^{(k)}(\alpha))\bigr). 
  \end{equation*}
  Note that $\sigma$ acts trivially on the coefficients of $f$.
  Therefore, 
  \begin{equation*}
    \sigma^r (f^{(k)}(\alpha)) 
    =f^{((l-k)(r-1))}(f^{(k)}(\sigma(\alpha))).
  \end{equation*}
  We use (\ref{eq:fklsigma}) and find 
  \begin{equation*}
    \sigma^r (f^{(k)}(\alpha)) 
    =f^{((l-k)(r-1))}(f^{(l)}(\alpha))=f^{((l-k)(r-1)+l)}(\alpha).
  \end{equation*}
  Our claim follows from $(l-k)(r-1)+l = (l-k)r+k$. 
  
  As we let $i\in\IN_0$ vary, there are at most $[F(\alpha):F]$
  possible values for $\sigma^i(\alpha)$. By the Pigeonhole Principle
  there exist
  $i,j\in \{0,\ldots,[F(\alpha):F]\}$ with $i>j$ and
  $\sigma^i(\alpha)=\sigma^j(\alpha)$. We find
  $\sigma^r(\alpha)=\alpha$ with $r=i-j\in \{1,\ldots,[F(\alpha):F]\}$.
  
  With this choice for $r$ our claim yields
  $f^{(k)}(\alpha)=f^{(k)}(\sigma^r (\alpha))=
  \sigma^r (f^{(k)}(\alpha))= f^{((l-k)r+k)}(\alpha)$.
  This equality must hold for all roots of the
  irreducible $A$ since the Galois group acts transitively on such
  roots.
  Therefore, part (i) holds.

  Part (ii) follows from part (i).
\end{proof}

\begin{remark}
  \label{rem:AkAlpreperiodic}
  Let us keep the notation of Lemma~\ref{lem:AkAlpreperiodic}.
  Let $\alpha\in\overline F$ be a root of $A$.

  The proof of Lemma~\ref{lem:AkAlpreperiodic}
    shows that we can take $r\in \{1,\ldots,r_0\}$ where
    \begin{equation*}
      r_0 =
      \max_{  \sigma\in \mathrm{Gal}(\overline F/F)}
      \min \{r\ge 1 : \sigma^r \in \mathrm{Gal}(\overline
      F/F(\alpha))\}.
    \end{equation*}    
    We note that $r_0\le [F(\alpha):F]$ by the argument given in the
    proof.
    In some situations we can expect a better bound. For
    example, if $F(\alpha)/F$ is a Galois extension we can take $r_0$ to
    be the exponent of $\mathrm{Gal}(F(\alpha)/F)$. In particular, if
    the  Galois group of $F(\alpha)/F$ is
    isomorphic to a 
    power of $\IZ/2\IZ$, then we may take $r_0=2$.
\end{remark}

We now prove Theorem~\ref{thm:main} by induction on $D=[F(\alpha):F]$, where
 $\alpha\in\overline F$ is an algebraic integer. We do the
induction step and the base case $\alpha\in F$ simultaneously.
Let $(k,l)$ be an $(F,f)$-admissible pair with $k<l$. 

Suppose $v$ is an archimedean place of $F(\alpha)$.
We first relate $\lambda_{f,v}(\alpha)$, per the definition
(\ref{def:lambdafcomplex}), to the local canonical height on the
complex numbers as in (\ref{def:lambdafv}).
Let $\sigma\colon
F(\alpha)\rightarrow \IC$ be one of the at most 2 homomorphisms 
attached to $v$. By definition we have 
\begin{equation*}  
  \lambda_{f,v}(\alpha) = \lim_{n\rightarrow\infty}
  \frac{\log^+ |f^{(n)}(\alpha)|_v}{d^n}
  =  \lim_{n\rightarrow\infty}
  \frac{\log^+|\sigma(f)^{(n)}(\sigma(\alpha))|}{d^n}
  =\lambda_{\sigma(f)}(\sigma(\alpha)).
\end{equation*}
Therefore,
\begin{equation}
  \label{eq:twodynamicalhouses}
  \max_{v\in M_{F(\alpha)}^\infty} \lambda_{f,v}(\alpha)=
  \max_{\sigma\colon F(\alpha)\rightarrow \IC}
  \lambda_{\sigma(f)}(\sigma(\alpha)). 
\end{equation}
And we will continue to work with the right-hand side of
(\ref{eq:twodynamicalhouses}) as opposed to the left-hand side which
appears in (\ref{eq:thmmainbound}).

We write $c_1$ and $c_2$ for positive constants that depend on the
given data $F,f,k,$ and $l$,
but not on $\alpha$. We will fix $c>0$ to be sufficiently small   in terms of
$c_1$ and $c_2$.
It is important to keep in mind that $k$ and $l$ are part of the fixed
data. 
They do not depend on our algebraic integer $\alpha$.

If the inequality (\ref{eq:thmmainbound}) holds, then $\alpha$ is $f$-wandering. 
So to prove the theorem if suffices to assume 
\begin{equation}
  \label{eq:lambdasigmafsmallworkinghyp}
  \max_{\sigma \colon  F(\alpha)\rightarrow \IC}
  \lambda_{\sigma(f)}(\sigma(\alpha))< \frac{c}{[F(\alpha):F]^{k}}=\frac{c}{D^k}
\end{equation}
and show that $\alpha$ satisfies conclusion (ii) of Theorem~\ref{thm:main}.

As $f$ is postcritically finite, it is postcritically bounded above
$p$. 
Moreover, by hypothesis $f$  has coefficients and barycenter in $\cO_F$.
We take $S$ to be the set of
archimedean places of $F$, so $\cO_{F,S}=\cO_F$.
All assumptions of
Proposition~\ref{prop:maincongruence} are satisfied.

Consider the
$F$-minimal polynomial $A\in F[T]$ of $\alpha$. As $\alpha$ is an
algebraic integer, we have $A\in \cO_F[T]$. 
Let $\phi \in 1+ X^{-1}\cO_F[[X^{-1}]]$ be the power series as in
Proposition~\ref{prop:maincongruence}. We recall that
\begin{equation*}
  \phi^d = \frac{A_l}{A_k}. 
\end{equation*}

Let $\sigma\colon F\rightarrow\IC$ be a field embedding.
Then $\sigma(\phi)$ is a power series in
$\IC[[1/X]]$ with
\begin{equation}
  \label{eq:sigmaprodfinalproof}
  \sigma(\phi)^d = 
  \prod_{i=1}^{D}
  \frac{1-\sigma(f)^{(l)}(z_{\sigma,i})/X}{1-\sigma(f)^{(k)}(z_{\sigma,i})/X}
\end{equation}
where the  $z_{\sigma,1},\ldots,z_{\sigma,D}$ are the roots of $\sigma(A)\in \IC[X]$. 

Our task is to attach a compact and non-empty $\cT_\sigma\subset\IC$
to each $\sigma\in\mathrm{Hom}(F,\IC)$. In fact, we will establish
that each 
$\widehat\IC\ssm\cT_\sigma$ is simply connected and that
$\sigma(\phi)$ extends to a holomorphic map on $\widehat\IC\ssm
\cT_\sigma$.

We begin by treating the case where $\sigma$ is $\sigma_0$, the
embedding provided by the hypothesis of Theorem~\ref{thm:main}.
So $\sigma_0(f)$ is hyperbolic. 
We want to apply Corollary~\ref{cor:growtree} to $\sigma_0(f)$ and the
$2D$ points 
$\sigma_0(f^{(k)})(z_{\sigma_0,i}),\sigma_0(f^{(l)})(z_{\sigma_0,i})$
for
$i\in \{1,\ldots,D\}$. We let $c_{1}=c_1(f)>0$
denote the constant
from the corollary. 

Let us now check the  inequality in (\ref{eq:growtreeshyp}) (with $2D$
instead of $D$).
In other words, we need to verify
\begin{equation}
  \label{eq:lambdafklzsigma}
  \max_{1\le i\le D}
  \lambda_{\sigma_0(f)}\bigl(\sigma_0(f^{(m)})(z_{\sigma_0,i})\bigr)
  < \frac{c_1(f)}{2D}
  \quad\text{for both}\quad m=k,l. 
\end{equation}
If $i\in \{1,\ldots,D\}$, then
\begin{equation}
  \label{eq:lambdaflzibound}
  \lambda_{\sigma_0(f)}(\sigma_0(f^{(m)})(z_{\sigma_0,i})) = d^m
  \lambda_{\sigma_0(f)}(z_{\sigma_0,i})
  < \frac{d^m c}{D^{k}} \le
  \frac{d^m c}{D}
\end{equation}
by (\ref{eq:lambdasigmafsmallworkinghyp}). 
So (\ref{eq:lambdafklzsigma}) is a consequence of
(\ref{eq:lambdaflzibound}) as we may assume
$c\le c_1(f) /(2d^m)$. 

Corollary~\ref{cor:growtree} provides a finite topological tree
$\cT_{\sigma_0}\subset \IC$ containing all
$\sigma_0(f)^{(k)}(z_{\sigma_0,i}),\sigma_0(f)^{(l)}(z_{\sigma_0,i})$
and such that
\begin{equation}
  \label{eq:dinfsigma0bound}
   \mathsf{d}_{\infty}(\cT_{\sigma_0}) \le -\frac{c_1(f)}{2D}. 
\end{equation}

By $\phi = \Phi_d(dpB/A_k^d)$ with $\deg (dpB)<\deg(A_k^d)$ as in
Proposition~\ref{prop:maincongruence}, the power series $\sigma_0(\phi)(z)$
converges for all $z\in\IC$ of sufficiently large absolute value. The
right-hand side of (\ref{eq:sigmaprodfinalproof}) determines a
rational function on $\widehat \IC$ that is regular and non-zero
outside the points
$\sigma_0(f^{(k)})(z_{\sigma_0,i}),\sigma_0(f^{(l)})(z_{\sigma_0,i})$.
Its value at $\infty$ is $1$. In particular, the right-hand side of
(\ref{eq:sigmaprodfinalproof}) is a
holomorphic function on $\widehat\IC\ssm \cT_{\sigma_0}$. The
complement $\widehat\IC\ssm \cT_{\sigma_0}$ is simply connected as
$\cT_{\sigma_0}$ is a finite topological tree. By the Monodromy
Theorem, $\sigma_0(\phi)$ extends to a holomorphic function on all of
$\widehat \IC\ssm \cT_{\sigma_0}$.

For $\sigma=\overline{\sigma_0}$ and if $\sigma\not=\sigma_0$ we let
$\cT_{\overline{\sigma_0}}$ take the image of 
$\cT_{{\sigma_0}}$ under complex conjugation.
Thus $\overline\sigma_0(\phi)$ is holomorphic on
the simply connected complement of $\cT_{\overline{\sigma_0}}$.
By definition around (\ref{def:dinfty}), complex conjugation leaves
 the transfinite diameter of a set invariant. 
Therefore,
\begin{equation}
  \label{eq:dinfsigma0boundconj}
   \mathsf{d}_{\infty}(\cT_{\overline{\sigma_0}})
  \le -\frac{c_1(f)}{2D}. 
\end{equation}

Suppose now $\sigma\not\in\{\sigma_0,\overline{\sigma_0}\}$.
In this case, we use Lemma~\ref{lem:constructKsmalllambda}, applied to
$\sigma(f)$.
Indeed, as $f$ is postcritically finite, all critical points of
$\sigma(f)$ lie in $K_{\sigma(f)}$. 
We take $t = c/D$. 
The same argument as in (\ref{eq:lambdaflzibound})
with $m=k,l$ implies that all
$\sigma(f^{(k)})(z_{\sigma,i}),\sigma_0(f^{(l)})(z_{\sigma,i})$ 
lie in 
\begin{equation}
  \label{def:Ksigmanotsigma}
  \cT_\sigma = \left\{z\in \IC : \lambda_{\sigma(f)}(z)\le 
    \frac{c}{D} \right\}. 
\end{equation}
By Lemma~\ref{lem:constructKsmalllambda}, $\cT_\sigma\subset\IC$ is compact,
non-empty, $\widehat \IC\ssm \cT_\sigma$ is simply connected and 
\begin{equation}
  \label{eq:dinfsigmabound}
  \mathsf{d}_{\infty}(\cT_\sigma)  = \frac{c}{D}. 
\end{equation}

As with $\sigma_0$, we find that $\sigma(\phi)$, as in
(\ref{eq:sigmaprodfinalproof}), extends to a holomorphic map on
$\widehat\IC\ssm \cT_\sigma$.

This construction is compatible with complex conjugation.
Indeed,
we have $\lambda_{\overline{\sigma}(f)}(\overline z)  =
\lambda_{\sigma(f)}( z)$ for all $z\in\IC$. 
So  $\cT_{\overline\sigma}$ is the image
of $\cT_\sigma$ under complex conjugation.

We now sum over all $\sigma$ in preparation for
the P\'olya--Bertrandias Theorem.
Let $\sigma \in \mathrm{Hom}(F,\IC)$.
By (\ref{eq:dinfsigma0bound}) and (\ref{eq:dinfsigma0boundconj}) for $\sigma\in
\{\sigma_0,\overline{\sigma_0}\}$ and
(\ref{eq:dinfsigmabound}) for the remaining embeddings we obtain
\begin{equation*}
  \sum_{\sigma\in \mathrm{Hom}(F,\IC)} \mathsf{d}_\infty(\cT_\sigma) \le
  -\frac{c_1(f)}{2D} + \frac{c[F:\IQ]}{D}. 
\end{equation*}
We may assume  that $c[F:\IQ] \le c_1(f)/4$ by fixing $c>0$ small enough. Therefore,
\begin{equation}
  \label{eq:dinftsum}
  \sum_{\sigma\in \mathrm{Hom}(F,\IC)} \mathsf{d}_\infty(\cT_\sigma) \le
  -\frac{c_1(f)}{4D}<0.
\end{equation}

Recall that the coefficients of the series $\phi$ are algebraic
integers. So we may apply the P\'olya--Bertrandias Theorem,
Th\'eor\`eme~5.4.6,~\cite{Amice}, to $\phi$ while working only with
the archimedean places of $F$.
To formally apply this theorem we need to work with $\phi-1$, which
has no constant term. But this does not change the analytic properties
of the series.
We also use the language of field
embeddings instead of the language of archimedean places of $F$.
The strict inequality (\ref{eq:dinftsum})
implies that $\phi$ represents a rational function. Since $\phi^d =
A_l/A_k$ we conclude that $A_l/A_k$ is a $d$-th power in the field
$F(X)$.

We split up into two cases depending on the polynomials $A_k$ and $A_l$.
Recall that the roots of $A_m$ are of the form $f^{(m)}(\alpha')$ with
$\alpha'\in\overline F$ a root of the $F$-minimal polynomial $A$ of $\alpha$.

\medskip
\noindent\textbf{Case 1.}
The polynomials $A_l$ and $A_k$ are not coprime.

In this case we apply Lemma~\ref{lem:AkAlpreperiodic}(ii) and $k<l$.
Thus 
the point $\alpha$ is
$f$-preperiodic. Moreover, we find $\mathrm{preper}(\alpha)\le
k$ and $\mathrm{per}(\alpha)\le (l-k)D$. 
So  conclusion (ii) of this theorem  is satisfied, the bound for
$\mathrm{perper}(\alpha)$ with ample margin.

Note that in this case we did not use
the induction hypothesis. This case incorporates
the base case $\alpha\in F$, as we shall see in case 2 below.

\medskip
\noindent\textbf{Case 2.}
The polynomials $A_l$ and $A_k$ are coprime. 

As $A(\alpha)=0$, we find $A_k(f^{(k)}(\alpha))=0$. Let $\widetilde A\in F[X]$ be an
irreducible factor  of $A_k$ that vanishes at
$f^{(k)}(\alpha)$.
As $A_l/A_k$ is a $d$-th power in $F(X)$ and as $A_l,A_k$ are coprime, 
we find  $\widetilde A^d \mid A_k$. So 
\begin{equation}
  \label{eq:case2degreedrop}
  [F(f^{(k)}(\alpha)):F] = \deg \widetilde A \le \frac{\deg A_k}{d}
= \frac{\deg A}{d}=\frac{D}{d}< D. 
\end{equation}
In particular, $D\ge d\ge 2$ and so we are not in the base case of
our induction.

The inequality (\ref{eq:lambdasigmafsmallworkinghyp}) and $d=p^e$ imply
\begin{equation*}
  \max_{\sigma\colon F(\alpha)\rightarrow\IC}
  \lambda_{\sigma(f)}(\sigma(f^{(k)}(\alpha))) <
  \frac{c d^k}{D^{k}}             
  \le \frac{c}{[F(f^{(k)}(\alpha)):F]^{k}}  
\end{equation*}
where we used (\ref{eq:case2degreedrop}).

By the induction hypothesis applied to $f^{(k)}(\alpha)$ we are in
conclusion (ii) of the theorem. Therefore, $f^{(k)}(\alpha)$ is
$f$-preperiodic and so is $\alpha$. Moreover, 
\begin{equation}
  \label{eq:inductionpreperper}
  d^{\frac{\mathrm{preper}(f^{(k)}(\alpha))}{k}-1} \le 
  [F(f^{(k)}(\alpha)):F] 
  \quad\text{and}\quad
  {\mathrm{per}(f^{(k)}(\alpha))} \le
  (l-k)[F(f^{(k)}(\alpha)):F]. 
\end{equation}

Elementary  considerations imply
$\mathrm{preper}(\alpha)-k\le \mathrm{preper}(f^{(k)}(\alpha))$
and $ \mathrm{per}(\alpha)=\mathrm{per}(f^{(k)}(\alpha))$. Thus
\begin{equation*}
  d^{\frac{\mathrm{preper}(\alpha)}{k}-1} \le
  d^{\frac{\mathrm{preper}(f^{(k)}(\alpha))}{k}} \le
  d[F(f^{(k)}(\alpha)):F] \le D=[F(\alpha):F]
\end{equation*}
by (\ref{eq:inductionpreperper}) and (\ref{eq:case2degreedrop}).
Moreover,
\begin{equation*}
  {\mathrm{per}(\alpha)}
  =  {\mathrm{per}(f^{(k)}(\alpha))}\le   
  (l-k)[F(f^{(k)}(\alpha)):F]\le   (l-k)[F(\alpha):F]
\end{equation*}
by (\ref{eq:inductionpreperper}), the final inequality is trivial.
Therefore, both conclusions of part (ii) hold and our proof of
Theorem~\ref{thm:main} is complete. \qed

We now prove Theorem~\ref{thm:mainpreper}. 
The argument is largely similar to the proof of Theorem~\ref{thm:main}.

We do an induction on $[F(\alpha):F]$.

Let $\alpha\in \overline F$ be $f$-preperiodic.

We claim that $\alpha$ is an algebraic integer. Indeed, 
let $n>m\ge 0$
be integers with $f^{(n)}(\alpha) = f^{(m)}(\alpha)$. So $\alpha$ is a
root of the monic polynomial $f^{(n)}-f^{(m)}\in \cO_F[T]$. Thus
$\alpha$ is an algebraic integer, as claimed.

So the $F$-minimal polynomial $A\in F[T]$ of $\alpha$ lies in
$\cO_F[T]$. 

All assumptions are satisfied in
Proposition~\ref{prop:maincongruence} with $S$ the set of archimedean
places of $F$.

Let 
$\phi\in \cO_F[[1/X]]$
be the power series provided by this proposition. 
Thus $\phi^d =A_l/A_k$.  
Let $\sigma\colon F\rightarrow\IC$ be a field embedding. Then
\begin{equation*}
  \sigma(\phi)^d = 
  \prod_{i=1}^{\deg A} \frac{1- \sigma(f)^{(l)}(z_{\sigma,i})/X}{1- \sigma(f)^{(k)}(z_{\sigma,i})/X}
\end{equation*}
where $z_{\sigma,i}$ are the roots of $\sigma(A)$. 

To each $\sigma$ as above we now attach a finite topological tree
$\cT_\sigma$. This choice is different from the choice in the proof of
Theorem~\ref{thm:main}:  we will use Proposition~\ref{prop:cappreper}
instead of the Hubbard tree construction.
We require our hypothesis  that $f$ is not $\overline F$-linearly
conjugate to a Chebyshev polynomial up-to sign.
It is not difficult to check that  $\sigma(f)$ is not $\IC$-linearly conjugate to a Chebshev
polynomial up-to sign.
We  also need that
$f$ is postcritically finite: this entails that $K_{\sigma(f)}$ is
arcwise connected for all $\sigma\colon F\rightarrow\IC$, see
Lemma~\ref{lem:postcriticalprops}(iii).

Recall that $\alpha$ is $f$-preperiodic and hence all
conjugates $z_{\sigma,i}$ are $\sigma(f)$-preperiodic. In particular,
these complex numbers lie in the filled Julia set $K_{\sigma(f)}$.
Furthermore, 
$\sigma(f)^{(l)}(z_{\sigma,i}),
\sigma(f)^{(k)}(z_{\sigma,i})$,
for all $i\in \{1,\ldots,\deg A\}$, are also contained in the filled Julia
set $K_{\sigma(f)}$.
So by Proposition~\ref{prop:cappreper} we obtain a
finite topological tree $\cT_\sigma \subset
K_{\sigma(f)}$ that contains $f^{(l)}(z_{\sigma,i}),
f^{(k)}(z_{\sigma,i})$
for all $i\in \{1,\ldots,\deg A\}$. 
We can arrange that $\cT_{\overline \sigma}$ is
the complex conjugate of $\cT_\sigma$. 
The corollary provides
\begin{equation}
  \label{eq:trandianegative}
  \mathsf{d}_\infty(\cT_\sigma)<0
  \quad\text{for all}
  \quad \sigma\colon
  F\rightarrow \IC.
\end{equation}

Note that $\widehat \IC\ssm \cT_{\sigma}$ is simply connected. All
poles and zeros of the rational function $\sigma(\phi)^d$ are
contained in $\cT_{\sigma}$. So the Monodromy Theorem implies that
$\sigma(\phi)$ extends to a holomorphic map $\widehat \IC\ssm
\cT_{\sigma}\rightarrow\IC$.

Since the coefficients of $\phi$ are algebraic integers and
$\sum_{\sigma\in \mathrm{Hom}(F,\IC)} \mathsf{d}_\infty(\cT_{\sigma})
< 0$, the P\'olya--Bertrandias Theorem, already invoked in the Proof of
Theorem~\ref{thm:main}, implies that $\phi$ is represented by a
rational function. We again split up into two cases.

\medskip
\noindent\textbf{Case 1.} The polynomials $A_l$ and $A_k$ are not coprime.

In this case we again apply Lemma~\ref{lem:AkAlpreperiodic} and $k<l$.
The element $\alpha$ is $f$-preperiodic by hypothesis. But the lemma
provides the bounds $\mathrm{preper}(\alpha)\le k$ and
$\mathrm{per}(\alpha)\le (l-k)[F(\alpha):F]$. So
(\ref{eq:thmpreperbounds}) and (\ref{eq:thmpreperbounds_deg2}) hold
true.

Note that in this case we did not use
the induction hypothesis. This case incorporates
the base case $\alpha\in F$, as we shall see in case 2 below.

\medskip
\noindent\textbf{Case 2.}
The polynomials $A_l$ and $A_k$ are coprime. 

As $A(\alpha)=0$, we find $A_k(f^{(k)}(\alpha))=0$. Let $\widetilde A$ be an
irreducible factor in $F[X]$ of $A_k$ that vanishes at
$f^{(k)}(\alpha)$.
Recall that $A_l/A_k=\phi^d$ and $\phi$ is a rational function. 
Since $A_l,A_k$ are coprime, we find $\widetilde A^d \mid A_k$. So 
\begin{equation}
  \label{eq:case2degreedrop2}
  [F(f^{(k)}(\alpha)):F] = \deg \widetilde A \le \frac{\deg A}{d}
= \frac{\deg A}{d}=\frac{[F(\alpha):F]}{d}<[F(\alpha):F].
\end{equation}
In particular, $\alpha\not\in F$ and so we are not in the base case of
our induction.

Now  $f^{(k)}(\alpha)$ is $f$-preperiodic since $\alpha$ is
$f$-preperiodic.
By this theorem applied by
induction to $f^{(k)}(\alpha)$ we find 
\begin{equation}
  \label{eq:inductionpreperper2}
  d^{\frac{\mathrm{preper}(f^{(k)}(\alpha))}{k}-1} \le 
  [F(f^{(k)}(\alpha)):F] 
  \text{ and }
{\mathrm{per}(f^{(k)}(\alpha))} \le
(l-k)[F(f^{(k)}(\alpha)):F].
\end{equation}

Just as in the proof of Theorem~\ref{thm:main} we use the general bounds
$\mathrm{preper}(\alpha)-k\le \mathrm{preper}(f^{(k)}(\alpha))$
and $ \mathrm{per}(\alpha)=\mathrm{per}(f^{(k)}(\alpha))$. Thus
\begin{equation*}
  d^{\frac{\mathrm{preper}(\alpha)}{k}-1} \le 
  d^{\frac{\mathrm{preper}(f^{(k)}(\alpha))}{k}} \le 
  d[F(f^{(k)}(\alpha)):F] \le [F(\alpha):F] 
\end{equation*}
by (\ref{eq:inductionpreperper2}) and (\ref{eq:case2degreedrop2}).
Moreover,
\begin{equation*}
  {\mathrm{per}(\alpha)}
  =  {\mathrm{per}(f^{(k)}(\alpha))}
  \le (l-k) [F(\alpha):F]
\end{equation*}
by (\ref{eq:inductionpreperper2}). This completes the proof of
Theorem~\ref{thm:mainpreper}. \qed

\begin{proof}[Proof of Theorem~\ref{thm:mainpreperdeg2}]
  We have $f=T^2+b$ and $F=\IQ(b)$. By hypothesis, the critical
  point $0$ is periodic of period $n$. We use
  Lemma~\ref{lem:conglemma_unicritical}.
  Thus $(k,l)=(1,1+n)$ is an $(F,f)$-admissible pair.
  The  critical point of $\pm
  (T^2-2)$ is not 
  periodic. Therefore, our $T^2+b$ is not $\overline F$-linearly
  conjugate to $T^2-2$ or $-(T^2-2)$.  Theorem~\ref{thm:mainpreperdeg2}
  now follows from Theorem~\ref{thm:mainpreper}.
\end{proof}

\begin{proof}[Proof of Corollary~\ref{cor:canhgtlb}]
  If $\alpha\in \overline F$ is $f$-preperiodic then, as in the proof
  of Theorem~\ref{thm:mainpreper},  $\alpha$ is
  an algebraic integer. We are in alternative (ii) of
  Theorem~\ref{thm:main}. Thus the property in alternative (ii) of
  this  corollary follows.
  
  Now suppose  that $\alpha\in \overline F$ is $f$-wandering.
  Thus
  $\hat h_f(\alpha)>0$.
  
  Suppose that $\lambda_{f,v}(\alpha)>0$ for some finite place
  $v$ of $F(\alpha)$. Let $\ell$ be the  prime
  number with $v\mid \ell$. 
  Recall that $f\in\cO_F[T]$. Therefore, the
  local canonical height is given by
  \begin{equation}
    \label{eq:localheightgoodred}
    \lambda_{f,v}(\alpha) = \log\max\{1,|\alpha|_v\},
  \end{equation}
  see Exercise 3.30~\cite{Silverman:gtm241}. Thus we must have $|\alpha|_v>1$. Our
  normalization implies $|\alpha|_v = \ell^{k/e_v}$ for some integer $k>0$
  and where $e_v\ge 1$ is the ramification index at $v$.
  So $|\alpha|_v^{d_v}\ge \ell^{f_v}\ge 2$ with $f_v$ the residue
  degree at $v$.
  All local contributions to the canonical height
  (\ref{def:canonicalhgt})
  are non-negative.
  Therefore, 
  $\hat h_f(\alpha)\ge \frac{d_v}{[F(\alpha):\IQ]}\lambda_{f,v}(\alpha)\ge \frac{\log
    2}{[F(\alpha):\IQ]}$. The conclusion in (i) of this  corollary follows easily. 

  We have reduced the corollary to the case $\lambda_{f,v}(\alpha)=0$ for
  all finite places $v$. So only the archimedean places of $F(\alpha)$
  contribute to the canonical height
  \begin{equation}
    \label{eq:canheightonlyinfinite}
    \hat h_f(\alpha) = \sum_{v\mid\infty} \frac{d_v}{[F(\alpha):\IQ]}\lambda_{f,v}(\alpha).
  \end{equation}
  Moreover, by (\ref{eq:localheightgoodred}) we find that $|\alpha|_v\le 1$ for all finite $v$.
  Thus $\alpha$ is an algebraic integer. We will apply the bound in
  Theorem~\ref{thm:main}(i) to $\alpha$. 
  
  Let $v$ be an archimedean place of $F(\alpha)$ with
  $\lambda_{f,v}(\alpha)$ maximal. 
  So
  $\lambda_{f,v}(\alpha)\ge c(F,f,k,l)/[F(\alpha):F]^{k}$  
  and 
  \begin{equation*}
    \hat h_{f}(\alpha) \ge
    \frac{d_v }{[F(\alpha):\IQ]}\lambda_{f,v}(\alpha)
    \ge \frac{c(F,f,k,l)}{[F:\IQ]}\frac{1}{[F(\alpha):F]^{k+1}} 
  \end{equation*}
  by (\ref{eq:canheightonlyinfinite}).
  Again the conclusion of (i) in this corollary follows. 
\end{proof}

\begin{proof}[Proof of Theorem~\ref{thm:mainintro}]
  By hypothesis,
  the extension $F/\IQ$ is unramified above $p$.
  Let $\mathfrak f$ denote the least common multiple of all residue
  degrees of all prime ideals of $\cO_F$ containing $p$. 
  According to Lemma~\ref{lem:conglemma_new_cor}, see
  Remark~\ref{rmk:choicekl}, the pair $(1,1+p\mathfrak f)$ is 
  $(F,f)$-admissible. Theorem~\ref{thm:mainintro} now follows from
  Corollary~\ref{cor:canhgtlb}.
\end{proof}

\begin{proof}[Proof of Theorem~\ref{thm:main2intro}]
  The argument is similar as for the proof of
  Theorem~\ref{thm:mainintro}. However, we refer to
  Theorem~\ref{thm:main}. 
\end{proof}

\begin{proof}[Proof of Theorem~\ref{thm:intropreper}]
  We argue as in the proof of Theorem~\ref{thm:mainintro}.
  By Lemma~\ref{lem:conglemma_new_cor},
  the pair $(k,k+p\mathfrak f)$ is
  $(F,f)$-admissible for $k=1+\lceil \frac{\log\mathfrak{e}}{\log d}\rceil$.
  Now we apply
  Theorem~\ref{thm:mainpreper}.
  The period length of $\alpha$ is at most $p\mathfrak f
  [F(\alpha):F]$.
  The preperiod length $\mathrm{preper}(\alpha)$ is at most
  \begin{equation*}
    k \frac{\log(d[F(\alpha):F])}{\log
      d} = \left(1+\left\lceil \frac{\log\mathfrak{e}}{\log d}\right\rceil\right)\frac{\log(d[F(\alpha):F])}{\log
      d},
  \end{equation*}
  as desired.
\end{proof}

\begin{proof}[Proof of Theorem~\ref{thm:unicriticalwandering}]
  This theorem follows from Theorem~\ref{thm:mainintro} with
  $F=\IQ(b)$ and
  Lemma~\ref{lem:critperimplieshyp} from the introduction.
\end{proof}

\begin{proof}[Proof of Corollary~\ref{cor:unicritialpreperiodic}]
  This corollary follows from Theorem~\ref{thm:intropreper} with
  $F=\IQ(b)$ and
  Lemma~\ref{lem:critperimplieshyp} from the introduction.
  Recall that neither a Chebyshev polynomial nor its negative is
  hyperbolic. 
\end{proof}

\begin{proof}[Proof of Corollary~\ref{cor:unicritialdeg2preperiodic}]
  This corollary follows from Theorem~\ref{thm:mainpreperdeg2}.
\end{proof}


\section{Examples}
\label{sec:examples}

We exhibit some polynomials for which are method applies.
There is a well-discussed analogy between postcritically finite maps
and elliptic curves with complex multiplication. However, this analogy
also has its limits. An
elliptic curve with complex multiplication that is defined over a
number field has potentially good everywhere. This is not  the case
for postcritically finite polynomials. We refer to work of Morton and
Silverman~\cite{MortonSilverman} for a definition of good reduction.
We are mainly interested in the case of polynomial mappings. 

Let $f\in F[T]$ be a postcritically finite and  monic polynomial of degree
$d\ge 2$ with $f(0)=0$. Let $v$ be a finite place of $F$ whose residue
characteristic is greater than $d$. Then all coefficients of $f$ are
$v$-integral, which follows from Anderson's Theorem 4.1~\cite{Anderson:radius}.

If the residue characteristic is less than $d$, then we may have
bad reduction in the postcritically finite case. In fact, bad
reduction seems difficult to avoid.

Indeed, 
Ingram~\cite{Ingram:PCF} found the postcritically finite cubic
polynomial $f= T^3-\frac 32 T^2\in\IQ[T]$
with a coefficient that is not integral at $2$.
Moreover, $-1/2$ is a fixed point of $f$ with $f'(-1/2)=9/4$. So
$f$ has a 
 $2$-adically repelling fixed point. By Morton and Silverman's Corollary
5.3(a)~\cite{MortonSilverman},
$f$ does not have good reduction above $2$ in any
finite extension of $\IQ$.
The two critical points of $f$ are $0$ and $1$. The first is a
superattracting fixed point. The second maps to the fixed point $-1/2$ which is 
also repelling at the archimedean place.
Therefore,  $-1/2$ and $1$ lie on the Julia set of $f$. 
In particular, $f$ is not hyperbolic, see
Lemma~\ref{lem:hyperboliccharacterization}. 

Anderson, Manes, and Tobin~\cite{AndersonManesTobin} classified
all cubic postcritically finite polynomials with rational coefficients
up-to $\IQbar$-linear conjugacy. Their list contains $15$ polynomials,
among them is Ingram's $T^3-\frac 32 T^2$.
In his 2023 Master's Thesis, Maximilian Reber considered
the postcritically finite $ T^3 + \frac{3\sqrt{-2}}{2} T^2$ with non-integral
coefficients above $2$ and two distinct fixed critical points.

We exhibit a postcritical polynomial with multiple critical points and
with integral coefficients. For this we move to degree $4=2^2$ and
consider $f=(T^2-1)^2=T^4-2T^2+1\in \IZ[T]$. Here $f'=4(T-1)T(T+1)$.
The critical points are $0,\pm 1$ and $f(-1)=0,f(0)=1,f(1)=0$. In
other words $\pco_{\ge 1}(f) = \{0, 1\}$. The cycle $\{0,1\}$ is
superattracting and so lies in the Fatou set of $f$. By Lemma
\ref{lem:hyperboliccharacterization} we conclude that $f$ is
hyperbolic. The polynomial $f$ has integral coefficients, so
Theorem~\ref{thm:main} applies to $f$ with $F=\IQ$ and $(k,l)=(1,3)$.

Slightly more generally, consider $f = (T^2+b)^2$ with $b$ non-zero
and algebraic. The $3$ critical points are $\pm \sqrt{-b},$ and $0$.
We have $f(\pm \sqrt{-b}) = 0$. For $n\in\IN$, the expression
$f^{(n)}(0)$ is an integral, monic polynomial in $b$. For example, for
$n=1,2,3$ they are
$$b^2, b^2(b+1)^2(b^2-b+1)^2,
b^2(b^{15}+4b^{12}+6b^9+4b^6+b^3+1)^2,$$ respectively.

Let $b$ be a complex root of one of these polynomials. Then $0$ is
periodic of period length $m\le n$ and $f$ is postcritically finite.
Moreover, $f$ has integral coefficients with no $T^3$ term. Finally,
${f^{(m)}}'(0)$ is a multiple of $f'(0)=0$ by the chain rule. So $0$
is a superattracting periodic point. It follows that the postcritical
orbit of $f$ lies in $K_f\ssm J_f$ and thus $f$ is hyperbolic. Thus
Theorem~\ref{thm:mainintro} applies to $f$.


\appendix

\section{Some Facts about Polynomial Dynamical Systems}
\label{app:hyperpolys}

In this appendix we collect several fundamental results concerning the
dynamics of complex polynomials. It serves as a reference for the
body of the paper. Our sources are the books of
Milnor~\cite{Milnor} and Carleson--Gamelin~\cite{CarlesonGamelin}.

\subsection{Hyperbolic Polynomials and Postcritically Finite Polynomials}

Recall that for a complex polynomial $f$ of degree at least $2$, we
let $J_f\subset\IC$ denote its Julia set and $K_f\subset\IC$ its
filled Julia set. Both sets are fully invariant under $f$.

\begin{lemma}
  \label{lem:fatoucomponents}
  Let $f\in\IC[T]$ be of degree at least  $2$.
  \begin{enumerate}
  \item[(i)] Both $J_f$ and $K_f$ are compact and non-empty subsets
    of $\IC$. 
  \item[(ii)] The topological boundary $\partial K_f$ of $K_f$ equals
    $J_f$. 
  \item [(iii)]   Let $U$ be a connected component of $K_f\ssm J_f$.
    Then $f(U)$ is a connected component of $K_f\ssm J_f$.    
  \item[(iv)] Let $U$ be as in part (i). The forward orbit of $U$
    under $f$  is
    finite.   
  \item[(v)] The complement $\IC\ssm K_f$ is a connected component
    of $\IC\ssm J_f$. 
  \item[(vi)] Each connected component of $K_f\ssm J_f$ is simply
    connected.  
  \end{enumerate}
\end{lemma}
\begin{proof}
  The Julia set $J_f$ is non-empty by Lemma~4.8~\cite{Milnor}.
  Parts (i), (ii), (v), and (vi) follow from  Lemma~9.4~\cite{Milnor}.

  Part (iii) is Problem 4-i~\cite{Milnor}.

  Part (iv) is Sullivan's Nonwandering Theorem, see Theorem
  16.4~\cite{Milnor} and Theorem~IV.1.3~\cite{CarlesonGamelin}. The
  former reference only contains the outline of a proof.   
\end{proof}

The second statement above is Sullivan's deep Nonwandering Theorem.
However, we will only every invoke this result when $f$ is
postcritically finite. In the hyperbolic case,
see Definition~\ref{def:hyperbolic}, one can refer to
Theorem 19.1~\cite{Milnor}.

If $f$ is in addition postcritically finite we can say more. 
\begin{lemma}
  \label{lem:postcriticalprops}
  Let $f\in\IC[T]$ be of degree at least $2$. Suppose that $f$ is
  postcritically finite. Then the following properties hold true.
  \begin{enumerate}
  \item [(i)] Both $J_f$ and $K_f$ are connected.
  \item[(ii)] Both  $J_f$ and $K_f$ are locally connected.
  \item[(iii)] Both  $J_f$ and $K_f$ are arcwise connected.
  \end{enumerate}
\end{lemma}
\begin{proof}
  Recall that $\partial K_f = J_f$. Part (i) follows from
  Theorem~9.5~\cite{Milnor}. By Theorems~19.6 and 19.7~\cite{Milnor}, the
  Julia set $J_f$ is locally connected. By the characterization given by
  Theorem~18.3~\cite{Milnor}, $K_f$ is also locally connected.  So (ii) holds.
  By a general fact from topology, a connected, locally
  connected,  compact metric Hausdorff space is arcwise connected, see
  Lemmas~17.17 and 17.18~\cite{Milnor}. We conclude part (iii).
\end{proof}

\begin{defi}
  \label{def:hyperbolic}
  Let $f\in\IC[T]$ be of degree at least $2$. We call $f$ hyperbolic
  if there exist $m\in\IN$ and $\lambda >1$ such that
  $|(f^{(m)})'(z)|\ge\lambda$ for all $z\in J_f$. 
\end{defi}

By Lemma~V.2.1~\cite{CarlesonGamelin} and since $J_f$ is compact, this
definition of hyperbolic is equivalent to the one stated in the book
of Carleson--Gamelin.
Note that hyperbolic is called \textit{dynamically hyperbolic} in Milnor's
book~\cite{Milnor}.

For $m\in\IN_0$ we define
\begin{equation}
  \label{def:pcom}
  \pco_{\ge m}(f) = \{f^{(n)}(c) : n\in\IN_0, n\ge m, \text{ and
  }c\in\mathrm{Crit}(f)\}. 
\end{equation}

\begin{lemma}
  \label{lem:hyperboliccharacterization}
  Let $f\in\IC[T]$ be of degree at least $2$.  Then the following are
  equivalent.
  \begin{enumerate}
  \item [(i)] The polynomial $f$ is hyperbolic.
  \item[(ii)] The closure of $\pco_{\ge 0}(f)$ in $\IC$ is disjoint from $J_f$.
  \item[(iii)] There exists an integer $m\ge 0$ such that
    closure of $\pco_{\ge m}(f)$ in $\IC$ is disjoint from $J_f$.
  \end{enumerate}
\end{lemma}
\begin{proof}
  The equivalence ``(i)$\Leftrightarrow$(ii)'' follows from
  Theorem~V.2.2~\cite{CarlesonGamelin}. Note that
  Carleson--Gamelin call $\pco_{\ge 0}(f)$ the postcritical set
  of $f$, \textit{cf.} Section V.1~\textit{loc.cit.}
  The implication
  ``(ii)$\Rightarrow$(iii)'' is trivial. Suppose (iii) holds for $m$.
  The closure of $\pco_{\ge 0}(f)$ is the union of the closure of
  $\pco_{\ge m}(f)$ and a finite subset of $\pco_{\ge 0}(f)$.    
  If the conclusion of (ii) is false, then $f^{(n)}(c)\in J_f$ for
  some integer $n\ge 0$ and $c\in \crit{f}$. As $J_f$ is fully
  invariant under $f$ we conclude first $c\in J_f$ and second $f^{(m)}(c)\in
  J_f$. This contradicts the hypothesis of (iii). 
\end{proof}

\begin{lemma}
  \label{lem:KfJfdenseinKf}
  Let $f\in\IC[T]$ be of degree at least $2$. Suppose that $f$ is
  postcritically finite and hyperbolic. Then the following properties hold true.
  \begin{enumerate}
  \item[(i)]  
    All critical points of $f$ lie in $K_f\ssm J_f$.
    In particular,   $J_f\subsetneq K_f$.
  \item [(ii)]
    The complement $K_f\ssm J_f$ lies dense in $K_f$.  
   \end{enumerate}
\end{lemma}
\begin{proof}
  Let $c$ be a critical point of $f$. Then $c\not\in
  J_f$ as $f$ is hyperbolic. The forward orbit of $c$ 
  under $f$ is finite and so in particular 
  $c\in K_f$. Part (i) follows as $f$ has a critical point.
  
  The complement $K_f\ssm J_f\not=\emptyset$ is
  fully invariant under $f$.
  So its boundary equals $J_f$ by  \cite[Theorem III.1.7]{CarlesonGamelin}.
  Thus $K_f\ssm J_f$ lies dense in $K_f$ and the lemma follows.
\end{proof}

\subsection{Basic Metric Estimates}

Let $f\in \IC[T]$ be  of degree $d\ge 2$ with leading term
$f_0\not=0$. 
We prove several elementary metric estimates.

\begin{lemma}
  \label{lem:elementaryest}
  Let $w,z\in\IC$. There exists $x\in\IC$ with $f(x)=w$ and
  $|f(z)-w|\ge |f_0||z-x|^d$. 
\end{lemma}
\begin{proof}
  We
  factor $f-w = f_0 \prod_{i=1}^{d} (T-x_i)$ with complex $x_i$. 
  Note that $f(x_i) = w$ for all $i$.
  Now fix $i_0$ such that $|z-x_{i_0}|$ is minimal. 
  Thus
  $|f(z)-w| \ge |f_0|\prod_{i=1}^d |z-x_i|\ge |f_0||z-x_{i_0}|^d$. 
  The lemma follows with $x=x_{i_0}$.
\end{proof}

\begin{lemma}
  \label{lem:hyperbolic}  
  Suppose $\alpha>0$ and $\lambda\in (0,\alpha)$ with $|f'(z)|\ge
  \alpha$ for all $z\in J_f$. There exists $\epsilon\in (0,1]$ that depends on
  $f,\alpha,$ and $\lambda$ such that the following two properties hold.
  \begin{enumerate}
  \item[(i)]
    For all $z\in\IC$ we have 
    $\mathrm{dist}(f(z),J_f) \ge
    \min\{\epsilon,\lambda\,
    \mathrm{dist}(z,J_f)\}$.
  \item[(ii)]
    Suppose $\lambda \ge 1$.
    For all $z\in\IC$ and all $s\in\IN_0$ we have 
    $\mathrm{dist}(f^{(s)}(z),J_f) \ge
    \min\{\epsilon,\lambda^s\,
    \mathrm{dist}(z,J_f)\}$.
  \end{enumerate}
\end{lemma}
\begin{proof}
  We fix $\epsilon\in (0,1]$ during the proof.
  Let $w\in J_f$ satisfy $|f(z)-w| = \mathrm{dist}(f(z),J_f)$.
  By Lemma~\ref{lem:elementaryest} there exists $x\in \IC$ with
  $f(x)=w$ and 
  $|z-x|\le |f_0|^{-1/d}|f(z)-w|^{1/d} = |f_0|^{-1/d}\mathrm{dist}(f(z),J_f)^{1/d}$.
  As $x\in f^{-1}(J_f)=J_f$ we find $|f'(x)|\ge\alpha$ by hypothesis.
  We develop $z\mapsto f(z)-f(x)$ around $x$   to  estimate
  \begin{alignat*}1
    |f(z)-f(x)| &= \left|\sum_{k=1}^{d}
      \frac{1}{k!}\frac{\mathrm{d}^k f}{\mathrm{d}T^k}(x) (z-x)^k \right|
    \ge  
    \left( \alpha - \sum_{k=2}^{d}
    \frac{1}{k!}\left|\frac{\mathrm{d}^k f}{\mathrm{d}T^k}(x)
    (z-x)^{k-1}\right|\right)|z-x|.
  \end{alignat*}

  Recall that $x$ lies in the bounded set $J_f$. So
  $|\frac{\mathrm{d}^k f}{\mathrm{d}T^k}(x)|$ is bounded from above in
  terms of $f$ only. 
  For the proof of (i) we may assume
  $\mathrm{dist}(f(z),J_f)\le
  \epsilon$. So $|z-x|\le |f_0|^{-1/d}\epsilon^{1/d}$.
  If $\epsilon$ is small enough in terms of $f,\alpha,$ and $\lambda$, then
  $\sum_{k=2}^{d}
  \frac{1}{k!}|\frac{\mathrm{d}^k f}{\mathrm{d}T^k}(x)
  (z-x)^{k-1}| \le \alpha-\lambda$.

  Part (i) follows from
  $$\mathrm{dist}(f(z),J_f) =|f(z)-w|= |f(z)-f(x)|\ge \lambda |z-x| \ge
  \lambda \mathrm{dist}(z,J_f).$$

  Part (ii) holds trivially for $s=0$. For $s\ge 1$ we  apply part (i) to
  $f^{(s-1)}(z)$ and find
  \begin{equation*}
    \mathrm{dist}(f^{(s)}(z),J_f) \ge
    \min\{\epsilon,\lambda\,\mathrm{dist}(f^{(s-1)}(z),J_f)\}
  \end{equation*}
  for all $z\in\IC$. By
  induction on $s$ and
  since $\lambda\ge 1$ we find 
  \begin{alignat*}1
    \mathrm{dist}(f^{(s)}(z),J_f) &\ge
    \min\bigl\{\epsilon,\lambda \min\{\epsilon, \lambda^{s-1}
    \mathrm{dist}(z,J_f)\}\bigr\}    
    \ge 
    \min\{\epsilon,\lambda^s \mathrm{dist}(z,J_f)\}.\qedhere
  \end{alignat*}
\end{proof}

\begin{lemma}
  \label{lem:hyperbolic2}
  Suppose $f$ is as above and hyperbolic.  
  There exist $\epsilon>0,r\ge 1,$ and $\lambda>1$ that depend
  only on $f$ with the following property. For all $z\in\IC$ and all
  $s\in\IN_0$ we have $\mathrm{dist}(f^{(s)}(z),J_f) \ge \epsilon
  \min\{1,\lambda^s\, \mathrm{dist}(z,J_f)\}^r$.
\end{lemma}
\begin{proof}
  By hypothesis  there exists $m\in\IN$ and $\alpha>1$ such that
  $|{f^{(m)}}'(z)|\ge
  \alpha$ for all $z\in J_f$.
  We write 
  $s=mk+r$ for a non-negative integer $k$
  and some $r\in \{0,\ldots,m-1\}$.
  All iterates of $f$ have  Julia set $J_f$.
  We will apply Lemma~\ref{lem:hyperbolic}(ii) to
  $f^{(m)},\alpha,$ and $\lambda=(1+\alpha)/2$ and find
  $\mathrm{dist}(f^{(mk)}(z),J_f) \ge
  \min\{\epsilon,\lambda^{k}\mathrm{dist}(z,J_f)\}$ for some small
  enough $\epsilon>0$.
  We fix $w\in J_f$ with $|f^{(s)}(z)-w| =
  \mathrm{dist}(f^{(s)}(z),J_f)$. 
  Lemma~\ref{lem:elementaryest}  applied to the degree $d^r$
  polynomial
  $f^{(r)}$ with leading term $f_{0,r}$  yields
  $x\in (f^{(r)})^{-1}(w)\subset J_f$ with 
  $
  |f^{(mk+r)}(z)-w|
  \ge
  |f_{0,r}| |f^{(mk)}(z)-x|^{d^r}$.  We conclude
  $$
  \mathrm{dist}(f^{(s)}(z),J_f)
  \ge |f_{0,r}|\mathrm{dist}(f^{(mk)}(z),J_f)^{d^r}
  \ge |f_{0,r}|\min\{\epsilon,\lambda^{k}
  \mathrm{dist}(z,J_f)\}^{d^r}.$$ 
  The lemma follows on using $k>\frac sm -1$ and adjusting
  $\epsilon, \lambda,$ and $r$.
\end{proof}

\subsection{Behavior of Hyperbolic Polynomials near the Julia Set}

\newcommand{\varz}{w}
\newcommand{\varw}{z}
Recall that the postcritical orbit
$\pco_{\ge 1}(f)$
was defined in (\ref{def:pcoge1}).

\begin{lemma}
  \label{lem:hyperboliccovering0}
  Suppose $f\in\IC[T]$ has degree at least $2$. Let $w\in\IC$ and $\epsilon>0$
  such that $\pco_{\ge 1}(f)$ does not meet
  the open disk $B_\epsilon(w)$ of radius $\epsilon$ centered at $w$. Let $s\in\IN_0$. 
  \begin{enumerate}
  \item [(i)]
    The restriction of $f^{(s)}$ to $(f^{(s)})^{-1}(B_\epsilon(w))$ is a
    covering map onto $B_\epsilon(w)$. 
  \item[(ii)] Let $V$ be a connected component of
    $(f^{(s)})^{-1}(B_\epsilon(w))$. Then $f^{(s)}|_V\colon
    V\rightarrow B_\epsilon(w)$ is a biholomorphism.
  \end{enumerate}
\end{lemma}
\begin{proof}
  The set of  critical values of $f^{(s)}$ is contained in
  $\pco_{\ge 1}(f)$ by the chain rule.

  For brevity, we write $F = f^{(s)}$ and $U=F^{-1}(B_\epsilon(w))$. As
  $\pco_{\ge 1}(f)\cap B_\epsilon(w)=\emptyset$, we find that $F|_U\colon U
  \rightarrow B_\epsilon(w)$ is surjective with derivative that is never
  zero. So $F'|_U$ is local homeomorphism whose fibers have $(\deg
  f)^s$ elements. 
  The non-constant polynomial $F$ taken
  as an entire map is proper and so is $F|_U\colon U\rightarrow B_\epsilon(w)$.
  By Lemma~2~\cite{Ho:propermaps}, 
  $F|_U\colon U \rightarrow B_\epsilon(w)$ is a covering map. This implies
  (i).

  Let $V$ be a connected component of $U$. Then $F|_V\colon
  V\rightarrow B_{\epsilon}(w)$ is a covering map, see Corollary
  4.13~\cite{Voelklein}, and thus
  surjective per the reference's terminology. But $B_{\epsilon}(w)$
  is simply connected, so $F|_V\colon V\rightarrow B_\epsilon(w)$ is a
  homeomorphism. Part (ii) follows as a bijective holomorphic map
  between two domains in $\IC$ is a biholomorphism.
\end{proof}

\begin{lemma}
  \label{lem:hyperboliccovering}  
  Suppose $f\in\IC[T]$ has degree at least $2$ and is hyperbolic.
  For all sufficiently small $\epsilon > 0$ the following holds.
  Let $w\in J_f$ and $s\in\IN_0$.
  \begin{enumerate}
  \item [(i)]
    The restriction of $f^{(s)}$ to $(f^{(s)})^{-1}(B_\epsilon(w))$ is a
    covering map onto $B_\epsilon(w)$. 
  \item[(ii)] Let $V$ be a connected component of
    $(f^{(s)})^{-1}(B_\epsilon(w))$. Then $f^{(s)}|_V\colon
    V\rightarrow B_\epsilon(w)$ is a biholomorphism.
  \end{enumerate}
\end{lemma}
\begin{proof}  
  As $f$ is hyperbolic, the closure $\overline{\pco_{\ge 1}(f)}$ in $\IC$ is disjoint from
  the Julia set $J_f$, see Lemma~\ref{lem:hyperboliccharacterization}.  
  The Julia set $J_f$ is compact. So for all sufficiently small  $\epsilon > 0$
  we have
  $\overline{\pco_{\ge 1}(f)} \cap \{z\in \IC :
  \mathrm{dist}(z,J_f)<\epsilon\}=\emptyset$. 
  So if  $w\in J_{f}$, then
   $ \pco_{\ge 1}(f)\cap B_\epsilon(w)=\emptyset$.
  The lemma follows from Lemma~\ref{lem:hyperboliccovering0}.
\end{proof}

We come to an estimate based on the famous Koebe $1/4$-Theorem.
\begin{lemma}
  \label{lem:koebequarterapp}
  Let $U\subset\IC$ be a domain and let $g\colon U\rightarrow\IC$ be
  an injective and holomorphic map. Suppose $w,w'\in U$  such that
  $\overline{B_{|w-w'|}(w)}\subset U$. Then
  \begin{equation*}
    |g(w)-g(w')|\ge \frac{|g'(w)|}{4}|w-w'|.
  \end{equation*}
\end{lemma}
\begin{proof}
  We may assume $w\not=w'$. Let $D$ denote the open disk
  $B_{|w-w'|}(w)$ with closure $\overline D$.

  We claim that $g(w')$ lies in the topological boundary $\partial g(D)$
  of $g(D)$.

  As $g$ is continuous on $\overline D$ and since $w'\in \overline D$,
  we have $g(w') \in
  \overline{g(D)}$. If $g(w')\in g(D)$, then $g(w')=g(w'')$ for some $w''\in
  D$. But $g$ is injective on $U$ by hypothesis. So
  $w'=w''\in D$, contradicting $w'\in \partial D$.
  So $g(w')\in \overline {g(D)}\ssm g(D)$.
  As $g$ is holomorphic and non-constant, the image $g(D)$ is open in $\IC$.
  Thus $g(w')\in \overline {g(D)}\ssm g(D) = \partial g(D)$ proves our
  claim.

  Theorem I.1.4~\cite{CarlesonGamelin} applied to $g|_D$ and $w\in D$ implies
  \begin{equation*}
    \frac{|g'(w)|}{4} \mathrm{dist}(w,\partial D)
    \le \mathrm{dist}(g(w),\partial g(D)).
  \end{equation*}
  The lemma follows as $\mathrm{dist}(w,\partial D)=|w-w'|$ and
  $\mathrm{dist}(g(w),\partial g(D))\le |g(w)-g(w')|$.  
\end{proof}

We come to an elementary estimate that will prove useful later on. 
\begin{lemma}
  \label{lem:elementary}
  Let $z_1,\ldots,z_s\in\IC$ and $\delta_1,\ldots,\delta_s \in (0,1/2)$
  with $||z_i|-1|\le \delta_i$. Set
  $\delta=\sum_{i=1}^s\delta_i$, then 
  \begin{equation*}
    \bigl||z_1\cdots z_s|-1\bigr| \le e^{2\delta}-1.
  \end{equation*}  
\end{lemma}
\begin{proof}
  Note that $z_i$ are all non-zero.
  It suffices to prove the lemma when $s\ge 1$ and
  $z_1,\ldots,z_s$ are positive real numbers.
  For each $i$ we write $z_i = 1+\epsilon_i$ with
  $|\epsilon_i|\le\delta_i$. We begin by assuming the weaker
  hypothesis $|\epsilon_i|<1$ for all $i$. 
  By the arithmetic-geometric mean inequality we estimate
  \begin{equation}
  \label{eq:boundP}
  P=\prod_{i=1}^s z_i\le \left(\sum_{i=1}^s \frac{z_i}{s}\right)^s
  =\left(1+\sum_{i=1}^s \frac{\epsilon_i}{s}\right)^s\le
  \left(1+\sum_{i=1}^s \frac{|\epsilon_i|}{s}\right)^s
  \le
  e^{\sum_{i=1}^s|\epsilon_i|}  
   \le   e^{\delta}.
  \end{equation}

  The lemma follows if $P\ge 1$.
  Let us now assume $P<1$ and also $\delta_i \in (0,1/2)$ for all $i$.
  Let $\epsilon'_i =
  -\epsilon_i/(1+\epsilon_i)$, then $|\epsilon'_i|\le 2|\epsilon_i|\le
  2\delta_i<1$.
  By (\ref{eq:boundP}) with $z_i$ replaced by 
  $z_i^{-1} = 1+\epsilon'_i$ we get $P^{-1}\le
  e^{\sum_{i=1}^s|\epsilon'_i|}\le e^{2\delta}$.  As $P<1$ we find $0<1-P = P(P^{-1}-1)\le
  P^{-1}-1\le e^{2\delta}-1$ and the lemma holds too.    
\end{proof}

In the following lemma we exhibit a well-known uniformity property of
hyperbolic maps. We closely follow the argument in the proof of
Theorem V.2.3~\cite{CarlesonGamelin}.

\begin{lemma}
  \label{lem:fninvuniform_new1}
  Suppose $f\in\IC[T]$ is of degree at least $2$ and
  hyperbolic. The following holds for all $\epsilon > 0$  sufficiently small in terms of
  $f$. Let
  $s\in\IN_0$, $\widehat w\in J_{f}$, and let
  $V$ be a connected component of $(f^{(s)})^{-1}(B_{\epsilon}(\widehat{w}))$.
  Then $f^{(s)}|_V\colon V\rightarrow B_{\epsilon}(\widehat{w})$ is
  a biholomorphism.
  Let $g\colon
  B_{\epsilon}(\widehat{\varz})\rightarrow V$ denote the inverse map and
  $\widehat z = g(\widehat w)$. The following hold true.
  \begin{enumerate}
  \item[(i)] 
    We have $\frac 12 |g'(\widehat w)|\le |g'(w)|\le 2|g'(\widehat w)|$ for all $w\in
    B_{\epsilon}(\widehat w)$. 
  \item[(ii)]  
    We have ${f^{(s)}}'(\widehat{\varw})\not=0$ and
    \begin{equation*}
      \frac 12 |z-z'| \le \frac{|{f^{(s)}}(z)-{f^{(s)}}(z')|}{|{f^{(s)}}'(\widehat{\varw})|} \le 8 |z-z'|
    \end{equation*}
    for all $z,z'\in V$. 
  \end{enumerate}
\end{lemma}
\begin{proof}
  If $s=0$, then $f^{(s)}$ is the identity and so is $g$. In this
  case, the lemma is trivial. So we assume $s\ge 1$. 
  We let $\epsilon \le 1$ be small enough so that 
  Lemma~\ref{lem:hyperboliccovering} holds.
  The first claim, on $f^{(s)}|_V$, holds by the said lemma.
  
  In the proof let
  $g\colon B_{\epsilon}(\widehat w)\rightarrow V$ be the inverse of
  $f^{(s)}|_V\colon V\rightarrow B_{\epsilon}(\widehat w)$.

  For brevity, we write $U = B_{\epsilon}(\widehat w)$. Let us define
  $g_k = f^{(s-k)}\circ g\colon B_{\epsilon}(\widehat w)\rightarrow \IC$
  for all $k\in \{1,\ldots,s\}$. In particular, $g_s=g$. It is
  convenient to define $g_0$ to be the identity.

  We will shrink $\epsilon\in (0,1]$ several times. The values of
  $\lambda,c,l,c_1,c_2,\ldots$ are positive and fixed in the proof. They
  will depend only on $f$ but not on $\widehat w,s,V$ or the branch
  $g$ of $(f^{(s)})^{-1}$.

  We first prove that there exists $\lambda > 1$
  that depends only on $f$
  such that
  \begin{equation}
    \label{eq:fninvuniform}
    \mathrm{dist}(g(\varz),J_f) \le c \lambda^{-s}
    \quad\text{and}\quad
    |g'(\varz)|\le c\lambda^{-s}
    \quad\text{for all}\quad w\in U. 
  \end{equation}

  Say $\varz\in U$ and $k\in \{0,\ldots,s\}$. By Lemma~\ref{lem:hyperbolic2}
  applied to $g_k(\varz)$ and $k$ we find
  \begin{equation*}
    \mathrm{dist}(\varz,J_f) \ge c_1^{-1}
    \min\{1,\lambda^k \mathrm{dist}(g_k(\varz),J_f)\}^r;   
  \end{equation*}
  here $c_1>0, \lambda > 1,$
  and $r\ge 1$ depend only on $f$.
  We may assume $\epsilon \le 1/c_1$. 
  As $\mathrm{dist}(\varz,J_f)<\epsilon\le 1/c_1$ we get
  $\epsilon >  c_1^{-1}
  \lambda^{rk} \mathrm{dist}(g_k(\varz),J_f)^r$.
  So 
  \begin{equation}
    \label{eq:gkapproaching}
    \mathrm{dist}(g_k(\varz),J_f) \le (c_1\epsilon)^{1/r} \lambda^{-k} < (c_1\epsilon)^{1/r}
    \quad\text{for all } \varz\in U\text{ and all }k\in
    \{0,\ldots,s\}. 
  \end{equation}
  Recall that $g_s=g$.
  The first inequality in 
  (\ref{eq:fninvuniform}) follows with
  $c = (c_1\epsilon)^{1/r}$.

  As $f$ is hyperbolic, there exist $l\in\IN$ and $\mu > 1$ with
  $|F'(\varw)|\ge \mu$ for all $\varw$ sufficiently close to $J_f$ and where
  $F = f^{(l)}$. 
  This holds in particular for all  $z\in g_k(U)$ and all $k\in
  \{0,\ldots,s\}$ by
  (\ref{eq:gkapproaching}) as we may assume $\epsilon$ is small in
  terms of $f$.

  We continue to assume $k\in\{0,\ldots,s\}$ and 
  write $k=ml+t$ for some $m\in\IN_0$
  and $t\in \{0,\ldots,l-1\}$. Thus
  $F^{(m)}\circ f^{(t)}\circ g_k$  is the identity on $U$.
  Let $\varz\in U$ be arbitrary, the chain rule yields
  \begin{equation}
    \label{eq:chainrulehyperbolic}
    1 =    {f^{(t)}}'(g_k(\varz)) g'_k(\varz)\prod_{j=1}^m
      F'\bigl(F^{(m-j)}(f^{(t)}(g_k(\varz)))\bigr).
  \end{equation}
  If $j\in \{1,\ldots,m\}$, then
  $F^{(m-j)}(f^{(t)}(g_k(\varz))) = f^{(k-jl)}(g_{k}(\varz)) =
  g_{jl}(\varz)$. So 
  \begin{alignat}1
    \label{eq:chainrulehyperbolic2}
    1 &=
    \left|      {f^{(t)}}'(g_k(\varz)) g'_k(\varz)\prod_{j=1}^m F'\bigl(g_{jl}(w)\bigr)
\right|
    \ge    |      {f^{(t)}}'(g_k(\varz))| |g_k'(\varz)|\mu^m.
  \end{alignat}

  The derivative ${f^{(t)}}'$ does not vanish on $J_f$ for any $t\in
  \{0,\ldots,l\}$.
  Indeed, if
  ${f^{(t)}}'(w_0)=0$ then $f^{(t)}(w_0) \in \pco_{\ge 1}(f)$ by the
  chain rule.
  So $w_0\not\in J_f$ by Lemma~\ref{lem:hyperboliccharacterization}
  and since $f$ is hyperbolic. 
  So $|{f^{(t)}}'|$ is bounded from below
  by a positive value on some neighborhood of the compact set $J_f$. For
  $\epsilon>0$ small enough in terms of $f$, (\ref{eq:gkapproaching}) guarantees
  that $g_k(w)$ is sufficiently close to $J_f$. Thus
  \begin{equation*}
    |{f^{(t)}}'(g_k(w))|\ge c_2
    \text{ for all } t\in
    \{0,\ldots,l\},\text{ all }k\in\{0,\ldots,s\},\text{ and all }w\in U
  \end{equation*}
  where $c_2>0$ depends only on $f$ and $l$. But $l$ was fixed in
  terms of $f$, so $c_2$ depends only on $f$. 
  In particular, and this will be useful later on, $f'(g_k(w))\not=0$ for all $w\in U$.
  From (\ref{eq:chainrulehyperbolic2}) 
  we deduce
  \begin{equation*}
    |g_k'(w)|\le c_3 \mu^{-m}\quad\text{for all}\quad w\in U
    \text{ and all } k\in \{0,\ldots,s\}
  \end{equation*}
  where $c_3>0$ depends only on $f$. 
  As $k< ml+l$ we conclude $m> k/l-1$ and
  \begin{equation}
    \label{eq:gkprime}
    |g_k'(\varz)|\le c_4  \mu^{-k/l}\text{ for all } \varz\in U
    \text{ and all } k\in \{0,\ldots,s\}.
  \end{equation}
  Again $c_4>0$ and also $\mu>1$ and $l$ 
  are independent of $s,k,$ and the
  choice of branch $g$ of $(f^{(s)})^{-1}$.
  
  So the second bound in the claim (\ref{eq:fninvuniform}) holds after
  adjusting $\lambda$ in the case $k=s$.

  The claim is now proved. We now turn to conclusions (i) and (ii). 

  Let $w,w'$ lie in the disk $U$. 
  Then $|g_k(w)-g_k(w')| \le \sup_{w''\in L} |g_k'(w'')|
  |w- w'|$ where $L\subset U$ is the line segment connecting
  $w$ to $w'$; just integrate $g'_k$ along $L$. 

  So $|g_k(\varz)-g_k(\widehat{\varz})|\le 2c_4
  \mu^{-k/l} |\varz- \widehat{\varz}|$ for all $\varz\in U$
  using (\ref{eq:gkprime}). 
  By (\ref{eq:gkapproaching}) and compactness this implies
  $|f'(g_k(\varz))-f'(g_k( \widehat{\varz}))| \le c_5 \mu^{-k/l}|\varz- \widehat{\varz}|$.
  We saw above that $|f'(g_k(w))|\ge c_2$. Taking $c_6=c_5/c_2$ yields
  \begin{equation}
    \label{eq:1minusfprimef}
    \left|1-\frac{f'(g_k( \widehat{\varz}))}{f'(g_k(\varz))}\right|\le c_6
    \mu^{-k/l}|\varz- \widehat{\varz}|\quad\text{for all}\quad k\in
    \{0,\ldots,s\}\quad\text{and all}\quad w\in U.
  \end{equation}
  The chain rule implies $1 = g'(w) {f^{(s)}}'(g(w)) = g'(w)
  f'(f^{(s-1)}(g(w)))\cdots f'(g(w))$. We use the
  definition of $g_k$ and find
  $1 = g'(w)f'(g_1(w)) \cdots f'(g_s(w))$ on $U$.
  This holds in particular for $\widehat w$.
  We take the
  quotient at $\varz$ and $ \widehat{\varz}$ and find  
  \begin{equation*}
    \frac{g'(\varz)}{g'( \widehat{\varz})} = \frac{f'(g_1( \widehat{\varz}))}{f'(g_1(\varz))} \cdots \frac{f'(g_s( \widehat{\varz}))}{f'(g_s(\varz))}.
  \end{equation*}

  All factors on the right are near $1$ by   (\ref{eq:1minusfprimef})
  and for $\epsilon$ small.
  Indeed, as $|w-\widehat w |<\epsilon$ we may assume
  $c_6\mu^{-k/l}|w-\widehat w|< 1/2$.
  Thus $\sum_{k=1}^{s} c_6\mu^{-k/l}|w-\widehat w|
  \le {c_6}(1-\mu^{-1/l})^{-1}|w-\widehat w|\le c_7 |w-\widehat w|$. We
  refer to Lemma~\ref{lem:elementary} and obtain
  \begin{equation*}
    \left|\left|\frac{g'(\varz)}{g'( \widehat{\varz})}\right|-1\right|
    \le e^{2c_7|w-\widehat w|}-1\le c_8|\varz- \widehat{\varz}|.
  \end{equation*}
  We may assume $c_8\epsilon \le 1/2$.
  So 
  $$\frac 12 \le 1-c_8|\varz-\widehat{\varz}|\le
  \left|\frac{g'(\varz)}{g'(\widehat{\varz})}\right| \le
  1+c_8|\varz-\widehat{\varz}|\le 2$$
  for all $\varz\in U$.
  So $|g'(\varz)|\le 2  |g'(\widehat{\varz})|$ and
  $|g'(\widehat\varz)|\le 2 |g'({\varz})|$.
  This implies part (i).

    Let $z,z'\in V$ be as in (ii).
  We set $w=f^{(s)}(z)$ and $w'=f^{(s)}(z')$, both elements of $U$. 
  Recall that 
  $|g(w)-g(w')|\le  \sup_{w''\in L} |g'(w'')| |w-w'|$.
  So 
  $|g(w)-g(w')|\le 2|g'(\widehat{w})||w-w'|$ by part (i). 
  Now $g'(\widehat{w}){f^{(s)}}'(\widehat z)=1$ by the chain rule. So
  $|g(w)-g(w')|
  \le 2|{f^{(s)}}'(\widehat{z})|^{-1}|w-w'|$. 
  This completes the first inequality in (ii).

  For the second inequality we assume $z,z' \in
  g(B_{\epsilon/3}(\widehat w))$ and set
  $w=f^{(s)}(z)$ and $w'=f^{(s)}(z')$. So $|w-w'|<2\epsilon/3$ and the
  closed disk $\overline{B_{2\epsilon/3}(w)}$ lies entirely inside
  $B_\epsilon(\widehat w)$.
  By Lemma~\ref{lem:koebequarterapp} and part (i) we find
  $|g(w)-g(w')|\ge |g'(w)||w-w'|/4 \ge |g'(\widehat w)| |w-w'|/8$.
  As we have seen above, $g'(\widehat w) =1/{f^{(s)}}'(\widehat
  z)$.  Moreover, $g(w)=z$ and $g(w')=z'$. So
  we conclude the second bound in (ii) for all pairs in
  $g(B_{\epsilon/3}(\widehat w))$. We shrink $\epsilon$ a
  final time to get the second inequality of part (ii)
  on $g(B_\epsilon(\widehat w))$. The
  already proven
  estimates in (i) and (ii) remain true.
\end{proof}










\bibliographystyle{alpha}
\bibliography{literature}

@STRING{IMRN           = "Internat. Math. Res. Notices"}

@STRING{INVMATH        = "Invent. Math."}

@STRING{SPRINGER = "Springer"}

@STRING{PREPRINT = "Preprint."}

@article{Poonen,
author = "B.~Poonen",
title = "Mordell-{L}ang plus {B}ogomolov",
year = "1999",
pages = "413--425",
journal = INVMATH,
volume = 137,
number = 2
}

@preamble{
   "\def\cprime{$'$} "
}

@book {Voelklein,
    AUTHOR = {H.~V{\"o}lklein},
     TITLE = {Groups as {G}alois groups},
    SERIES = {Cambridge Studies in Advanced Mathematics},
    VOLUME = {53},
      NOTE = {An introduction},
 PUBLISHER = {Cambridge University Press, Cambridge},
      YEAR = {1996},
     PAGES = {xviii+248},
      ISBN = {0-521-56280-5},
   MRCLASS = {12F12}
}

@article {dimitrov:SZ,
    AUTHOR = {V.~Dimitrov},
     TITLE = {A proof of the {S}chinzel-{Z}assenhaus conjecture on polynomials},
   JOURNAL = {Preprint arXiv:1912.12545}
}

@article {Dubinin,
    AUTHOR = {Dubinin, V. N.},
     TITLE = {Change of harmonic measure in symmetrization},
   JOURNAL = {Mat. Sb. (N.S.)},
  FJOURNAL = {Matematicheski\u{\i} Sbornik. Novaya Seriya},
    VOLUME = {124(166)},
      YEAR = {1984},
    NUMBER = {2},
     PAGES = {272--279}
}

@book {Milnor,
    AUTHOR = {Milnor, J.},
     TITLE = {Dynamics in one complex variable},
    SERIES = {Annals of Mathematics Studies},
    VOLUME = {160},
   EDITION = {Third},
 PUBLISHER = {Princeton University Press, Princeton, NJ},
      YEAR = {2006},
     PAGES = {viii+304}
}

@article {Cook:DD,
    AUTHOR = {Cook, Jr., E.},
     TITLE = {Divided differences in complex function theory},
   JOURNAL = {Amer. Math. Monthly},
  FJOURNAL = {American Mathematical Monthly},
    VOLUME = {65},
      YEAR = {1958},
     PAGES = {17--24},
      ISSN = {0002-9890},
   MRCLASS = {30.00 (39.00)},
  MRNUMBER = {98169},
MRREVIEWER = {R. M. Redheffer},
       DOI = {10.2307/2310296},
       URL = {https://doi.org/10.2307/2310296},
}

@book {Ransford,
    AUTHOR = {Ransford, T.},
     TITLE = {Potential theory in the complex plane},
    SERIES = {London Mathematical Society Student Texts},
    VOLUME = {28},
 PUBLISHER = {Cambridge University Press, Cambridge},
      YEAR = {1995},
     PAGES = {x+232}
}

@book {CarlesonGamelin,
    AUTHOR = {Carleson, L. and Gamelin, T.~W.},
     TITLE = {Complex dynamics},
    SERIES = {Universitext: Tracts in Mathematics},
 PUBLISHER = {Springer-Verlag, New York},
      YEAR = {1993},
     PAGES = {x+175}
}

@article {Milnor:pasting,
    AUTHOR = {Milnor, J.},
     TITLE = {Pasting together {J}ulia sets: a worked out example of mating},
   JOURNAL = {Experiment. Math.},
  FJOURNAL = {Experimental Mathematics},
    VOLUME = {13},
      YEAR = {2004},
    NUMBER = {1},
     PAGES = {55--92}
}

@article {AlsedaFagella,
    AUTHOR = {Alsed\`a, L. and Fagella, N.},
     TITLE = {Dynamics on {H}ubbard trees},
   JOURNAL = {Fund. Math.},
  FJOURNAL = {Fundamenta Mathematicae},
    VOLUME = {164},
      YEAR = {2000},
    NUMBER = {2},
     PAGES = {115--141}
}

@article {BR:06,
    AUTHOR = {Baker, M.~H. and Rumely, R.},
     TITLE = {Equidistribution of small points, rational dynamics, and
              potential theory},
   JOURNAL = {Ann. Inst. Fourier (Grenoble)},
  FJOURNAL = {Universit\'{e} de Grenoble. Annales de l'Institut Fourier},
    VOLUME = {56},
      YEAR = {2006},
    NUMBER = {3},
     PAGES = {625--688}
}

@book {Mirsky:IntroLA,
    AUTHOR = {Mirsky, L.},
     TITLE = {An introduction to linear algebra},
 PUBLISHER = {Oxford, at the Clarendon Press},
      YEAR = {1955},
     PAGES = {xi+433}
}

@book {Silverman:gtm241,
    AUTHOR = {Silverman, J. H.},
     TITLE = {The arithmetic of dynamical systems},
    SERIES = {Graduate Texts in Mathematics},
    VOLUME = {241},
 PUBLISHER = {Springer, New York},
      YEAR = {2007},
     PAGES = {x+511}
}

@misc{hs:2021lower,
      title={Lower Bounds for the Canonical Height of a Unicritical Polynomial and Capacity}, 
      author={P. Habegger and H. Schmidt},
      year={2021},
      eprint={2111.01870},
      archivePrefix={arXiv},
      primaryClass={math.NT}
}

@article {Baker:06,
    AUTHOR = {Baker, M.},
     TITLE = {A lower bound for average values of dynamical {G}reen's
              functions},
   JOURNAL = {Math. Res. Lett.},
  FJOURNAL = {Mathematical Research Letters},
    VOLUME = {13},
      YEAR = {2006},
    NUMBER = {2-3},
     PAGES = {245--257}
}

@incollection {krattenthaler:adc,
    AUTHOR = {Krattenthaler, C.},
     TITLE = {Advanced determinant calculus},
      NOTE = {The Andrews Festschrift (Maratea, 1998)},
   JOURNAL = {S\'{e}m. Lothar. Combin.},
  FJOURNAL = {S\'{e}minaire Lotharingien de Combinatoire},
    VOLUME = {42},
      YEAR = {1999},
     PAGES = {Art. B42q, 67}
}

@article {Anderson:radius,
    AUTHOR = {Anderson, J.},
     TITLE = {Bounds on the radius of the {$p$}-adic {M}andelbrot set},
   JOURNAL = {Acta Arith.},
    VOLUME = {158},
      YEAR = {2013},
    NUMBER = {3},
     PAGES = {253--269}
}

@article {Ingram:PCF,
    AUTHOR = {Ingram, P.},
     TITLE = {A finiteness result for post-critically finite polynomials},
   JOURNAL = {Int. Math. Res. Not. IMRN},
  FJOURNAL = {International Mathematics Research Notices. IMRN},
      YEAR = {2012},
    NUMBER = {3},
     PAGES = {524--543},
      ISSN = {1073-7928}
}

@article {BIJMST:currenttrends,
    AUTHOR = {Benedetto, R. and Ingram, P. and Jones, R. and
              Manes, M. and Silverman, J.~H. and Tucker, T.
              J.},
     TITLE = {Current trends and open problems in arithmetic dynamics},
   JOURNAL = {Bull. Amer. Math. Soc. (N.S.)},
  FJOURNAL = {American Mathematical Society. Bulletin. New Series},
    VOLUME = {56},
      YEAR = {2019},
    NUMBER = {4},
     PAGES = {611--685}
}

@book {DH:etudedynI,
    AUTHOR = {Douady, A. and Hubbard, J. H.},
     TITLE = {\'{E}tude dynamique des polyn\^{o}mes complexes. {P}artie {I}},
    SERIES = {Publications Math\'{e}matiques d'Orsay [Mathematical Publications
              of Orsay]},
    VOLUME = {84},
 PUBLISHER = {Universit\'{e} de Paris-Sud, D\'{e}partement de Math\'{e}matiques, Orsay},
      YEAR = {1984},
     PAGES = {75}
}

@article {poirier:hubbardtrees,
    AUTHOR = {Poirier, A.},
     TITLE = {Hubbard trees},
   JOURNAL = {Fund. Math.},
    VOLUME = {208},
      YEAR = {2010},
    NUMBER = {3},
     PAGES = {193--248}
}

@book {Amice,
    AUTHOR = {Amice, Y.},
     TITLE = {Les nombres {$p$}-adiques},
    SERIES = {Collection SUP: ``Le Math\'{e}maticien''},
    VOLUME = {14},
      NOTE = {Pr\'{e}face de Ch. Pisot},
 PUBLISHER = {Presses Universitaires de France, Paris},
      YEAR = {1975}
}

@article {epstein:12,
    AUTHOR = {Epstein, A.},
     TITLE = {Integrality and rigidity for postcritically finite
              polynomials},
      NOTE = {With an appendix by Epstein and Bjorn Poonen},
   JOURNAL = {Bull. Lond. Math. Soc.},
  FJOURNAL = {Bulletin of the London Mathematical Society},
    VOLUME = {44},
      YEAR = {2012},
    NUMBER = {1},
     PAGES = {39--46}
}

@article {hab:dubinin,
    AUTHOR = {Habegger, P.},
     TITLE = {Separating roots of polynomials and the transfinite diameter},
   JOURNAL = {Riv. Math. Univ. Parma (N.S.)},
    VOLUME = {13},
      YEAR = {2022},
    NUMBER = {1},
     PAGES = {111--136}
}

@incollection {thurstonetal:degreed,
    AUTHOR = {Thurston, W.~P. and Baik, H. and Yan, G. and
              Hubbard, J.~H. and Lindsey, K.~A. and Tan, L. and
              Thurston, D.~P.},
     TITLE = {Degree-{$d$}-invariant laminations},
 BOOKTITLE = {What's next?---the mathematical legacy of {W}illiam {P}.
              {T}hurston},
    SERIES = {Ann. of Math. Stud.},
    VOLUME = {205},
     PAGES = {259--325},
 PUBLISHER = {Princeton Univ. Press, Princeton, NJ},
      YEAR = {2020},
       DOI = {10.2307/j.ctvthhdvv.15},
 }

@article {BourdonClark:Torsion,
    AUTHOR = {Bourdon, Abbey and Clark, Pete L.},
     TITLE = {Torsion points and {G}alois representations on {CM} elliptic
              curves},
   JOURNAL = {Pacific J. Math.},
  FJOURNAL = {Pacific Journal of Mathematics},
    VOLUME = {305},
      YEAR = {2020},
    NUMBER = {1},
     PAGES = {43--88},
      ISSN = {0030-8730,1945-5844},
   MRCLASS = {11G05 (11G15)},
  MRNUMBER = {4077686},
MRREVIEWER = {Jordi\ Gu\`ardia},
       DOI = {10.2140/pjm.2020.305.43},
       URL = {https://doi.org/10.2140/pjm.2020.305.43},
}

@article {AHPW:18,
    AUTHOR = {Anderson, J. and Hamblen, S. and Poonen, B.
              and Walton, L.},
     TITLE = {Local arboreal representations},
   JOURNAL = {Int. Math. Res. Not. IMRN},
  FJOURNAL = {International Mathematics Research Notices. IMRN},
      YEAR = {2018},
    NUMBER = {19},
     PAGES = {5974--5994},
      ISSN = {1073-7928,1687-0247}
}

@article {BFKP:Gleason,
    AUTHOR = {Buff, X. and Floyd, W. and Koch, S. and Parry,
              W.},
     TITLE = {Factoring {G}leason polynomials modulo 2},
   JOURNAL = {J. Th\'eor. Nombres Bordeaux},
  FJOURNAL = {Journal de Th\'eorie des Nombres de Bordeaux},
    VOLUME = {34},
      YEAR = {2022},
    NUMBER = {3},
     PAGES = {787--812}
}

@article {smyth:coloringproof,
    AUTHOR = {Smyth, C. J.},
     TITLE = {A coloring proof of a generalisation of {F}ermat's little
              theorem},
   JOURNAL = {Amer. Math. Monthly},
    VOLUME = {93},
      YEAR = {1986},
    NUMBER = {6},
     PAGES = {469--471}
}

@article {arnold:eulerfermat,
    AUTHOR = {Arnol{\cprime}d, V. I.},
     TITLE = {The matrix {E}uler-{F}ermat theorem},
   JOURNAL = {Izv. Ross. Akad. Nauk Ser. Mat.},
    VOLUME = {68},
      YEAR = {2004},
    NUMBER = {6},
     PAGES = {61--70}
}

@article {Lehmer:33,
    AUTHOR = {Lehmer, D. H.},
     TITLE = {Factorization of certain cyclotomic functions},
   JOURNAL = {Ann. of Math. (2)},
  FJOURNAL = {Annals of Mathematics. Second Series},
    VOLUME = {34},
      YEAR = {1933},
    NUMBER = {3},
     PAGES = {461--479}
}

@article {MortonSilverman,
    AUTHOR = {Morton, P. and Silverman, J. H.},
     TITLE = {Periodic points, multiplicities, and dynamical units},
   JOURNAL = {J. Reine Angew. Math.},
    VOLUME = {461},
      YEAR = {1995},
     PAGES = {81--122}
}

@inproceedings {AndersonManesTobin,
    AUTHOR = {Anderson, J. and Manes, M. and Tobin, B.},
     TITLE = {Cubic post-critically finite polynomials defined over {$\Bbb
              Q$}},
 BOOKTITLE = {A{NTS} {XIV}---{P}roceedings of the {F}ourteenth {A}lgorithmic
              {N}umber {T}heory {S}ymposium},
    SERIES = {Open Book Ser.},
    VOLUME = {4},
     PAGES = {23--38},
 PUBLISHER = {Math. Sci. Publ., Berkeley, CA},
      YEAR = {2020}
}

@article {Ho:propermaps,
    AUTHOR = {Ho, C.-W.},
     TITLE = {A note on proper maps},
   JOURNAL = {Proc. Amer. Math. Soc.},
    VOLUME = {51},
      YEAR = {1975},
     PAGES = {237--241}
}

@incollection {Douady:compact,
    AUTHOR = {Douady, A.},
     TITLE = {Descriptions of compact sets in {${\bf C}$}},
 BOOKTITLE = {Topological methods in modern mathematics ({S}tony {B}rook,
              {NY}, 1991)},
     PAGES = {429--465},
 PUBLISHER = {Publish or Perish, Houston, TX},
      YEAR = {1993}
}

\end{document}